\documentclass[a4paper,11pt]{article}
\usepackage{amsmath,amsthm,amssymb,enumitem,xcolor}

\usepackage[hang]{footmisc}
\usepackage[nosort,nocompress,noadjust]{cite}

\usepackage[bookmarks=false,hyperfootnotes=false,colorlinks,
    linkcolor={red!60!black},
    citecolor={blue!50!black},
    urlcolor={blue!80!black}]{hyperref}

\renewcommand{\eqref}[1]{\hyperref[#1]{(\ref{#1})}}

\newlist{enumlist}{enumerate}{2}
\setlist[enumlist,1]{labelindent=0cm,label=\arabic*.,ref=\arabic*,labelwidth=2.5ex,labelsep=0.5ex,leftmargin=3ex,align=left,topsep=0.5ex,itemsep=1ex,parsep=1ex}
\setlist[enumlist,2]{labelindent=0cm,label=\theenumlisti.\arabic*.,ref=\arabic*,labelwidth=5ex,labelsep=0.5ex,leftmargin=5.5ex,align=left,topsep=0.5ex,itemsep=1ex,parsep=1ex}

\newlist{itemlist}{itemize}{1}
\setlist[itemlist]{labelindent=0cm,label=$\bullet$,labelwidth=2.5ex,labelsep=0.5ex,leftmargin=3ex,align=left,topsep=0.5ex,itemsep=1ex,parsep=1ex}

\numberwithin{equation}{section}

{\theoremstyle{definition}\newtheorem{definition}{Definition}[section]
\newtheorem*{definition*}{Definition}

\newtheorem{remark}[definition]{Remark}

\newtheorem*{example*}{Example}
\newtheorem*{examples*}{Examples}}

\newtheorem{proposition}[definition]{Proposition}
\newtheorem{lemma}[definition]{Lemma}
\newtheorem{theorem}[definition]{Theorem}
\newtheorem{corollary}[definition]{Corollary}

\newtheorem{letterthm}{Theorem}

\newtheorem{letterprop}[letterthm]{Proposition}

{\theoremstyle{definition}}

\renewcommand{\Re}{\operatorname{Re}}

\newcommand{\C}{\mathbb{C}}

\newcommand{\al}{\alpha}
\newcommand{\be}{\beta}
\newcommand{\ot}{\otimes}

\newcommand{\Z}{\mathbb{Z}}
\newcommand{\vphi}{\varphi}

\newcommand{\id}{\mathord{\text{\rm id}}}
\newcommand{\om}{\omega}
\newcommand{\N}{\mathbb{N}}
\newcommand{\ovt}{\mathbin{\overline{\otimes}}}

\newcommand{\Om}{\Omega}

\newcommand{\si}{\sigma}
\newcommand{\R}{\mathbb{R}}
\newcommand{\F}{\mathbb{F}}
\newcommand{\cH}{\mathcal{H}}

\newcommand{\cJ}{\mathcal{J}}
\newcommand{\cF}{\mathcal{F}}
\newcommand{\T}{\mathbb{T}}
\newcommand{\actson}{\curvearrowright}

\newcommand{\cB}{\mathcal{B}}
\newcommand{\cW}{\mathcal{W}}
\newcommand{\cU}{\mathcal{U}}
\newcommand{\Ker}{\operatorname{Ker}}

\newcommand{\cV}{\mathcal{V}}
\newcommand{\cE}{\mathcal{E}}

\newcommand{\Aut}{\operatorname{Aut}}

\newcommand{\cS}{\mathcal{S}}
\newcommand{\Prob}{\operatorname{Prob}}

\newcommand{\mutil}{\widetilde{\mu}}

\newcommand{\dpr}{^{\prime\prime}}

\newcommand{\nutil}{\widetilde{\nu}}
\newcommand{\Var}{\operatorname{Var}}

\newcommand{\Kh}{\widehat{K}}

\newcommand{\etatil}{\widetilde{\eta}}
\newcommand{\muh}{\widehat{\mu}}
\newcommand{\etah}{\widehat{\eta}}
\newcommand{\thetah}{\widehat{\theta}}
\newcommand{\alh}{\widehat{\alpha}}
\newcommand{\betah}{\widehat{\beta}}
\newcommand{\gammah}{\widehat{\gamma}}

\newcommand{\Lambdah}{\widehat{\Lambda}}
\newcommand{\Q}{\mathbb{Q}}
\newcommand{\cT}{\mathcal{T}}
\newcommand{\Omegatil}{\widetilde{\Omega}}
\newcommand{\Atil}{\widetilde{A}}
\newcommand{\dTV}{d_{\operatorname{TV}}}
\newcommand{\Lh}{\widehat{L}}

\DeclareMathOperator*{\bigconv}{{\text{\Large $\ast$}}}

\begin{document}
\begin{center}
{\boldmath\LARGE\bf Bernoulli actions of type III$_0$ with\vspace{0.5ex}\\ prescribed associated flow}

\bigskip

{\sc by Tey Berendschot\footnote{\noindent KU~Leuven, Department of Mathematics, Leuven (Belgium).\\ E-mails: tey.berendschot@kuleuven.be and stefaan.vaes@kuleuven.be.}\textsuperscript{,}\footnote{\noindent T.B.\ and S.V.\ are supported by long term structural funding~-- Methusalem grant of the Flemish Government.} and Stefaan Vaes\textsuperscript{1,2,}\footnote{S.V.\ is supported by FWO research project G090420N of the Research Foundation Flanders.}}
\end{center}

\vspace{0.5ex}

\begin{abstract}\noindent
We prove that many, but not all injective factors arise as crossed products by nonsingular Bernoulli actions of the group $\mathbb{Z}$. We obtain this result by proving a completely general result on the ergodicity, type and Krieger's associated flow for Bernoulli shifts with arbitrary base spaces. We prove that the associated flow must satisfy a structural property of infinite divisibility. Conversely, we prove that all almost periodic flows, as well as many other ergodic flows, do arise as associated flow of a weakly mixing Bernoulli action of any infinite amenable group. As a byproduct, we prove that all injective factors with almost periodic flow of weights are infinite tensor products of $2 \times 2$ matrices. Finally, we construct Poisson suspension actions with prescribed associated flow for any locally compact second countable group that does not have property~(T).
\end{abstract}

\vspace{0.5ex}

\section{Introduction}\label{sec:introduction}

To a countable infinite group $G$ and a standard measure space $(X_0,\mu_0)$, called the \emph{base space}, one associates the \emph{Bernoulli action} $G\actson (X,\mu_0^{G})=\prod_{g\in G}(X_0,\mu_0)$ given by translating the coordinates by left multiplication. Bernoulli actions are at the heart of many classical, as well as recent theorems in ergodic theory and operator algebras. Especially the role of Bernoulli actions in the theory of von Neumann algebras has been very prominent, see \cite{Pop03, Pop06, CI09, Ioa10, PV21}.

By construction $\mu_0^G$ is a probability measure and it is preserved by the Bernoulli action of $G$. We rather equip $X = X_0^G$ with a product of possibly distinct probability measures $\mu_g$ on $X_0$ and thus consider the Bernoulli action
\begin{align}\label{eq:Bernoulli action}
G\actson (X,\mu)=\prod_{g\in G}(X_0,\mu_g) : (g^{-1} \cdot x)_h = x_{gh} \; .
\end{align}
We require that the action \eqref{eq:Bernoulli action} is \emph{nonsingular}, i.e.\ preserves sets of measure zero. By Kakutani's criterion for the equivalence of product measures, this is equivalent to all the measures $(\mu_g)_{g\in G}$ being equivalent and
\begin{align}\label{eq:Kakutani criterion}
\sum_{h\in G} H^2(\mu_{gh},\mu_h) <+\infty \;\;\text{for every $g\in G$,}
\end{align}
where $H(\mu,\nu)$ denotes the Hellinger distance, see \eqref{eq.hellinger}.

The key question that we address is the following: given a countable infinite group $G$, what are the possible Krieger types of nonsingular Bernoulli actions $G \actson (X,\mu)$~? This question is particularly interesting in the classical case $G = \Z$.

Recall that an essentially free ergodic nonsingular action $G \actson (X,\mu)$ is said to be of type II$_1$ if it admits an equivalent $G$-invariant probability measure, of type II$_\infty$ if it admits an equivalent $G$-invariant infinite measure and of type III otherwise. Moreover, type III actions are further classified by \emph{Krieger's associated flow} \cite{Kri76}, an ergodic nonsingular action of $\R$ that is also equal to the Connes-Takesaki \emph{flow of weights} \cite{CT77} of the crossed product von Neumann algebra $L^\infty(X) \rtimes G$. If the associated flow is trivial, the action is of type III$_1$. If it is periodic with period $|\log \lambda|$ and $\lambda \in (0,1)$, the action is of type III$_\lambda$. If the associated flow is properly ergodic, the action is of type III$_0$ and we are particularly interested to understand which associated flows may arise from nonsingular Bernoulli actions.

The first example of an ergodic Bernoulli action of type III was given by Hamachi for the group of integers \cite{Ham81}. Much later in \cite{Kos09}, Kosloff could give an example of a nonsingular Bernoulli action of $\Z$ that is of type III$_{1}$.

In the past few years, the study of nonsingular Bernoulli actions has gained momentum. The first systematic results for nonsingular Bernoulli actions of nonamenable groups $G$ were obtained in \cite{VW17}. In \cite{BKV19}, very complete results on the ergodicity and type of nonsingular Bernoulli actions with base space $X_0=\{0,1\}$ were obtained, building on important earlier work in \cite{Kos12,DL16,Kos18,Dan18}. In particular, it was shown in \cite{VW17} that the free groups $\F_n$, $n \geq 2$, admit Bernoulli actions of type III$_\lambda$ for all $\lambda \in (0,1]$. In \cite{BKV19}, it was proven that \emph{locally finite} groups admit Bernoulli actions of all possible types: II$_{1}$, II$_{\infty}$ and III$_{\lambda}$, for $\lambda\in [0,1]$. In \cite{BV20}, we proved that the same holds for all infinite amenable groups, if we allow the base space $X_0$ to be infinite. The latter is a necessary assumption, since it was proven in \cite{BKV19} that Bernoulli actions of $\Z$ with finite base space are never of type II$_\infty$. In \cite{KS20}, it was proven independently that infinite amenable groups admit Bernoulli actions of type III$_\lambda$ for all $\lambda \in (0,1]$.

Ergodic, essentially free, nonsingular actions $G \actson (X,\mu)$ of amenable groups are completely classified, both up to orbit equivalence and up to isomorphism of their crossed product von Neumann algebras, by their type and associated flow, see \cite{Kri76,Con76,CT77,CFW81,Haa85}. It is thus a very natural question to ask which ergodic flows arise as the associated flow of a nonsingular Bernoulli action, in particular of the group $\Z$. Put in an equivalent form, the question is which factors arise as crossed products $L^\infty(X) \rtimes \Z$ by nonsingular Bernoulli shifts.

We prove in this paper the surprising result that not all injective factors can arise in this way. We also prove that many injective type III$_0$ factors do arise. In particular, we prove that all \emph{infinite tensor products of factors of type I$_2$} (i.e.\ $2 \times 2$ matrices), the so called ITPFI$_2$ factors, arise as crossed products of nonsingular Bernoulli shifts.

In this paper, we call a \emph{flow} any nonsingular action of $\R$. We introduce below (see Definition \ref{def.infinitely-divisible}) the concept of an \emph{infinitely divisible} flow. By the classification of injective factors, the flow of weights of an injective factor $M$ is infinitely divisible if and only if for every integer $n \geq 1$, there exists an injective factor $N$ such that $M \cong N^{\ovt n}$. By \cite[Theorem 2.1]{GSW84}, not every injective factor, and not even every ITPFI factor, is a tensor square. So not all ergodic flows are infinitely divisible.

Our first main result says that the associated flow of a nonsingular Bernoulli shift $\Z \actson (X,\mu)$ must be infinitely divisible. We prove this result in complete generality, without making any other assumptions on the nature of the base space $X_0$ or the probability measures $\mu_n$, apart from the shift being nonsingular.

We thus also need a completely general result on the ergodicity of nonsingular Bernoulli shifts. Ruling out the trivial cases where $\mu$ admits an atom or where the action $\Z \actson (X,\mu)$ admits a fundamental domain (i.e.\ is dissipative), we prove the following result. We actually provide in Theorem \ref{thm.main-structure-variant} below a more precise description, also saying exactly what happens in the trivial cases with an atom or a fundamental domain.

\begin{letterthm}\label{thm.main-structure}
Let $\Z \actson (X,\mu) = \prod_{n \in \Z} (X_0,\mu_n)$ be a nonsingular Bernoulli shift such that $(X,\mu)$ is nonatomic and $\Z \actson (X,\mu)$ is not dissipative.

There exists an essentially unique Borel set $C_0 \subset X_0$ such that $C_0^\Z \subset X$ has positive measure and the following holds.
\begin{itemlist}
\item The nonsingular Bernoulli shift $\Z \actson C_0^\Z$ is weakly mixing and its associated flow is infinitely divisible.
\item The action $\Z \actson X \setminus C_0^\Z$ is dissipative.
\end{itemlist}
\end{letterthm}

Note that it was proven in \cite[Theorem A]{BKV19} that a Bernoulli shift of $\Z$ with base space $\{0,1\}$ is either weakly mixing, dissipative or atomic. This is compatible with Theorem \ref{thm.main-structure} because a two point base space is the only case in which a subset $C_0 \subset X_0$ is either empty, a single point or everything. Theorem~\ref{thm.main-structure} says in particular that for every conservative nonsingular Bernoulli shift $\Z \actson (X,\mu)$, the associated flow is infinitely divisible. The crossed product $L^\infty(X) \rtimes \Z$ associated with a Bernoulli shift can only be a factor if $\Z \actson (X,\mu)$ is conservative and ergodic.
As mentioned above, by \cite[Theorem 2.1]{GSW84}, it thus follows that not every injective factor, and not even every ITPFI factor, is of the form $L^\infty(X) \rtimes \Z$ where $\Z \actson (X,\mu)$ is a nonsingular Bernoulli shift.

Complementing Theorem \ref{thm.main-structure}, we determine in Theorem \ref{thm:type classification Z} in equally complete generality the type of an arbitrary nonsingular Bernoulli shift $\Z \actson (X,\mu)$.

In the converse direction, we prove that many ergodic flows do arise as associated flows of nonsingular Bernoulli actions. By \cite{CW88}, the possible flows of weights of ITPFI factors are precisely the \emph{tail boundary flows}, i.e.\ the actions of $\R$ on the Poisson boundary of a time dependent random walk on $\R$ given by a (nonconstant) sequence of transition probability measures $\mu_n$ on $\R$. If the transition probabilities $\mu_n$ can be chosen to be compound Poisson distributions, we call the tail boundary flow a \emph{Poisson flow\footnote{Sometimes, the term Poisson flow is used as a synonym for Poisson process. In this paper, a flow is always a nonsingular action of $\R$ so that no confusion should arise.}}, see Definition \ref{def.poisson-flow}.

We prove in Theorems \ref{thm.almost-periodic-Poisson} and \ref{thm.Poisson-vs-ITPFI-2} that the class of Poisson flows is large: it includes all almost periodic flows (i.e.\ flows with pure point spectrum) and it includes the flow of weights of any ITPFI$_2$ factor. By definition, Poisson flows are infinitely divisible and therefore, not every ergodic flow is a Poisson flow. In Section \ref{sec.poisson-flow}, we obtain several results on the class of Poisson flows. They can be equivalently characterized as the tail boundary flows with transition probabilities $\mu_n$ supported on two points and having uniformly bounded variance (see Proposition \ref{prop.characterize-Poisson-flow-two-point-measures}). Also, the flows of weights of ITPFI$_2$ factors can be precisely characterized as the Poisson flows of \emph{positive type} (see Definition \ref{def.poisson-flow} and Theorem \ref{thm.Poisson-vs-ITPFI-2}).

We then prove that all these Poisson flows arise as the associated flow of a weakly mixing nonsingular Bernoulli action of any amenable group. As a corollary, it thus follows that every ITPFI$_2$ factor is of the form $L^\infty(X) \rtimes \Z$ for a nonsingular, weakly mixing Bernoulli shift $\Z \actson (X,\mu)$.

\begin{letterthm}\label{thm.main-example}
Let $G$ be any countable infinite amenable group and let $\R \actson (Z,\eta)$ be any Poisson flow. There exists a family of equivalent probability measures $(\mu_g)_{g \in G}$ on a countable infinite base space $X_0$ such that the Bernoulli action $G \actson (X,\mu) = \prod_{g \in G} (X_0,\mu_g)$ is nonsingular, weakly mixing and has associated flow $\R \actson Z$.
\end{letterthm}

As a byproduct of our results on Poisson flows, it follows that every almost periodic flow is the flow of weights of an ITPFI$_2$ factor, answering a question that remained open since \cite{HO83,GS84}.

\begin{letterthm}\label{thm.almost-periodic-ITPFI-2}
Every injective factor with almost periodic flow of weights is an ITPFI$_2$ factor.
\end{letterthm}

To put Theorem \ref{thm.almost-periodic-ITPFI-2} in a proper context, recall that an ergodic almost periodic flow is precisely given by the translation action $\R \actson \Lambdah$ where $\Lambda \subset \R$ is a countable subgroup. So, for every countable subgroup $\Lambda \subset \R$, there is a unique injective factor $M_\Lambda$ with flow of weights the translation action $\R \actson \Lambdah$. Connes' $T$-invariant of $M_\Lambda$ equals $\Lambda$. In \cite[Theorem 2]{HO83}, it was proven that for every $\al \in \R\setminus \{0\}$ and every subgroup $\Lambda \subset \al \Q$, the factor $M_\Lambda$ is ITPFI$_2$. In \cite[Proposition 1.1]{GS84}, it was proven that for every countable subgroup $\Lambda \subset \R$, there exists an ITPFI$_2$ factor $M$ with $T(M) = \Lambda$, but it remained unclear if $M \cong M_\Lambda$. We now prove in Theorem \ref{thm.almost-periodic-ITPFI-2} that all $M_\Lambda$ are ITPFI$_2$ factors.

Every ITPFI$_2$ factor is infinitely divisible. It is natural to speculate that also the converse holds. One may also speculate that the infinitely divisible ergodic flows are exactly the Poisson flows. If both of these speculations are true, it follows from Theorem \ref{thm.main-structure} and Theorem \ref{thm.main-example} that the class of injective factors that can be realized as the crossed product $L^{\infty}(X)\rtimes \Z$ by a conservative nonsingular Bernoulli action $\Z\actson (X,\mu)$ equals the class of ITPFI$_2$. We refer to Remark \ref{rem.concluding-flows} for a further discussion.

Poisson flows also appear naturally in the context of nonsingular Poisson suspensions (see Section \ref{sec:Poisson suspension} for terminology). Given an infinite, $\sigma$-finite, standard measure space $(X_0,\mu_0)$ with Poisson suspension $(X,\mu)$, under the appropriate assumptions, a nonsingular action $G \actson (X_0,\mu_0)$ has a canonical nonsingular suspension $G \actson (X,\mu)$. To a certain extent, Poisson suspensions can be viewed as generalizations of Bernoulli actions, and many results were obtained recently (see \cite{Roy08,DK20,DKR20,Dan21}). In particular in \cite{Dan21}, it was proven that any locally compact second countable group $G$ that does not have property~(T) admits nonsingular Poisson suspension actions of any possible type. We prove that in type III$_0$, any Poisson flow may arise as associated flow.

\begin{letterprop}\label{prop.poisson-suspension}
Let $G$ be any locally compact second countable group that does not have property~(T). Let $\R \actson (Z,\eta)$ be any Poisson flow. Then $G$ admits a nonsingular action $G \actson (X_0,\mu_0)$ of which the Poisson suspension $G \actson (X,\mu)$ is well-defined, weakly mixing, essentially free and has associated flow $\R \actson Z$.
\end{letterprop}

By \cite[Theorem G]{DKR20}, Proposition \ref{prop.poisson-suspension} is sharp: if $G$ has property~(T), then every nonsingular Poisson suspension action of $G$ admits an equivalent $G$-invariant probability measure. We actually prove first Proposition \ref{prop.poisson-suspension} and then deduce Theorem \ref{thm.main-example} as a special case, by taking $G \actson G \times \N : g \cdot (h,n) = (gh,n)$ and choosing the measure $\mu_0$ on $G \times \N$ appropriately.

{\it Acknowledgment.} We thank Alexandre Danilenko for his useful comments on the first draft of this paper.

\section{Preliminaries}\label{sec:preliminaries}

\subsection{Nonsingular group actions}

Suppose that $(X,\mu)$ is a standard measure space. The \emph{push forward} of $\mu$ along a Borel map $\vphi\colon X\rightarrow X$, that we denote by $\vphi\mu$ or $\vphi_*\mu$, is the measure given by $(\vphi_*\mu)(\mathcal{U})=\mu(\vphi^{-1}(\mathcal{U}))$. We say that $\vphi$ is \emph{nonsingular} if the measures $\vphi_*\mu$ and $\mu$ are equivalent. If in addition $\vphi$ is invertible, we say it is a \emph{nonsingular automorphism}. In that case we will also use the notation $\mu\circ \vphi$ for the push forward measure $\vphi^{-1}\mu$. The set $\Aut(X,\mu)$ is the group of all nonsingular automorphisms of $(X,\mu)$, where we identify two elements if they agree almost everywhere. It caries a canonical topology, making it into a Polish group. Both the set $\Aut(X,\mu)$ and its topology only depend on the measure class of $\mu$.

For a nonsingular automorphism $\vphi\colon X\rightarrow X$ the Radon-Nikodym derivative
\begin{align*}
\frac{d(\mu\circ \vphi)}{d\mu}\in L^{1}(X,\mu)
\end{align*}
is uniquely determined by the equality
\begin{align*}
\int_{X}F(\vphi(x))\frac{d(\mu\circ\vphi)}{d\mu}(x)d\mu(x)=\int_{X}F(x)d\mu(x), \;\;\text{for every $F\in L^{\infty}(X,\mu)$.}
\end{align*}
A \emph{nonsingular action} of a locally compact second countable group $H$ on a standard measure space $(X,\mu)$ is a continuous homomorphism $\alpha\colon H\rightarrow \Aut(X,\mu)$, which we will also write as $H\overset{\alpha}{\actson} (X,\mu)$. Such an action induces an action $H\curvearrowright L^{\infty}(X,\mu)$ by $(h\cdot F)(x)=F(h^{-1}\cdot x)$. Conversely, if $A$ is a separably represented abelian von Neumann algebra, and $H\actson A$ acts by automorphisms, then there is a standard measure space $(Y,\nu)$, together with a nonsingular action $H\actson (Y,\nu)$, such that the action $H\actson A$ is the action induced by $H\actson (Y,\nu)$. We will frequently identify the actions $H\actson A$ and $H\actson (Y,\nu)$. Recall that a nonsingular group action $H\actson (X,\mu)$ is called \emph{ergodic} if there are no nontrivial $H$-invariant Borel sets (up to measure zero).

We call a \emph{flow} any nonsingular action of $H = \R$. We say that a flow $\R\actson (X,\mu)$ is \emph{periodic} if there exists a $t\in \R\setminus \{0\}$ such that $t\cdot F=F$ for every $F\in L^{\infty}(X,\mu)$.

A nonsingular action $G\actson (X,\mu)$ of a countable group $G$ on a standard measure space $(X,\mu)$ is called \emph{essentially free} if the set $\{x\in X:g\cdot x=x\}$ has measure zero for every $g\neq e$. If $G\actson (X,\mu)$ is essentially free, and if there exists a fundamental domain, i.e.\ a Borel set $\cW \subset X$ such that $(g \cdot \cW)_{g \in G}$ is a partition of $X$ up to measure zero, then $G \actson (X,\mu)$ is called \emph{dissipative}. On the other hand, if for every nonnegligible Borel set $\mathcal{U}\subset X$ there are infinitely many $g\in G$ such that $\mu(g\mathcal{U}\cap\mathcal{U})>0$, then we say the action $G\actson (X,\mu)$ is \emph{conservative}.

If $G \actson (X,\mu)$ is an essentially free nonsingular action, there is a unique partition (up to measure zero) $X = \cU \sqcup \cW$ of $X$ into $G$-invariant Borel sets $\cU$ and $\cW$ such that the action $G \actson (\cU,\mu)$ is conservative and $G \actson (\cW,\mu)$ is dissipative. The actions $G \actson \cU$ and $G \actson \cW$ are called the \emph{conservative, resp.\ dissipative parts} of the action $G \actson X$. Finally note that an essentially free ergodic action must be conservative, except when the action is transitive, i.e.\ when $\mu$ is atomic and supported on a single $G$-orbit, which means that the action is isomorphic with the translation action $G \actson G$.

\subsection{Nonsingular Bernoulli actions}

Suppose that $X_0$ is a standard measurable space and that $G$ is a countable infinite group. For a family of equivalent probability measures $(\mu_g)_{g\in G}$ on $X_0$, consider the product probability space
\begin{align*}
(X,\mu)=\prod_{g\in G}(X_0,\mu_g).
\end{align*}
The action $G\curvearrowright (X,\mu)$ given by $(g\cdot x)_h = x_{g^{-1}h}$, is called a \emph{Bernoulli action} of $G$, or \emph{Bernoulli shift} if $G = \Z$.

By Kakutani's criterion for the equivalence of product measures \cite{Kak48}, the Bernoulli action $G\actson (X,\mu)$ is nonsingular if and only if \eqref{eq:Kakutani criterion} holds, where $H(\mu,\nu)$ denotes the \emph{Hellinger distance} defined by
\begin{equation}\label{eq.hellinger}
H^2(\mu,\nu)=\frac{1}{2}\int_{X_0}\bigl(\sqrt{d\mu / d\zeta}-\sqrt{d\nu / d\zeta}\bigr)^2 \, d\zeta = 1 - \int_{X_0} \sqrt{\frac{d\mu}{d\zeta}} \, \sqrt{\frac{d\nu}{d\zeta}} \, d\zeta \; ,
\end{equation}
where $\zeta$ is any probability measure on $X$ with $\mu,\nu \prec \zeta$.

If $(X,\mu)=\prod_{g\in G}(X_0,\mu_g)$ is nonatomic and if the Bernoulli action $G\actson (X,\mu)$ is nonsingular, then it is essentially free, by \cite[Lemma 2.2]{BKV19}.

Recall that for any permutation $\rho$ of $G$ (finite or infinite) such that the induced transformation
\begin{align*}
\alpha_{\rho}\colon X\rightarrow X: (\rho\cdot x)_h=x_{\rho^{-1}(h)}
\end{align*}
is nonsingular we have that
\begin{align*}
\frac{d(\mu\circ \alpha_{\rho})}{d\mu}(x)=\prod_{h\in G}\frac{d\mu_{\rho(h)}}{d\mu_h}(x_h), \;\;\text{with unconditional convergence a.e. on $X$}.
\end{align*}
We then have that
\begin{align*}
1-H^2(\mu \circ \al_\rho,\mu)=\prod_{h\in G}(1-H^{2}(\mu_{\rho(h)},\mu_h))>0.
\end{align*}

\subsection{Maharam extension and associated flow}

Let $\lambda$ be the Lebesgue measure on $\R$. The \emph{Maharam extension} of a nonsingular automorphism $\vphi\in \Aut(X,\mu)$ is the nonsingular automorphism $\widetilde{\vphi}\in \Aut(X\times \mathbb{R},\mu\times \lambda)$ that is given by
\begin{align*}
\widetilde{\vphi}(x,t)=\bigl(\vphi(x), t+\log\frac{d(\mu\circ \vphi)}{d\mu}(x)\bigr) \; .
\end{align*}
Note that $\widetilde{\vphi}$ preserves the infinite measure $d\mu(x) \times \exp(-t) d\lambda(t)$. Also note that $\vphi\mapsto \widetilde{\vphi}$ is a continuous group homomorphism between the Polish groups $\Aut(X,\mu)$ and $\Aut(X\times\R,\mu\times \lambda)$.

The translation action $s \cdot (x,t) = (x, t+s)$ commutes with every $\widetilde{\vphi}$. For any nonsingular action $G \actson (X,\mu)$ Krieger's \emph{associated flow} (see \cite{Kri76}) is defined as the action of $\R$ on the ergodic decomposition of the Maharam extension $G \actson X \times \R$, which amounts to the action of $\R$ on $L^\infty(X \times \R)^G$.

Recall from the introduction how the type of a nonsingular group action $G \actson (X,\mu)$ is defined and that, for essentially free, ergodic actions of amenable groups, the type and associated flow form a complete invariant of the action, both up to orbit equivalence and up to isomorphism of the crossed product factors $L^\infty(X) \rtimes G$.

\section{\boldmath Poisson flows and infinite divisibility: proof of Theorem \ref{thm.almost-periodic-ITPFI-2}}\label{sec.poisson-flow}

\subsection{Tail boundary flows}\label{sec.tail-boundary}

Recall from \cite{CW88} the construction of the tail boundary flow as the Poisson boundary of a time-dependent Markov random walk on $\R$ with transition probabilities $(\mu_n)_{n \in \N}$. Consider
\begin{equation}\label{eq.path}
(\Omega, \mu)=\prod_{k=1}^{\infty}(\R,\mu_k) \quad\text{and}\quad (\Omega_n, \mutil_n)=\prod_{k=n+1}^{\infty}(\R,\mu_k) \; .
\end{equation}
Choose a probability measure $\mu_0$ on $\R$ that is equivalent with the Lebesgue measure and define the nonsingular maps
\begin{align*}
\pi_n\colon (\R \times \Omega, \mu_0 \times \mu) \to (\R \times \Omega_n,\mu_0 \times \mutil_n) : \pi_n(t,\omega)= (t+\omega_1+\cdots +\omega_n,\omega_{n+1},\ldots) \; .
\end{align*}
Define the von Neumann algebras $B = L^\infty(\R \times \Omega)$ and $A_n = (\pi_n)_*(L^\infty(\R \times \Omega_n))$. Then the tail boundary is defined as $A=\bigcap_{n\geq 0} A_n$. The translation action of $\R$ in the first variable defines an ergodic action of $\R$ on $A$, which is called the tail boundary flow.

We refer to \cite[Section 2.3]{BV20} for several basic results on tail boundary flows.

Tail boundary flows play a key role in this paper. When working with elements $x$ of a product space as in \eqref{eq.path}, we always denote by $x_k$ or $x_n$ the natural coordinates of $x$.

Note that the tail boundary flow does not change if we permute the probability measures $\mu_n$. More precisely, if $\si : \N \to \N$ is any bijection, then the tail boundary flows of $(\mu_n)_{n \in \N}$ and $(\mu_{\si(n)})_{n \in \N}$ are canonically isomorphic in the following way. Denoting by $(\Omegatil,\mutil)$ the path space for the family $(\mu_{\sigma(n)})_{n \in \N}$, the map
$$\theta : (\R \times \Omega , \mu_0 \times \mu)  \to (\R \times \Omegatil,\mu_0 \times \mutil) : \theta(t,\omega) = (t,\omega_{\si(1)},\omega_{\si(2)},\ldots)$$
is a measure preserving bijection. We consider the von Neumann subalgebras $A_n \subset L^\infty(\R \times \Omega)$ and $\Atil_n \subset L^\infty(\R \times \Omegatil)$, with intersections $A = \bigcap_n A_n$ and $\Atil = \bigcap_n \Atil_n$ realizing the respective tail boundaries. For every $n \in \N$, there exists an $m \in \N$ such that $\si(\{1,\ldots,n\}) \subset \{1,\ldots,m\}$ and $\si^{-1}(\{1,\ldots,n\}) \subset \{1,\ldots,m\}$.  Then, $\theta_*(A_m) \subset \Atil_n$ and $\Atil_m \subset \theta_*(A_n)$. Therefore, $\theta_*(A) = \Atil$, implementing the isomorphism.

Note that tail boundary $G$-actions can be associated in a similar way to any family $(\mu_n)_{n \in \N}$ of Borel probability measures on a locally compact second countable abelian group $G$.

We prove in this section two results on tail boundary flows that are of independent interest. First recall for future reference the following well known and easy result. For completeness, we include the short proof. The Hellinger distance $H$ was defined in \eqref{eq.hellinger}. We also make use of the \emph{total variation distance} between probability measures $\mu,\nu$ on a standard Borel space $X$:
$$\dTV(\mu,\nu) = \sup\bigl\{ |\mu(A) - \nu(A)| \bigm| A \subset X \;\text{Borel}\;\bigr\} = \frac{1}{2} \int_X |d\mu / d\zeta - d\nu /d\zeta| \, d\zeta \; ,$$
whenever $\zeta$ is a probability measure on $X$ with $\mu,\nu \prec \zeta$. Note that
$$H^2(\mu,\nu) \leq \dTV(\mu,\nu) \leq \sqrt{2} \, H(\mu,\nu)$$
for all probability measures $\mu$ and $\nu$ on $X$.

\begin{lemma}[{Cf.\ \cite[Lemma 2.5]{CW88}}]\label{lem.close-in-hellinger}
If $(\mu_n)_{n \in \N}$ and $(\nu_n)_{n \in \N}$ are sequences of probability measures on $\R$ such that
\begin{equation}\label{eq.hell-sum-finite}
\sum_{n=1}^\infty H^2(\mu_n,\nu_n) < +\infty \; ,
\end{equation}
then the tail boundary flows of $(\mu_n)_{n \in \N}$ and $(\nu_n)_{n \in \N}$ are isomorphic. This conclusion holds in particular if $\sum_{n=1}^\infty \dTV(\mu_n,\nu_n) < +\infty$.
\end{lemma}
\begin{proof}
Partition $\R$ into Borel sets $\R = U_n \sqcup V_n \sqcup W_n$ such that $\mu_n|_{U_n} \sim \nu_n|_{U_n}$, $\nu_n(V_n) = 0$ and $\mu_n(W_n) = 0$. From \eqref{eq.hell-sum-finite}, it follows in particular that
\begin{equation}\label{eq.ae}
\sum_{n=1}^\infty \mu_n(V_n) < + \infty \quad\text{and}\quad \sum_{n =1}^\infty \nu_n(W_n) < + \infty \; .
\end{equation}
Define $(\Om,\mu)$ and $(\Om_n,\mutil_n)$ as in \eqref{eq.path}, so that the tail boundary $A$ of $(\mu_n)_{n \in \N}$ is realized as the intersection of $A_n = (\pi_n)_*(L^\infty(\R \times \Om_n))$ inside $B = L^\infty(\R \times \Om)$. Similarly define $(\Om',\nu)$, $(\Om'_n,\nutil_n)$ and $\pi'_n$ so that the tail boundary $C$ of $(\nu_n)_{n \in \N}$ is realized as the intersection of $C_n = (\pi'_n)_*(L^\infty(\R \times \Om'_n))$ inside $D = L^\infty(\R \times \Om')$.

Define $X_n \subset \R \times \Om_n$ by $(t,x) \in X_n$ iff $x_m \in U_m$ for all $m \geq n+1$. Similarly define $X'_n \subset \R \times \Om'_n$. Denote $p_n = (\pi_n)_*(1_{X_n})$ and $p'_n = (\pi'_n)_*(1_{X'_n})$.
The Kakutani criterion for the equivalence of product measures implies that the identity map $X_n \to X'_n$ is a nonsingular isomorphism, inducing a $*$-isomorphism $\theta_n : A_n p_n \to C_n p'_n$. By \eqref{eq.ae}, the projections $p_n \in A$ and $p'_n \in D$ are increasing to $1$. The $*$-isomorphisms $\theta_n$ are compatible, so that there is a unique $*$-isomorphism $\theta : A \to C$ satisfying $\theta(a)p'_n = \theta_n(a p_n)$ for all $a \in A$ and $n \in \N$. By construction, $\theta$ conjugates the tail boundary flows of $(\mu_n)_{n \in \N}$ and $(\nu_n)_{n \in \N}$.
\end{proof}

We start by proving that such an identification of tail boundary flows also holds under a different approximation assumption, replacing the Hellinger distance by the Wasserstein $2$-metric. We can do even slightly better by taking the Wasserstein $2$-metric w.r.t.\ the metric on $\R$ given by $d(x,y) = |T_\kappa(x-y)|$, where for $\kappa > 0$, we denote by $T_\kappa$ the cutoff function
\begin{equation}\label{eq.cutoff}
T_\kappa : \R \to [-\kappa,\kappa] : T_\kappa(x)= \begin{cases} -\kappa &\;\;\text{if $x \leq -\kappa$,}\\ x &\;\;\text{if $-\kappa \leq x \leq \kappa$,}\\ \kappa &\;\;\text{if $x \geq \kappa$.}\end{cases}
\end{equation}

Recall that a \emph{coupling} between probability measures $\mu$, $\nu$ on $\R$ is a probability measure $\eta$ on $\R^2$ such that, writing $\pi_1 : \R^2 \to \R : \pi_1(x,y) = x$ and $\pi_2 : \R^2 \to \R : \pi_2(x,y) = y$, we have $(\pi_1)_*(\eta) = \mu$ and $(\pi_2)_*(\eta) = \nu$. The set of all couplings between $\mu$, $\nu$ is denoted as $\Gamma(\mu,\nu)$. For every $\kappa > 0$, we then denote by
$$W_{2,\kappa}(\mu,\nu) = \inf_{\eta \in \Gamma(\mu,\nu)} \Bigl(\int_{\R^2} T_\kappa(x-y)^2 \, d\eta(x,y)\Bigr)^{1/2}$$
the Wasserstein $2$-metric between $\mu$ and $\nu$, w.r.t.\ $d(x,y) = |T_\kappa(x-y)|$ on $\R$. Note that the metrics $W_{2,\kappa}$ are equivalent for different values of $\kappa > 0$ and are dominated by the usual Wasserstein $2$-metric $W_2$ w.r.t.\ the metric $d(x,y) = |x-y|$.

\begin{proposition}\label{prop.iso-wasserstein}
If $(\mu_n)_{n \in \N}$ and $(\nu_n)_{n \in \N}$ are sequences of probability measures on $\R$ such that for some $\kappa > 0$,
$$\sum_{n=1}^\infty W_{2,\kappa}(\mu_n,\nu_n)^2 < + \infty \; ,$$
then the tail boundary flows of $(\mu_n)_{n \in \N}$ and $(\nu_n)_{n \in \N}$ are isomorphic.
\end{proposition}

\begin{proof}
Fix $\kappa > 0$ and choose couplings $\zeta_n \in \Gamma(\mu_n,\nu_n)$ such that
$$\int_{\R^2} T_\kappa(x-y)^2 \, d\zeta_n(x,y) \leq 2^{-n} + W_{2,\kappa}(\mu_n,\nu_n)^2 \; .$$
So we get that
$$\sum_{n=1}^\infty \int_{\R^2} T_\kappa(x-y)^2 \, d\zeta_n(x,y) < +\infty \; .$$
Define $\psi : \R^2 \to \R^2 : \psi(x,y) = (x-y,y)$ and put $\eta_n = \psi_*(\zeta_n)$. Fix a probability measure $\lambda$ on $\R$ that is equivalent with the Lebesgue measure. Consider
$$(X,\mu) = \prod_{n=1}^\infty (\R,\mu_n) \;\; , \quad (Y,\nu) = \prod_{n=1}^\infty (\R,\nu_n) \quad\text{and}\quad (\Om,\eta) = \prod_{n=1}^\infty (\R^2,\eta_n) \; .$$
Whenever $x,y \in \R^\N$, we denote by $(x,y) \in \Om$ the element $(x_1,y_1,x_2,y_2,\ldots)$. As in the definition of the tail boundary flow, define
\begin{align*}
& (\Om_m,\etatil_m) = \prod_{n=m+1}^\infty (\R^2,\eta_n) \quad\text{and}\\
& S_m : \R \times \Om \to \R \times \Om_m : S_m(t,x,y) = \bigl(t + \sum_{i=1}^m (x_i+y_i), x_{m+1},y_{m+1},x_{m+2},y_{m+2},\ldots\bigr) \; .
\end{align*}
Note that the maps $S_m$ are nonsingular factor maps. Denote $B = L^\infty(\R \times \Om,\lambda \times \eta)$ and define $A_m \subset B$ by $A_m = (S_m)_*(L^\infty(\R \times \Om_m))$. Define $A = \bigcap_{m=1}^\infty A_m$. We let $\R$ act by translation in the first variable and obtain in this way an ergodic action $\R \actson A$.

We identify the ergodic action $\R \actson A$ with both the tail boundary flow of $(\mu_n)_{n \in \N}$ and the tail boundary flow of $(\nu_n)_{n \in \N}$.

The first identification can be easily proved as follows and basically holds by definition. Writing
\begin{align*}
& (X_m,\mutil_m) = \prod_{n=m+1}^\infty (\R,\mu_n) \quad\text{and}\\
& s_m : \R \times X \to \R \times X_m : s_m(t,x) = \bigl(t + \sum_{i=1}^m x_i, x_{m+1},x_{m+2},\ldots\bigr) \; ,
\end{align*}
the tail boundary of $(\mu_n)$ is defined by the intersection $C$ of $C_m = (s_m)_*(L^\infty(\R \times X_m))$ inside $D = L^\infty(\R \times X)$. Write $S : \R^2 \to \R : S(x,y) = x+y$. By construction, $S_*(\eta_n) = \mu_n$ and we get the measure preserving factor map
$$P : \R \times \Om \to \R \times X : P(t,x,y) = (t, x_1 + y_1, x_2 + y_2 , \ldots ) \; .$$
By independence, we have that $A \subset P_*(D)$ and $A_m \cap P_*(D) = P_*(C_m)$. Therefore, $A = P_*(C)$.

The second identification requires more work. We disintegrate the probability measures $\eta_n$ w.r.t.\ the second variable. We thus find probability measures $\eta_{n,y}$ on $\R$ such that
$$\int_{\R^2} F \, d\eta_n = \int_\R d\nu_n(y) \, \int_\R d\eta_{n,y}(x) \, F(x,y)$$
for all bounded Borel functions $F : \R^2 \to \R$. Writing, for $y \in Y$,
$$(\Om_y,\eta_y) = \prod_{n=1}^\infty (\R,\eta_{n,y_n}) \; ,$$
we can view $(\Om_y,\eta_y)_{y \in Y}$ as the disintegration of $(\Om,\eta)$ w.r.t.\ the measure preserving factor map $R : \Om \to Y : R(x,y) = y$.

Define the Borel functions
\begin{align*}
& \vphi_n : \R \to [-\kappa,\kappa] : \vphi_n(z) = \int_\R T_\kappa(x) \, d\eta_{n,z}(x) \quad\text{and}\\
& F : Y \to [0,+\infty] : F(y) = \sum_{n=1}^\infty \int_\R T_\kappa(x)^2 \, d\eta_{n,y_n}(x) \; .
\end{align*}
Since
$$\int_Y F(y) \, d\nu(y) = \sum_{n=1}^\infty \int_{\R^2} T_\kappa(x)^2 \, d\eta_n(x,y) < +\infty \; ,$$
we get that $F(y) < +\infty$ for $\nu$-a.e.\ $y \in Y$. Also note that
$$\vphi_n(y)^2 \leq \int_\R T_\kappa(x)^2 \, d\eta_{n,y}(x) \quad\text{so that}\quad \sum_{n=1}^\infty \int_\R \vphi_n(y)^2 \, d\nu_n(y) < +\infty \; .$$
Write $s_n = \int_\R \vphi_n(y) \, d\nu_n(y)$. Define the Borel sets $\cU \subset Y$ and $\cV_y \subset \Om_y$ by
\begin{align*}
y \in \cU &\quad\text{iff}\quad \vphi(y) = \lim_{n \to +\infty} \sum_{k=1}^n (\vphi_k(y_k) - s_k) \;\;\text{exists,}\\
x \in \cV_y &\quad\text{iff}\quad \theta_y(x) = \lim_{n \to +\infty} \sum_{k=1}^n (x_k - \vphi_k(y_k)) \;\;\text{exists.}
\end{align*}
Also define the Borel set
$$\cV = \{(x,y) \in \Om \mid F(y) < +\infty \;\; , \;\; y \in \cU \;\;\text{and}\;\; x \in \cV_y \} \; .$$
We already proved that $F(y) <+\infty$ for $\nu$-a.e.\ $y \in Y$. By van Kampen's version of Kolmogorov's three series theorem (see e.g.\ \cite[Theorem 3 in Section 4.2]{Shi19}), we have that $\nu(Y \setminus \cU) = 0$ and that for all $y \in Y$ with $F(y) < +\infty$, also $\eta_y(\Om_y \setminus \cV_y) = 0$. Since $F(y) < +\infty$ for $\nu$-a.e.\ $y \in Y$, it follows that $\eta(\Om \setminus \cV) = 0$. Write
\begin{equation}\label{eq.good-expression}
\pi : \cV \to \R : \pi(x,y) = \theta_y(x) + \vphi(y) = \lim_{n \to +\infty} \sum_{k=1}^n (x_k - s_k)
\end{equation}
and define the nonsingular factor map $Q : \R \times \Om \to \R \times Y : Q(t,x,y) = (t+\pi(x,y),y)$. We view the tail boundary of $(\nu_n)_{n \in \N}$ as a von Neumann subalgebra $N \subset L^\infty(\R \times Y)$. We prove that $A = Q_*(N)$.

More precisely, we write
\begin{align*}
& (Y_m,\nutil_m) = \prod_{n=m+1}^\infty (\R,\nu_n) \quad\text{and}\\
& r_m : \R \times Y \to \R \times Y_m : r_m(t,y) = \bigl(t + \sum_{i=1}^m y_i, y_{m+1},y_{m+2},\ldots) \; .
\end{align*}
Defining $N_m = (r_m)_*(L^\infty(\R \times Y_m))$, we have $N = \bigcap_{m=1}^\infty N_m$.

It immediately follows from \eqref{eq.good-expression} that $Q_*(N_m) \subset A_m$, so that $Q_*(N) \subset A$. To prove the converse, fix $F \in A$. It follows that for $\nu$-a.e.\ $y \in Y$, we may view the function $F(\cdot,y)$ as an element of the tail boundary for the measures $(\eta_{n,y})_{n \in \N}$. By e.g.\ \cite[Proposition 2.1]{BV20}, this tail boundary is given by the translation action $\R \actson \R$ and we find a unique $K_y \in L^\infty(\R)$ such that $F(t,x,y) = K_y(t+\theta_y(x))$ for a.e.\ $(t,x) \in \R \times \Om_y$. Defining $G_y(t) = K_y(t-\vphi(y))$, we get that $F(t,x,y) = G_y(t+\pi(x,y))$ for a.e.\ $(t,x) \in \R \times \Om_y$. Writing $G(t,y) = G_y(t)$, we have found $G \in L^\infty(\R \times Y)$ such that $Q_*(G) = F$.

It also follows from \eqref{eq.good-expression} that $Q_*(L^\infty(\R \times Y)) \cap A_m = Q_*(N_m)$. Therefore, $G \in N_m$ for all $m \in \N$, so that $G \in N$. This concludes the proof of the proposition.
\end{proof}

In our applications of Proposition \ref{prop.iso-wasserstein} in this paper, we will only need the following elementary estimate for the Wasserstein $2$-distance. Assume that $(\be_n)_{n \in \N}$ are probability measures on $\R$ and assume that for every $n \in \N$, we have $t_n \in \R$ and $p_n \in [0,1]$ with $\sum_{n=1}^\infty p_n = 1$. Define the probability measures $\be = \sum_{n=1}^\infty p_n \, \be_n$ and $\nu = \sum_{n=1}^\infty p_n \, \delta_{t_n}$. Then,
\begin{equation}\label{eq.wasserstein-estimate}
W_2(\be,\nu)^2 \leq \sum_{n=1}^\infty p_n \int_\R (x-t_n)^2 \, d\be_n(x) \; ,
\end{equation}
which follows immediately by using the coupling $\eta = \sum_{n=1}^\infty p_n \, (\be_n \times \delta_{t_n})$.

Secondly, we prove a generalization of Orey's fundamental result in \cite[Theorem 3.1]{Ore66}. He proved that if $(\mu_n)_{n \in \N}$ is a sequence of probability measures on $\R$ with uniformly bounded support, i.e.\ for which there exists a $C > 0$ such that $\mu_n([-C,C]) = 1$ for all $n \in \N$, the tail boundary flow is never properly ergodic: if $\sum_{n=1}^\infty \Var \mu_n = +\infty$, the tail boundary flow is periodic, and if $\sum_{n=1}^\infty \Var \mu_n < +\infty$, the tail boundary flow is given by the translation action $\R \actson \R$.

The following result says that periodicity of the tail boundary flow already follows if we can find finite positive measures $\be_n \leq \mu_n$ such that the sequence $(\be_n)_{n \in \N}$ has uniformly bounded width and such that the sum of the properly normalized variances of $\be_n$ is infinite.

More precisely, we prove the following and in particular, provide a functional analytic proof to \cite[Theorem 3.1]{Ore66}.

\begin{proposition} \label{prop.periodicity-criterion}
Let $(\mu_n)_{n \in \R}$ be a sequence of probability measures on $\R$. Assume that $\beta_n$ are positive finite measures on $\R$ satisfying $\be_n \leq \mu_n$. Assume that there exists a $C > 0$ such that $|x-y| \leq C$ for $\beta_n$-a.e.\ $x,y \in \R$. If
\begin{equation}\label{eq.sum-variance}
\sum_{n=1}^\infty \be_n(\R) \, \Var(\be_n(\R)^{-1} \be_n) = +\infty \; ,
\end{equation}
then the tail boundary flow of $(\mu_n)_{n \in \R}$ is periodic.
\end{proposition}

Note that, by convention, if certain $\be_n$ are zero, we interpret the corresponding term in \eqref{eq.sum-variance} as zero. The assumptions of Proposition \ref{prop.periodicity-criterion} say in particular that, for each fixed $n$, the measure $\beta_n$ has a bounded support, so that its mean value and variance are well defined and finite.

\begin{proof}
Note that it suffices to prove the proposition assuming that all $\be_n \neq 0$. Indeed, it then follows that for $I = \{n \in \N \mid \beta_n \neq 0\}$, the tail boundary flow of $(\mu_n)_{n \in I}$ is periodic, from which it follows that, a fortiori, the tail boundary flow of $(\mu_n)_{n \in \N}$ is periodic.

Assume that the sum in \eqref{eq.sum-variance} is infinite. We prove that the tail boundary flow of $(\mu_n)_{n \in \R}$ is periodic.

Define the finite positive measures $\al_n$ such that $\mu_n = \al_n + \be_n$. Define $X_0 = \R \sqcup \R^2$ and equip $X_0$ with the probability measures $\zeta_n = \al_n \sqcup \be_n(\R)^{-1}(\beta_n \times \beta_n)$. Define the map $T : X_0 \to \R$ by $T(x) = 0$ if $x \in \R$ and $T(x,y) = x-y$ if $(x,y) \in \R^2$. Define $(X,\zeta) = \prod_{n \in \N} (X_0,\zeta_n)$ and the independent random variables $T_n : X \to \R : T_n(z) = T(z_n)$. Note that $|T_n| \leq C$ for all $n$, $E(T_n) = 0$ and
$$E(T_n^2) = 2 \be_n(\R) \, \Var(\be_n(\R)^{-1} \be_n) \; .$$
Since $\sum_{n=1}^\infty E(T_n^2) =+\infty$ and $E(T_n^2) \leq C^2$ for every $n \in \N$, we can choose $0 \leq n_1 < n_2 < \cdots$ such that
$$(40 C)^2 \leq \sum_{n=n_{k}+1}^{n_{k+1}} E(T_n^2) \leq (50 C)^2$$
for all $k \in \N$. Since $|T_n| \leq C$ for all $n$, we have that $E(|T_n|^3) \leq C \, E(T_n^2)$ for all $n \in \N$. Define $\si_k \in [40 C,50 C]$ by
$$\si_k^2 = \sum_{n=n_{k}+1}^{n_{k+1}} E(T_n^2) \; .$$
Write
$$S_k = \sum_{n=n_{k}+1}^{n_{k+1}} T_n \quad\text{and}\quad F_k(t) = \zeta\bigl(\{x \in X \mid S_k(x) \leq \si_k \, t \, \}\bigr) \; .$$
Denote by $G$ the distribution function of a standard Gaussian random variable. By the Berry-Esseen theorem (see e.g.\ \cite[Theorem 2.2.17]{Stro11}), we get that
$$|F_k(t) - G(t)| \leq \frac{10 C}{\si_k} \leq \frac{1}{4}$$
for all $k \in \N$ and all $t \in \R$. Take $\delta < 0$ such that $G(\delta) = 1/3$. Then, $F_k(\delta) \geq 1/12$. Since $\si_k \geq 40 C$ and $\delta < 0$, we conclude that
\begin{equation}\label{eq.not-to-zero}
\zeta\bigl(\{x \in X \mid S_k(x) \leq 40 C \delta \}\bigr) \geq \frac{1}{12} \quad\text{for all $k \in \N$.}
\end{equation}

Define $(\Om,\mu) = \prod_{n \in \N} (\R,\mu_n)$. Denote by $\mu_0$ the probability measure on $\R$ given by $d\mu_0(t) = (\pi(1+t^2))^{-1} \, dt$. Consider the Hilbert spaces
$$H = L^2(\R \times \Om, \mu_0 \times \mu) \quad\text{and}\quad K = L^2(\R \times X,\mu_0 \times \zeta)$$
and define the subspaces $H_N \subset H$ by $H_N = L^2((\R,\mu) \times \prod_{n=1}^N (\R,\mu_n))$.

Define the maps
\begin{align*}
&\theta_1 : X_0 \to \R : \theta_1(x) = x \;\;\text{if $x \in \R$,}\;\; \theta_1(x,y) = x \;\;\text{if $(x,y) \in \R^2$,}\\
&\theta_2 : X_0 \to \R : \theta_2(x) = x \;\;\text{if $x \in \R$,}\;\; \theta_2(x,y) = y \;\;\text{if $(x,y) \in \R^2$.}
\end{align*}
Note that $T(z) = \theta_1(z) - \theta_2(z)$ for all $z \in X_0$. We consider the associated measure preserving factor maps
$$\pi : X \to \Om : (\pi(z))_n = \theta_1(z_n) \;\;\text{and}\;\; \rho_k : X \to \Om : (\rho_k(z))_n = \begin{cases} \theta_2(z_n) &\;\text{if $n_k + 1 \leq n \leq n_{k+1}$,}\\
\theta_1(z_n) &\;\text{otherwise.}\end{cases}$$

Define the isometries $V : H \to K$ and $W_k : H \to K$ by
$$V(\xi)(t,z) = \xi(t,\pi(z)) \quad\text{and}\quad W_k(\xi)(t,z) = \xi(t,\rho_k(z)) \; .$$
When $\xi \in H_N$, we have $W_k(\xi) = V(\xi)$ for all $k$ large enough. By density, we get that $W_k \to V$ strongly.

Let $F \in L^\infty(\R \times \Om)$ be a function that generates the tail boundary of $(\mu_n)_{n \in \N}$. Since $F$ belongs to the tail boundary algebra, we have
\begin{equation}\label{eq.the-algebraic-relation}
(W_k(F))(t,z) = (V(F))(t-S_k(z),z) \quad\text{for all $k \in \N$ and a.e.\ $(t,z) \in \R \times X$.}
\end{equation}
Since
\begin{equation}\label{eq.my-inequality}
\frac{1+(t+s)^2}{1+t^2} \leq 2(1+s^2) \quad\text{for all $s,t \in \R$,}
\end{equation}
and since each $S_k$ is a bounded function, we can define the bounded linear operators
$$R_k : K \to K : (R_k(\xi))(t,z) = \xi(t-S_k(z),z) \; .$$
Write $D = \sqrt{2(1+(50C)^2)}$. We claim that
\begin{equation}\label{eq.good-estimate}
\|R_k(\xi)\|_2^2 \leq D \, \|\xi\|_\infty \, \|\xi\|_2 \quad\text{for all $\xi \in L^\infty(\R \times X) \subset K$.}
\end{equation}
To prove this claim, note that the left hand side can be estimated, using the Cauchy-Schwartz inequality and \eqref{eq.my-inequality}, by
\begin{align*}
\frac{1}{\pi}  \int_{X} d\zeta(z) \, &\int_\R dt \; \frac{1}{(1 + (t+S_k(z))^2)} \, |\xi(t,z)|^2 \\
& \leq \|\xi\|_\infty \, \int_{X} d\zeta(z) \, \int_\R dt \; \frac{1}{\sqrt{\pi (1+ t^2)}} \, |\xi(t,z)| \, \frac{\sqrt{1+t^2}}{\sqrt{\pi} (1+(t+S_k(z))^2)} \\
& \leq \|\xi\|_\infty \, \|\xi\|_2 \, \Bigl( \int_{X} d\zeta(z) \, \int_\R dt \; \frac{1 + t^2}{\pi (1+(t+S_k(z))^2)^2}\Bigr)^{1/2} \\
& = \|\xi\|_\infty \, \|\xi\|_2 \, \Bigl( \int_{X} d\zeta(z) \, \int_\R dt \; \frac{1 + (t-S_k(z))^2}{\pi (1+t^2)^2}\Bigr)^{1/2} \\
& \leq \|\xi\|_\infty \, \|\xi\|_2 \, \Bigl( \int_{X} d\zeta(z) \, \int_\R d\mu_0(t) \; 2 (1+S_k(z)^2)\Bigr)^{1/2} = D \, \|\xi\|_\infty \, \|\xi\|_2 \; .
\end{align*}
So, the claim is proven.

Define the probability measure $\eta_k = (S_k)_*(\zeta)$. Then
$$\int_\R t^2 \, d\eta_k(t) = E(S_k^2) \leq (50 C)^2 \quad\text{for all $k \in \N$,}$$
so that $\eta_k$ is a tight family of probability measures. Choose a sequence $k_j \to \infty$ such that $\eta_{k_j}$ converges weakly to a probability measure $\eta$ on $\R$ with $\int_\R t^2 \, d\eta(t) \leq (50 C)^2 < +\infty$. By \eqref{eq.not-to-zero}, we have that $\eta_k((-\infty,40 C \delta]) \geq 1/12$ for all $k \in \N$. Therefore, $\eta \neq \delta_0$.

We can then define the bounded convolution operator
$$K \to K : \xi \mapsto \eta * \xi : (\eta * \xi)(t,z) = \int_\R \xi(t-s,z) \, d\eta(s) \; .$$
Note that by the Cauchy-Schwarz inequality and \eqref{eq.my-inequality},
\begin{align*}
\|\eta * \xi\|_2^2 &\leq \int_X d\zeta(z) \, \int_\R d\eta(s) \, \int_\R d\mu_0(t) \, |\xi(t-s,z)|^2 \\ &\leq \int_X d\zeta(z) \, \int_\R d\eta(s) \, \int_\R d\mu_0(t) \, 2(1+s^2) \, |\xi(t,z)|^2 \; ,
\end{align*}
so that $\|\eta * \xi\|_2 \leq D \, \|\xi\|_2$ for all $\xi \in K$. We have a similar convolution operator $\xi \mapsto \eta * \xi$ on $H$. Note that $V(\eta * \xi) = \eta * V(\xi)$.

We finally prove that for every $\xi \in  L^\infty(\R \times X) \subset K$ and every $\xi' \in K$, we have that
\begin{equation}\label{eq.my-goal}
\lim_j \langle R_{k_j}(\xi),\xi' \rangle = \langle \eta * \xi , \xi' \rangle \; .
\end{equation}
Define the closed subspaces $K_N = L^2(\R \times \prod_{n=1}^N X_0)$ of $K$. Fix $\xi \in L^\infty(\R \times X) \subset K$ and $\xi' \in K$. The orthogonal projection $E_N : K \to K_N$ corresponds to the conditional expectation, so that $\|E_N(\xi)\|_\infty \leq \|\xi\|_\infty$ for all $N \in \N$. By \eqref{eq.good-estimate}, the sequence $\|R_{k_j}(\xi)\|_2$ is bounded. To prove \eqref{eq.my-goal}, we may thus assume that $\xi' \in K_N$ for some $N$. By \eqref{eq.good-estimate}, we also have that
\begin{align*}
\|R_{k_j}(\xi) - R_{k_j}(E_N(\xi))\|_2^2 &= \|R_{k_j}(\xi - E_N(\xi))\|_2^2 \leq D \, \|\xi - E_N(\xi)\|_\infty \, \|\xi - E_N(\xi)\|_2 \\
& \leq 2 D \, \|\xi\|_\infty \, \|\xi - E_N(\xi)\|_2 \; ,
\end{align*}
which tends to zero as $N \to \infty$, uniformly in $j$. To prove \eqref{eq.my-goal}, we may thus also assume that $\xi \in K_N$. When $\xi,\xi' \in K_N$ and $n_{k_j} > N$, we have
$$\langle R_{k_j}(\xi),\xi' \rangle = \langle \eta_{k_j} * \xi , \xi' \rangle$$
for all $j$, so that \eqref{eq.my-goal} follows.

We now return to our element $F$ generating the tail boundary. By \eqref{eq.the-algebraic-relation}, we get that $W_k(F) = R_k(V(F))$. Since $W_k \to V$ strongly, we get that $\|W_k(F) - V(F)\|_2 \to 0$. By \eqref{eq.my-goal}, we have that $R_k(V(F)) \to \eta * V(F) = V(\eta * F)$ weakly. We thus conclude that $F = \eta * F$. Since $\eta \neq \delta_0$, it follows from the Choquet-D\'{e}ny theorem (see \cite{CD60}) that $F$ is periodic in the first variable. So, the tail boundary flow is periodic.
\end{proof}

In our applications of Proposition \ref{prop.periodicity-criterion}, we will use a few times the following elementary equality and estimate
\begin{equation}\label{eq.elementary-variance}
\begin{split}
\Var \be &= \frac{1}{2} \int_{\R^2} (x-y)^2 \, d\be(x) \, d\be(y) \\ &\geq \frac{1}{2} \int_{(\R^* \times \{0\})\cup (\{0\} \times \R^*)} (x-y)^2 \, d\be(x) \, d\be(y)
= \be(\{0\}) \, \int_\R x^2 \, d\be(x) \; ,
\end{split}
\end{equation}
for every probability measure $\be$ on $\R$.

Proposition \ref{prop.periodicity-criterion} already says that aperiodic tail boundary flows can only occur if the measures $(\mu_n)_{n \in \N}$ are sufficiently sparse: for each sequence of intervals $I_n \subset \R$ with uniformly bounded length $|I_n|$, there exist points $s_n \in I_n$ such that
\begin{equation}\label{eq.concentration-one}
\sum_{n=1}^\infty \int_{I_n} (x-s_n)^2 \, d\mu_n(x) < +\infty \; .
\end{equation}
In the following proposition, we develop this further and prove that, under the appropriate assumptions, the resulting sequence $s_n$ must itself be sparse: partitioning $\R$ into intervals $I_k$ of uniformly bounded length, we may assume that within each $I_k$, the points $s_n$ lie close to a single element $t_k \in I_k$.

The proof of Proposition \ref{prop.further-concentration} is closely inspired by \cite[Proposition 1.1]{GS83} and \cite[Lemma 8.6]{AW68}. Although this is not strictly needed for the rest of this paper, we use the opportunity to also deduce from Proposition \ref{prop.further-concentration} a more conceptual proof for the main result of \cite{GS83} saying that every ITPFI factor of bounded type is isomorphic with an ITPFI$_2$ factor; see Theorem \ref{thm.ITPFI-bounded}. Recall that, by definition, an ITPFI factor of bounded type is an infinite tensor product of matrix algebras $M_{n_k}(\C)$ with $\sup_k n_k < +\infty$.

\begin{proposition}\label{prop.further-concentration}
Let $(\mu_n)_{n \in \N}$ be probability measures on $\R$. Let $(J_k)_{k \in \N}$ be a family of disjoint subsets of $\N$. Assume that for every $k \in \N$, we are given $p_k, q_k > 0$ and an interval $I_k \subset \R \setminus (-1,1)$. Assume that $|I_k| \leq C$ for all $k \in \N$.

Assume that for each $k \in \N$ and $n \in J_k$, we are given a probability measure $\be_n$ with $\be_n(I_k) = 1$ and $p_k \delta_0 + q_k \beta_n \leq \mu_n$.

If the tail boundary flow of $(\mu_n)_{n \in \N}$ is aperiodic, each $J_k$ is a finite set and there exist $t_k \in I_k$ such that
\begin{equation}\label{eq.concentration-two}
\sum_{k=1}^\infty \sum_{n \in J_k} p_k q_k \int_\R (x-t_k)^2 \, d\be_n(x) < +\infty \; .
\end{equation}
\end{proposition}

Note that there is a twofold difference between \eqref{eq.concentration-one} and \eqref{eq.concentration-two}. In \eqref{eq.concentration-two}, the concentration points $t_k$ only depend on the interval $I_k$ and are thus the same for each $n \in J_k$. On the other hand, the factor $p_k q_k$ appearing in \eqref{eq.concentration-two} is strictly smaller than the factor $q_k = (q_k \beta_n)(I_k)$ that we would get in \eqref{eq.concentration-one}.

\begin{proof}
For a fixed $k \in \N$, we have that the measures $(p_k \delta_0 + q_k \beta_n)_{n \in J_k}$ have a uniformly bounded support in $\{0\} \cup I_k$. It thus follows from Proposition \ref{prop.periodicity-criterion} that
$$(p_k + q_k) \sum_{n \in J_k} \Var((p_k+q_k)^{-1} (p_k \delta_0 + q_k \beta_n)) < +\infty \; .$$
Since $I_k \subset \R \setminus (-1,1)$ and $\be_n$ is supported on $I_k$, by \eqref{eq.elementary-variance}, the left hand side is larger or equal than $p_k q_k (p_k+q_k)^{-1} |J_k|$, so that $J_k$ is a finite set.

We can also apply Proposition \ref{prop.periodicity-criterion} to all the finite measures $q_k \beta_n \leq \mu_n$, with $k \in \N$ and $n \in J_k$, because they have uniformly bounded width $C$. Defining for any $k \in \N$ and $n \in J_k$, the point $s_n \in I_k$ by
$$s_n = \int_\R x \, d\beta_n(x) \; ,$$
we get that
\begin{equation}\label{eq.limit-here}
\sum_{k=1}^\infty \sum_{n \in J_k} q_k \int_\R (x-s_n)^2 \, d\beta_n(x) < +\infty \; .
\end{equation}
For every $k \in \N$, we consider the finitely many points $(s_n)_{n \in J_k}$ in the interval $I_k$. We denote by $t_k$ a ``middle point''. More precisely, if $|J_k|$ is odd, we write $J_k = A_k \sqcup B_k \sqcup \{j_k\}$ with $|A_k| = |B_k|$ such that $s_n \leq s_{j_k} \leq s_m$ for all $n \in A_k$ and $m \in B_k$. We put $t_k = s_{j_k}$. If $|J_k|$ is even, we write $J_k = A_k \sqcup B_k$ with $|A_k| = |B_k|$ such that for some $t_k \in I_k$, we again have that $s_n \leq t_k \leq s_m$ for all $n \in A_k$ and $m \in B_k$. We choose a bijection $\al_k : A_k \to B_k$.

Since the tail boundary flow of $(\mu_n)_{n \in \N}$ is aperiodic, also the measures $\mu_n * \mu_{\al_k(n)}$ with $k \in \N$ and $n \in A_k$ have an aperiodic tail boundary. By construction,
$$p_k q_k (\beta_n + \beta_{\al_k(n)}) \leq \mu_n * \mu_{\al_k(n)} \quad\text{for all $k \in \N$, $n \in A_k$.}$$
Moreover, the measures on the left have their support in $I_k$, with $|I_k| \leq C$ for all $k \in \N$. Again applying Proposition \ref{prop.periodicity-criterion}, we conclude that
$$\sum_{k=1}^\infty \sum_{n \in A_k} p_k q_k \, \Var((\beta_n + \beta_{\al_k(n)})/2) < + \infty \; .$$
The mean value of $(\beta_n + \beta_{\al_k(n)})/2$ is $(s_n + s_{\al_k(n)})/2$ and we conclude that
\begin{equation}\label{eq.intermediate}
\sum_{k=1}^\infty \sum_{n \in A_k} p_k q_k \int_\R \bigl(x - \frac{s_n + s_{\al_k(n)}}{2}\bigr)^2 \, d\beta_n(x) < + \infty \; .
\end{equation}
A direct computation gives that for $n \in A_k$,
$$\int_\R \bigl(x - \frac{s_n + s_{\al_k(n)}}{2}\bigr)^2 \, d\beta_n(x) = \Var(\beta_n) + \frac{1}{4}(s_{\al_k(n)}-s_n)^2 \geq \frac{1}{4}\bigl((s_{\al_k(n)}-t_k)^2 + (t_k - s_n)^2\bigr)$$
because $s_n \leq t_k \leq s_{\al_k(n)}$. It thus follows from \eqref{eq.intermediate} that
$$\sum_{k = 1}^\infty \sum_{n \in J_k} p_k q_k \, (s_n - t_k)^2 = \sum_{k=1}^\infty \sum_{n \in A_k} p_k q_k \, \bigl((s_{\al_k(n)}-t_k)^2 + (t_k - s_n)^2\bigr) < +\infty \; .$$
Since $(x-t_k)^2 \leq 2(x-s_n)^2 + 2(s_n -t_k)^2$ and since $p_k \leq 1$, in combination with \eqref{eq.limit-here}, we have proven that \eqref{eq.concentration-two} holds.
\end{proof}

\begin{theorem}\label{thm.ITPFI-bounded}
Every ITPFI factor of bounded type is isomorphic with an ITPFI$_2$ factor.
\end{theorem}
\begin{proof}
By induction, it suffices to prove that for every integer $N \geq 2$, every ITPFI$_{N+1}$ factor of type III$_0$ is isomorphic with the tensor product of an ITPFI$_N$ factor and an ITPFI$_2$ factor.

We denote in this proof by $\delta(a)$ the Dirac measure in $a \in \R$. For every $a \geq 0$, we define the probability measure
\begin{equation}\label{eq.nice-notation}
\gamma(a) = (1+\exp(-a))^{-1} \, \bigl(\delta(0) + \exp(-a) \delta(a)\bigr) \; .
\end{equation}
For every $a \in \R_{\geq 0}^N$, we define the probability measure
$$\rho(a) = \bigl(1 + \sum_{i=1}^N \exp(-a_i)\bigr)^{-1} \, \bigl(\delta(0) + \sum_{i=1}^N \exp(-a_i) \delta(a_i)\bigr)$$
and the state $\psi_a$ on $M_{N+1}(\C)$ by
$$\psi_a(A) = \bigl(1 + \sum_{i=1}^N \exp(-a_i)\bigr)^{-1} \, \bigl(A_{00} + \sum_{i=1}^N \exp(-a_i) A_{ii} \bigr)\; .$$
By diagonalizing states, any ITPFI$_{N+1}$ factor can be written as the tensor product of a sequence $(M_{N+1}(\C),\psi_{a_n})$ with $a_n \in \R_{\geq 0}^N$. By \cite[Theorem 3.1]{CW88}, its flow of weights is precisely the tail boundary flow of the sequence $(\rho(a_n))_{n \in \N}$.

We thus fix such a sequence $a_n \in \R_{\geq 0}^N$, we assume that the tail boundary flow of
$$\mu_n = \rho(a_n) \;\; , \;\; n \in \N \; ,$$
is aperiodic and we prove that it is isomorphic with the tail boundary flow of a family of probability measures of the form $\rho(b)$, $b \in \R_{\geq 0}^{N-1}$, and $\gamma(c)$, $c \geq 0$.

For every $k \in \N$ and $1 \leq i \leq N$, we denote $J_{k,i} = \{n \in \N \mid a_{n,i} \in [k-1,k)\}$. We put $p = (1+N)^{-1}$ and $q_k = p \, \exp(-k)$. Fix $i \in \{1,\ldots,N\}$.

For every $n \in J_{1,i}$, we have that
$$p \, \delta(0) + p \exp(-a_{n,i}) \, \delta(a_{n,i}) \leq \mu_n$$
and that the measures $\delta(a_{n,i})$ are supported on the interval $[0,1]$. It follows from Proposition \ref{prop.periodicity-criterion} and \eqref{eq.elementary-variance} that
\begin{equation}\label{eq.conc-1}
\sum_{n \in J_{1,i}} \exp(-a_{n,i}) \, a_{n,i}^2 < +\infty \; .
\end{equation}
For every $k \geq 2$ and $n \in J_{k,i}$, we have that
$$p \, \delta(0) + q_k \, \delta(a_{n,i}) \leq \mu_n$$
and that the measures at the left are supported on the interval $[k-1,k)$. For $n \in J_{k,i}$, we have that $\exp(-a_{n,i}) \leq \exp(1) \, (N+1) \, q_k$. By Proposition \ref{prop.further-concentration}, we thus find $b_{k,i} \in [k-1,k)$ such that
\begin{equation}\label{eq.conc-2}
\sum_{k=2}^\infty \sum_{n \in J_{k,i}} \exp(-a_{n,i}) \, (b_{k,i} - a_{n,i})^2 < +\infty \; .
\end{equation}
Defining $b_{1,i}=0$ for all $i \in \{1,\ldots,N\}$ and summing over $i \in \{1,\ldots,N\}$, it follows from \eqref{eq.conc-1} and \eqref{eq.conc-2} that
\begin{equation}\label{eq.conc}
\sum_{k=1}^\infty \sum_{i=1}^N \sum_{n \in J_{k,i}} \exp(-a_{n,i}) \, (b_{k,i} - a_{n,i})^2 < +\infty \; .
\end{equation}
Define the maps $\al_n : \{1,\ldots,N\} \to \N$ by $\al_n(i) = k$ if and only if $n \in J_{k,i}$. Define the probability measures
$$\mu'_n = \bigl(1 + \sum_{i=1}^N \exp(-a_{n,i})\bigr)^{-1} \, \bigl(\delta(0) + \sum_{i=1}^N \exp(-a_{n,i}) \delta(b_{\al_n(i),i})\bigr) \; .$$
Using the Wasserstein $2$-distance, it follows from \eqref{eq.conc} and \eqref{eq.wasserstein-estimate} that $\sum_{n=1}^\infty W_2(\mu_n,\mu'_n)^2 < +\infty$. It thus follows from Proposition \ref{prop.iso-wasserstein} that $(\mu_n)_{n \in \N}$ and $(\mu'_n)_{n \in \N}$ have isomorphic tail boundary flows.

When $s, t \in \R$ satisfy $|s-t| \leq 1$, we have that
$$\bigl(\exp(-s/2) - \exp(-t/2)\bigr)^2 = \exp(-s) \, \bigl(1-\exp((s-t)/2)\bigr)^2 \leq \exp(-s) \, (s-t)^2 \; .$$
It follows that for all $n \in J_{k,i}$,
\begin{equation}\label{eq.star}
\bigl(\exp(-a_{n,i}/2) - \exp(-b_{k,i}/2)\bigr)^2 \leq \exp(-a_{n,i}) \, (b_{k,i} - a_{n,i})^2 \; .
\end{equation}
Defining $\mu\dpr_n = \rho(b_{\al_n(1),1},\ldots,b_{\al_n(N),N})$, it thus follows from \eqref{eq.conc} and \eqref{eq.star} that
$$\sum_{n=1}^\infty H^2(\mu'_n,\mu\dpr_n) < + \infty \; .$$
By Lemma \ref{lem.close-in-hellinger}, also $(\mu'_n)_{n \in \N}$ and $(\mu\dpr_n)_{n \in \N}$ have isomorphic tail boundary flows.

We order the elements $(b_{k,i})_{i = 1,\ldots,N}$ in $[k-1,k)$ from smaller to larger, remove duplicates and relabel, so as to find an increasing sequence $0 \leq b_1 < b_2 < \cdots$ with $b_m \geq m/N - 1$ for all $m \in \N$ and with all $b_{k,i}$ appearing in this sequence. Permuting elements $(d_1,\ldots,d_N)$ does not change the probability measure $\rho(d_1,\ldots,d_N)$. So every $\mu\dpr_n$ is of the form
$$\mu\dpr_n = \rho(b_{\theta_n(1)},\ldots,b_{\theta_n(N)})$$
for a nondecreasing function $\theta_n : \{1,\ldots,N\} \to \N$.

For every $1 \leq M \leq N$, denote by $\cF_M$ the set of nondecreasing functions from $\{1,\ldots,M\}$ to $\N$. For every $\theta \in \cF_M$, denote
$$\zeta(\theta) = \rho(b_{\theta(1)},\ldots,b_{\theta(M)}) \; .$$
Define the subset $\cJ \subset \cF_N$ as the set of all $\theta \in \cF_N$ that are of the form $\theta_n$ for some $n \in \N$. Define $L_\theta \in \N$ as the cardinality of $\{ n \in \N \mid \theta_n = \theta\}$. As explained in the beginning of Section \ref{sec.tail-boundary}, the tail boundary flow does not depend on the order of the transition probability measures. So, by construction, the tail boundary flow of the family $(\mu\dpr_n)_{n \in \N}$, and hence also the tail boundary flow of $(\mu_n)_{n \in \N}$, is isomorphic with the tail boundary flow of the family of measures $\zeta(\theta)^{* L_\theta}$, $\theta \in \cJ$.

For every $\theta \in \cF_N$, we denote by $\thetah$ the restriction of $\theta$ to $\{1,\ldots,N-1\}$. For every $\theta \in \cJ$, we apply Lemma \ref{lem.binomial-approx} below to
\begin{align*}
& \al = \sum_{i=1}^{N-1} \exp(-b_{\theta(i)}) \geq \be = \exp(-b_{\theta(N)}) \quad\text{and}\\
& P = \al^{-1} \sum_{i=1}^{N-1} \exp(-b_{\theta(i)}) \delta(b_{\theta(i)}) \;\; , \;\; Q = \delta(b_{\theta(N)}) \; .
\end{align*}
By Lemma \ref{lem.binomial-approx}, we thus find $K_\theta, M_\theta \in \N$ such that with the notation of \eqref{eq.nice-notation},
\begin{equation}\label{eq.almost}
\dTV\bigl(\zeta(\theta)^{*L_\theta} , \zeta(\thetah)^{* K_\theta} * \gamma(b_{\theta(N)})^{* M_\theta}\bigr) \leq 4 \exp(-b_{\theta(N)}/2) \; .
\end{equation}
Given $k \in \N$, the number of elements $\theta \in \cF_N$ with $\theta(N) = k$ is smaller than $k^{N-1}$. Also recall that $b_k \geq k/N - 1$. Thus,
$$\sum_{\theta \in \cF_N} \exp(-b_{\theta(N)}) \leq \sum_{k=1}^\infty k^{N-1} \exp(-b_k) \leq \sum_{k=1}^\infty k^{N-1} \, \exp(1) \, \exp(-k/N) < +\infty \; .$$
It then follows from \eqref{eq.almost} and Lemma \ref{lem.close-in-hellinger} that the tail boundary flow of $(\mu_n)_{n \in \N}$ is isomorphic with the tail boundary flow of the family of measures
$$\{\zeta(\thetah)^{* K_\theta} \mid \theta \in \cJ\} \cup \{\gamma(b_{\theta(N)})^{* M_\theta} \mid \theta \in \cJ \} \; .$$
This concludes the proof of the theorem.
\end{proof}

Although the following lemma is an immediate consequence of \cite[Lemma 2.2]{GS83}, it has never been stated in this very general form.

\begin{lemma}\label{lem.binomial-approx}
Let $P$ and $Q$ be probability measures on $\R$ and $\al, \be > 0$. For every $L \in \N$, we have that
\begin{multline*}
\dTV\Bigl( \bigl((1+\al+\be)^{-1}(\delta_0 + \al P + \be Q)\bigr)^{*L} \; , \\ \bigl((1+\al)^{-1}(\delta_0+\al P)\bigr)^{* K} \; * \; \bigl((1+\be)^{-1} (\delta_0 + \be Q)\bigr)^{* M}\Bigr) \leq 2 \sqrt{\be}
\end{multline*}
for $K = L - \lfloor (1+\al+\be)^{-1} \be L \rfloor$ and $M = \lfloor (1+\al+\be)^{-1} (1+\be) L \rfloor$.
\end{lemma}
\begin{proof}
For integers $k,m \geq 0$ with $k+m \leq L$, we write
$$\mu(k,m) = \frac{L!}{(L-k-m)! \, k! \, m!} \, \frac{\al^k \, \be^m}{(1+\al+\be)^L} \; .$$
For all other integers $k,m \geq 0$, we write $\mu(k,m) = 0$. Then,
$$\bigl((1+\al+\be)^{-1}(\delta_0 + \al P + \be Q)\bigr)^{*L} = \sum_{k,m \geq 0} \mu(k,m) \, P^{*k} \, Q^{*m} \; .$$
We similarly write for all integers $0 \leq k \leq K$ and $0 \leq m \leq M$,
$$\mu'(k,m) = \frac{K!}{(K-k)! \, k!} \, \frac{\al^k}{(1+\al)^K} \, \frac{M!}{(M-m)! \, m!} \, \frac{\be^m}{(1+\be)^M}$$
and $\mu'(k,m) = 0$ for all other integers $k,m \geq 0$, so that
$$\bigl((1+\al)^{-1}(\delta_0+\al P)\bigr)^{* K} \; * \; \bigl((1+\be)^{-1} (\delta_0 + \be Q)\bigr)^{* M} = \sum_{k,m \geq 0} \mu'(k,m) \, P^{*k} \, Q^{*m} \; .$$
Parts (a) and (b) of the proof of \cite[Lemma 2.2]{GS83} are saying that $H^2(\mu,\mu') \leq 2 \beta$. Then also $\dTV(\mu,\mu') \leq 2 \sqrt{\beta}$.
\end{proof}

\subsection{Poisson flows}

Recall that to every finite positive measure $\mu$ on $\R$ is associated the \emph{compound Poisson distribution} on $\R$, which is defined as the probability measure
$$\cE(\mu) = \exp(-\|\mu\|) \, \exp(\mu) = \exp(-\mu(\R)) \, \Bigl( \delta_0 + \sum_{k=1}^\infty \frac{1}{k!} \mu^{*k} \Bigr) \; .$$
Note that $\cE(\mu)$ is supported on $\R_{\geq 0} = [0,+\infty)$ iff $\mu$ is supported on $\R_{\geq 0}$. If $x \mapsto x^2$ is $\mu$-integrable, we have
\begin{equation}\label{eq.poisson-mean-variance}
\int_\R x \, d\cE(\mu)(x) = \int_\R x \, d\mu(x) \quad\text{and}\quad \Var(\cE(\mu)) = \int_\R x^2 \, d\mu(x) \; .
\end{equation}

\begin{definition}\label{def.poisson-flow}
We call a \emph{Poisson flow} any ergodic flow $\R \actson (Z,\eta)$ that arises as the tail boundary flow of a sequence of compound Poisson distributions on $\R$. If these compound Poisson distributions can be chosen with support in $\R_{\geq 0}$, we call $\R \actson (Z,\eta)$ a \emph{Poisson flow of positive type}.
\end{definition}

Note that compound Poisson distributions and tail boundary flows make sense on any locally compact second countable abelian group $G$, leading to the concept of a Poisson $G$-action, which we will use for $G = \R$ and $G = \Z$. In Remark \ref{rem.concluding-flows}, we discuss the relation between general Poisson flows and Poisson flows of positive type.

The two main result of this section are the following.

\begin{theorem}\label{thm.almost-periodic-Poisson}
Every ergodic almost periodic flow is a Poisson flow of positive type.
\end{theorem}

\begin{theorem}\label{thm.Poisson-vs-ITPFI-2}
The Poisson flows of positive type are precisely the flows of weights of ITPFI$_2$ factors.
\end{theorem}

By \cite[Theorem 2.1]{GS83}, which we reproved as Theorem \ref{thm.ITPFI-bounded} above, every ITPFI factor of bounded type is isomorphic with an ITPFI$_2$ factor. So Poisson flows of positive type are also precisely the flows of weights of ITPFI factors of bounded type. Note that Theorem \ref{thm.almost-periodic-ITPFI-2} is an immediate consequence of Theorems \ref{thm.almost-periodic-Poisson} and \ref{thm.Poisson-vs-ITPFI-2}.

For Poisson flows with a nontrivial eigenvalue group, Theorem \ref{thm.Poisson-vs-ITPFI-2} was proven in \cite[Proposition 7.1]{GH08}. It is easy to see that $2\pi / p$ is an eigenvalue of a Poisson flow $\R \actson (Z,\eta)$ iff $\R \actson (Z,\eta)$ is the tail boundary flow of a sequence of compound Poisson distributions with support in $p \Z$. The main step in the proof of Theorem \ref{thm.Poisson-vs-ITPFI-2} is to show that in general, if $\R \actson (Z,\eta)$ is a Poisson flow that is aperiodic, then we may realize $\R \actson (Z,\eta)$ as the tail boundary flow of a sequence of compound Poisson distributions with very sparse support: at most one atom in each length one interval.

Before proving Theorem \ref{thm.almost-periodic-Poisson}, we recall the following definition from \cite[Section 3]{HOO74}, which we generalize in the natural way to multiple flows of locally compact abelian groups. We only use this concept for $G = \R$ and $G = \Z$.

\begin{definition}[{\cite[Section 3]{HOO74}}]\label{def.joint-flow}
Let $G$ be a locally compact second countable abelian group. For $i \in \{1,\ldots,n\}$, let $G \actson (Z_i,\eta_i)$ be a nonsingular action. Consider the direct sum $G^{\oplus n}$ and its natural action $G^{\oplus n} \actson Z = Z_1 \times \cdots \times Z_n$. Define the closed subgroup $G_0 = \{(g_1,\ldots,g_n) \in G^{\oplus n} \mid g_1 + \cdots + g_n = 0\}$ and the ergodic decomposition $A = L^\infty(Z)^{G_0}$ for the action $G_0 \actson Z$.

The action of $G$ on $A$ in the first variable (or, which is the same, in any of the other variables) is called the \emph{joint action} of $G \actson (Z_i,\eta_i)$. When $G = \R$, we use the terminology \emph{joint flow}, instead of joint action.
\end{definition}

By construction, the tail boundary flow of the disjoint union $(\mu_n)_{n \in I \sqcup J}$ of two countable infinite families of probability measures on $\R$ is the joint flow of the tail boundary flows of $(\mu_n)_{n \in I}$ and $(\mu_n)_{n \in J}$. Also by construction, the flow of weights of a tensor product factor $M_1 \ovt M_2$ is the joint flow of the flows of weights of $M_i$.

As a final preparation to proving Theorem \ref{thm.almost-periodic-Poisson}, we discuss the relation between tail boundary flows and induced actions. Assume that $G$ is a locally compact second countable abelian group with closed subgroup $H$. Choose a probability measure $\nu$ on $G/H$ that is equivalent with the Haar measure. Recall that any nonsingular action $H \actson (Y,\eta)$ has an induced action $G \actson (G/H \times Y,\nu \times \eta)$, which is defined as follows. Choosing a Borel lift $\psi : G/H \to G$, one defines the $1$-cocycle $\Om : G \times G/H \to H : \Om(g,x) = g + \psi(x) - \psi(g+x)$ and the induced action
$$G \actson G/H \times Y : g \cdot (x,y) = (g + x, \Om(g,x) \cdot y) \; .$$
Another choice of lift $\psi$ gives a cohomologous $1$-cocycle and thus, an isomorphic induced action.

If $(\mu_n)_{n \in \N}$ is a family of probability measures on $H$, we have an associated tail boundary $H$-action. We can also view $(\mu_n)_{n \in \N}$ as a family of probability measures on $G$. Since the induction of the translation action $H \actson H$ is, by construction, the translation action $G \actson G$, we have by definition that the tail boundary $G$-action of $(\mu_n)_{n \in \N}$ is isomorphic with the induction to $G$ of the tail boundary $H$-action of $(\mu_n)_{n \in \N}$.

\begin{proof}[{Proof of Theorem \ref{thm.almost-periodic-Poisson}}]
The cases of a trivial flow or a periodic flow are straightforward. By e.g.\ \cite[Proposition 2.1]{BV20}, taking $a > 0$ and $\mu_n = \cE(\delta_a)$ for all $n \in \N$, the tail boundary flow is the periodic flow $\R \actson \R/a\Z$. Similarly, taking $a > 0$ irrational and $\mu_n = \cE(\delta_1 + \delta_a)$ for all $n \in \N$, the tail boundary flow is trivial.

So, it suffices to consider an almost periodic, aperiodic, ergodic flow. Such a flow is given by the translation action of $\R$ on a compact second countable abelian group $L$ under a dense embedding $\pi : \R \to L$. We have to realize this flow as the tail boundary flow of a sequence of compound Poisson distributions $\cE(\mu_n)$ with $\mu_n$ supported on $\R_{\geq 0}$. We start by reducing this problem to a similar question for $\Z$-actions.

Fix a nontrivial character $\om_0 \in \Lh$. Take $t_0 \in \R \setminus \{0\}$ such that $\om_0(\pi(t)) = \exp(2\pi i t / t_0)$ for all $t \in \R$. So, $\om_0$ gives rise to a surjective, continuous group homomorphism $\theta : L \to \R/t_0 \Z$ such that $\theta \circ \pi : \R \to \R / t_0 \Z$ is the natural quotient map. Define $K = \Ker \theta$ as the kernel of $\theta$. Note that the restriction of $\pi$ to $t_0 \Z$ is a dense embedding of $t_0 \Z$ into $K$. Note that the translation action $\R \actson L$ is the induction to $\R$ of the translation action $t_0 \Z \actson K$. By the discussion preceding this proof, it thus suffices to prove that this translation action $t_0 \Z \actson K$ can be realized as the tail boundary action of $t_0 \Z$ associated with a sequence of compound Poisson distributions $\cE(\mu_n)$, where $\mu_n$ are finite positive measures on $t_0 \N \subset t_0 \Z$. Since we can rescale everything with $t_0$, we may assume that $t_0 = 1$.

Combining \cite[Corollary 2.8]{CW83} and \cite[Theorem 3.4]{CW88}, we can take finitely supported probability measures $(\eta_n)_{n \in \N}$ on $\Z$ such that the associated tail boundary flow is given by the translation $\Z \actson K$. Denote by $\etatil_n$ the probability measure given by $\etatil_n(\cU) = \eta_n(-\cU)$. The tail boundary flow of $(\etatil_n)_{n \in \N}$ is given by the action of $\Z$ on $K$ by $n \cdot k = -n + k$. Through the isomorphism $k \mapsto -k$, this flow is isomorphic with the original translation action $\Z \actson K$. Since the joint flow of $\Z \actson K$ with itself is again $\Z \actson K$, the tail boundary flow of the measures $(\eta_n * \etatil_n)_{n \in \N}$ is still given by $\Z \actson K$.

For every $\mu \in \ell^1(\Z)$, we consider the Fourier transform
$$\muh : \T \to \C : \muh(z) = \sum_{k \in \Z} \mu(k) z^k \; .$$
We view $\Kh$ as a dense countable subgroup of $\T$. The Fourier transform of $\eta_n * \etatil_n$ is a positive function. We replace $\eta_n$ by $\eta_n * \etatil_n$ and we may thus assume that the probability measures are finitely supported, with $\etah_n(\om) \geq 0$ for all $n \in \N$ and $\om \in \Kh$, and with associated tail boundary flow $\Z \actson K$.

As each $\omega\in \widehat{K}$ is an eigenvalue of the translation action $\mathbb{Z}\curvearrowright K$, we know from \cite[Theorem 4.2]{CW88} that
\begin{align*}
\lim_{n \to +\infty} \prod_{m=n}^{\infty}\etah_m(\omega) = 1 \;\;\text{for every $\omega\in \Kh$.}
\end{align*}
Choose an increasing sequence of finite subsets $\cF_k \subset \widehat{K}$ such that $\widehat{K} = \bigcup_{k=1}^\infty \cF_k$. Then inductively choose $1 \leq n_1 < n_2 < \cdots$ such that
$$\prod_{m=1+n_k}^\infty \etah_m(\omega) > 1 - (k+1)^{-3} \quad\text{for all $\om \in \cF_{k+1}$.}$$
Then define the probability measures $\alpha_k$ on $\Z$ by
\begin{align}\label{eq:contracted sequence}
\alpha_1=\eta_1 * \eta_2 * \cdots * \eta_{n_1} \;\;\text{and}\;\; \alpha_k=\eta_{1+n_{k-1}} * \cdots * \eta_{n_k} \;\;\text{for $k \geq 2$.}
\end{align}
Since $0 \leq \etah_m(\om) \leq 1$ for all $m \in \N$, $\om \in \Kh$, we thus have
\begin{equation}\label{eq.already-nice}
\alh_k(\om) \geq 1 - k^{-3} \quad\text{for every $\om \in \cF_k$.}
\end{equation}
Also, the tail boundary flow of $(\alpha_k)_{k \in \N}$ is still given by $\Z \actson K$.

For every $k \in \N$,
$$\cU_k = \{g \in K \mid \, |\om(g) - 1| < k^{-3} \;\;\text{for all $\om \in \cF_k$}\;\}$$
is a neighborhood of $0 \in K$. Since $\Z \subset K$ is dense, there are arbitrarily large positive integers $n \in \cU_k$. Since the measures $\al_k$ have finite support, we can choose $m_k \in \N$ large enough such that $m_k \in \cU_k$ and such that the translated measures defined by $\gamma_k(\cV) = \al_k(\cV-m_k)$ have their support in $\N$. Since $m_k \in \cU_k$, using \eqref{eq.already-nice}, we get that
\begin{equation}\label{eq.even-nicer}
|1 - \gammah_k(\om)| \leq 2k^{-3} \quad\text{for every $\om \in \cF_k$.}
\end{equation}
The tail boundary flow of $(\gamma_k)_{k \in \N}$ is still given by $\Z \actson K$.

Define the compound Poisson distributions $\beta_k = \cE(k \gamma_k)$ and denote by $\Z \actson X$ the tail boundary flow of $(\beta_k)_{k \in \N}$. We prove that $\Z \actson X$ is isomorphic with $\Z \actson K$.

Since $\betah_k(\om) = \exp(-k(1-\gammah_k(\om)))$, we get that $|\betah_k(\om)| = \exp(-k(1-\Re \gammah_k(\om)))$. From \eqref{eq.even-nicer} and the assumption that $\widehat{K} = \bigcup_{k=1}^\infty \cF_k$, it follows that
$$\prod_{k=1}^\infty |\betah_k(\om)| > 0 \quad\text{for all $\om \in \Kh$.}$$
Again by \cite[Theorem 4.2]{CW88}, it follows that each $\om \in \Kh$ is an eigenvalue of the tail boundary flow of $(\beta_k)_{k \in \N}$.

Since $\exp(-k)$ is summable, by Lemma \ref{lem.close-in-hellinger}, the term $\exp(-k)\delta_0$ in the definition of $\beta_k = \cE(k \gamma_k)$ is negligible, so that $\Z \actson X$ is isomorphic with the tail boundary flow of
$$
\rho_k = \gamma_k * \zeta_k \quad\text{where}\quad \zeta_k = (\exp(k)-1)^{-1} \Bigl(k \delta_0 + \sum_{j=1}^{\infty} \frac{k^{j+1}}{(j+1)!} \gamma_k^{*j}\Bigr) \; .
$$
Denoting by $\Z \actson Y$ the tail boundary flow of the sequence $(\zeta_k)_{k \in \N}$, we conclude that $\Z \actson X$ is the joint flow of $\Z \actson K$ and $\Z \actson Y$. Realizing this joint flow inside $L^\infty(K \times Y)$ with $\Z$ acting in the first variable, we conclude that the action $\Z \actson X$ can be continuously extended to an action of $K$. This means that $\Z \actson X$ is isomorphic with a factor of $\Z \actson K$. We have already seen that each $\om \in \Kh$ is an eigenvalue of $\Z \actson X$. So, $\Z \actson X$ must be isomorphic with $\Z \actson K$.
\end{proof}

We finally prove Theorem \ref{thm.Poisson-vs-ITPFI-2}. Here and in the rest of this section, we often make use of Prokhorov's theorem (see \cite{Pro53}) estimating the total variation distance between a binomial distribution and a Poisson distribution. Given $n \in \N$, $p \in [0,1]$ and $\lambda \geq 0$, define the binomial distribution $\beta$ and Poisson distribution $\pi$ by
\begin{align*}
& \beta_{n,p}(k) = \bigl(\begin{smallmatrix} n \\ k\end{smallmatrix}\bigr) \, p^k (1-p)^{n-k} \quad\text{for $k \in \{0,1,\ldots,n\}$, and}\\
& \pi_\lambda(k) = \exp(-\lambda) \frac{\lambda^k}{k!} \quad\text{for $k \in \{0,1,2,\ldots\}$.}
\end{align*}
In the version of \cite[Theorem 1]{BH83}, Prokhorov's theorem says that
\begin{equation}\label{eq.prokhorov}
\dTV(\beta_{n,p},\pi_{np}) \leq p \quad\text{for all $n \in \N$, $p \in [0,1]$.}
\end{equation}

\begin{proof}[Proof of Theorem \ref{thm.Poisson-vs-ITPFI-2}]
The trivial flow and every periodic flow arise as the flow of weights of an ITPFI$_2$ factor and also arise as a Poisson flow of positive type. We thus only consider the aperiodic case.

First assume that $\R \actson X$ is the flow of weights of an ITPFI$_2$ factor and that $\R \actson X$ is aperiodic. As we have seen in the proof of Theorem \ref{thm.ITPFI-bounded}, we find for all $k \in \N$, elements $b_k \in [k-1,k)$ and integers $L_k \geq 0$ such that $\R \actson X$ is isomorphic with the tail boundary flow of the sequence $\gamma(b_k)^{*L_k}$, where $\gamma(b)$ is defined by \eqref{eq.nice-notation}. By \eqref{eq.prokhorov}, we get that
$$\dTV\bigl(\gamma(b_k)^{*L_k} , \cE(\lambda_k \delta_{b_k})\bigr) \leq \exp(-b_k) \quad\text{where}\;\; \lambda_k = \frac{L_k \exp(-b_k)}{1+\exp(-b_k)} \; .$$
Since $\exp(-b_k) \leq \exp(-k+1)$ is summable, it follows from Lemma \ref{lem.close-in-hellinger} that $\R \actson X$ is a Poisson flow of positive type.

Conversely, assume that $\R \actson X$ is the Poisson flow of positive type defined by the sequence of probability measures $\mu_n = \cE(\eta_n)$, where each $\eta_n$ is a finite positive measure supported on $(0,+\infty)$. We may assume that $\R \actson X$ is aperiodic. Taking integers $L_n \geq \eta_n(\R)$, we can replace $\mu_n$ by $\cE(\eta_n / L_n)$, each repeated $L_n$ times. We may thus assume that $\lambda_n := \eta_n(\R) \leq 1$.

For every fixed $C > 0$, we have that
$$\beta_n := \exp(-\lambda_n) (\delta_0 + \eta_n|_{(0,C]}) \leq \mu_n$$
for all $n \in \N$ and that all $\beta_n$ are supported on $[0,C]$. It follows from Proposition \ref{prop.periodicity-criterion} that
\begin{equation}\label{eq.variance-estim-1}
\sum_{n=1}^\infty \beta_n(\R) \, \Var(\beta_n(\R)^{-1} \, \beta_n) < +\infty \; .
\end{equation}
Since $\exp(-\lambda_n) \geq \exp(-1)$ and $\eta_n((0,C]) \leq 1$ for all $n \in \N$, it follows from \eqref{eq.variance-estim-1} and \eqref{eq.elementary-variance} that for every $C > 0$,
\begin{equation}\label{eq.first-sum-finite}
\sum_{n=1}^\infty \int_{(0,C]} x^2 \, d\eta_n(x) < +\infty \; .
\end{equation}
Denote $\rho_n = \eta_n|_{(0,1]}$ and, for all $k \in \N$, $\eta_{n,k} = \eta_n|_{(k,k+1]}$. Since $\cE(\eta_n)$ is the convolution of the measures $\cE(\rho_n)$ and $\cE(\eta_{n,k})$, $n,k \in \N$, we conclude that $\R \actson X$ also is the tail boundary flow of the union of the probability measures $\cE(\rho_n)$, $n \in \N$, and $\cE(\eta_{n,k})$, $n,k \in \N$.

By \eqref{eq.first-sum-finite} and \eqref{eq.poisson-mean-variance}, we have $\sum_{n=1}^\infty \Var(\cE(\rho_n)) < +\infty$. By e.g.\ \cite[Proposition 2.1]{BV20}, the tail boundary flow of $(\cE(\rho_n))_{n \in \N}$ is given by the translation action $\R \actson \R$, so that $\R \actson X$ is the tail boundary flow of the measures $\cE(\eta_{n,k})$, $n,k \in \N$.

For every fixed $k \in \N$, it follows from \eqref{eq.first-sum-finite} that $\sum_{n =1}^\infty \eta_{n,k}(\R) < +\infty$. We define $\zeta_k = \sum_{n=1}^\infty \eta_{n,k}$ and conclude that $\R \actson X$ is the tail boundary flow of $\cE(\zeta_k)$, $k \in I$, where $\zeta_k$ is a finite positive measure supported on $(k,k+1]$.

Write $\ell_k = \zeta_k(\R)$ and fix integers $L_k \geq \ell_k$. Since $\R \actson X$ also is the tail boundary flow of the measures $\cE(\zeta_k/L_k)$, each repeated $L_k$ times, and since $\exp(-1) L_k^{-1} \zeta_k \leq \cE(\zeta_k/L_k)$, it follows from Proposition \ref{prop.periodicity-criterion} that the points
$$b_k = \ell_k^{-1} \int_\R x \, d\zeta_k(x) \in (k,k+1] \quad\text{satisfy}\quad \sum_{k=1}^\infty \int_\R (x-b_k)^2 \, d\zeta_k(x) < +\infty \; .$$
Denote by $W_2$ the Wasserstein $2$-distance. By \eqref{eq.wasserstein-estimate}, we have that
$$W_2(\cE(\zeta_k), \cE(\ell_k \delta_{b_k}))^2 \leq \exp(-\ell_k) \sum_{j=1}^\infty \frac{1}{j!} \int_\R (x - j b_k)^2 \, d\zeta_k^{* j}(x) \; .$$
For every $j \geq 1$, we have that
\begin{align*}
\int_\R (x - j b_k)^2 \, d\zeta_k^{* j}(x) &= \int_{\R^j} \bigl( (x_1 - b_k) + \cdots + (x_j - b_k) \bigr)^2 \, d\zeta_k(x_1) \, \cdots \, d\zeta_k(x_j)
\\ &= j \, \ell_k^{j-1} \, \int_\R (x-b_k)^2 \, d\zeta_k(x) \; .
\end{align*}
We conclude that
$$W_2(\cE(\zeta_k), \cE(\ell_k \delta_{b_k}))^2 \leq \int_\R (x-b_k)^2 \, d\zeta_k(x) \quad\text{so that}\quad \sum_{k=1}^\infty W_2(\cE(\zeta_k), \cE(\ell_k \delta_{b_k}))^2 < + \infty \; .$$
By Proposition \ref{prop.iso-wasserstein}, $\R \actson X$ is also the tail boundary flow of the sequence $\cE(\ell_k \delta_{b_k})$, $k \in \N$.
Write $\ell'_k = \exp(-b_k)(1+\exp(-b_k))^{-1}$. Take integers $M_k \geq 1$ such that $|M_k \ell'_k - \ell_k| \leq \ell'_k$.
For all $c,d > 0$ and $k \in \N$, one has
$$|c^k - d^k| = |c-d| \, \bigl|\sum_{i=0}^{k-1} c^i d^{k-1-i}\bigr| \leq k \, |c-d| \, (\max\{c,d\})^{k-1} \; ,$$
so that $\dTV(\cE(c \delta_b) , \cE(d \delta_b)) \leq |c-d|$ for all $c, d > 0$. Therefore,
$$\dTV\bigl(\cE(M_k \ell'_k \, \delta_{b_k}) , \cE(\ell_k \, \delta_{b_k})\bigr) \leq \exp(-b_k) \quad\text{for all $k \in \N$.}$$
By \eqref{eq.prokhorov}, we have
$$\dTV\bigl(\cE(M_k \ell'_k \, \delta_{b_k}) , \gamma(b_k)^{* M_k}\bigr) \leq \ell'_k \leq \exp(-b_k) \; .$$
Since $b_k > k$, the sequence $\exp(-b_k)$ is summable. It then follows from Lemma \ref{lem.close-in-hellinger} that $\R \actson X$ is the tail boundary flow of the measures $(\gamma(b_k)^{* M_k})_{k \in \N}$ and hence, is isomorphic with the flow of weights of an ITPFI$_2$ factor.
\end{proof}

The proof of Theorem \ref{thm.Poisson-vs-ITPFI-2} relied on Prokhorov's \eqref{eq.prokhorov}. Using the full force of \cite[Theorem 1]{BH83}, we can also prove the following result.

\begin{proposition}\label{prop.characterize-Poisson-flow-two-point-measures}
The Poisson flows are exactly the tail boundary flows of sequences $(\mu_n)_{n \in \N}$ where the support of all $\mu_n$ consists of two points and $\sup_{n \in \N} \Var \mu_n < +\infty$.

The Poisson flows of positive type are exactly the tail boundary flows of such sequences $(\mu_n)_{n \in \N}$ with support $a_n < b_n$ and $\mu_n(a_n) \geq \mu_n(b_n)$, and $\sup_{n \in \N} \Var \mu_n < +\infty$.
\end{proposition}

Some bound on the variances $\Var \mu_n$ must be imposed in order to get a Poisson flow. Combining \cite[Theorem 2.1]{GSW84} and Proposition \ref{prop.infinitely-divisible} below, it follows that the tail boundary flow of the sequence of probability measures $\bigl((\delta_0 + \delta_{8^n})/2\bigr)_{n \in \N}$ is not a Poisson flow.

\begin{proof}
First assume that $\R \actson Z$ is a Poisson flow. In most of the proof of Theorem \ref{thm.Poisson-vs-ITPFI-2}, we did not use that the measures are supported on the positive real line. Writing $I = \Z \setminus \{-1,0\}$, we find for every $k \in I$, elements $b_k \in (k,k+1]$ and constants $\lambda_k > 0$ such that $\R \actson Z$ is isomorphic with the tail boundary flow of the family $(\cE(\lambda_k \delta_{b_k}))_{k \in I}$. Choose for every $k \in I$ an integer $M_k \geq 1$ such that $M_k^{-1} \lambda_k \leq 2^{-|k|}$. Define for $k \in I$, the probability measure
$$\eta_k = (1- M_k^{-1} \lambda_k)\delta_0 + M_k^{-1} \lambda_k \delta_{b_k} \; .$$
By \eqref{eq.prokhorov}, we have $\dTV\bigl(\cE(\lambda_k \delta_{b_k}) , \eta_k^{*M_k}\bigr) \leq M_k^{-1} \lambda_k \leq 2^{-|k|}$, which is summable. By Lemma \ref{lem.close-in-hellinger}, $\R \actson Z$ is isomorphic with the tail boundary flow of the probability measures $\eta_k$ repeated $M_k$ times, $k \in I$. Since
$$\Var \eta_k \leq M_k^{-1} \lambda_k b_k^2 \leq 2^{-|k|} (|k|+1)^2 \to 0 \; ,$$
one implication of the proposition is proven.

Conversely, assume that $(\mu_n)_{n \in \N}$ is a sequence of probability measures whose support consists of two points and that satisfy $\Var \mu_n \leq C$ for all $n \in \N$. Denote by $\R \actson Z$ their tail boundary flow. We have to prove that $\R \actson Z$ is a Poisson flow. We may thus assume that $\R \actson Z$ is aperiodic. Since translating the measures $\mu_n$ does not change the tail boundary flow, we may assume that $\mu_n(0) \geq 1/2$ for all $n \in \N$ and we denote by $d_n \in \R$ the other atom of $\mu_n$. Write $p_n = \mu_n(d_n)$ Since $\Var \mu_n \leq C$, we get that $p_n d_n^2 \leq 2 C$ for all $n \in \N$.

For every $k \in \Z$, write $J_k = \{n \in \N \mid d_n \in [k,k+1)\}$. Write $I = \Z \setminus \{-1,0,1\}$. Arguing as in the proof of Theorem \ref{thm.Poisson-vs-ITPFI-2}, it follows from Proposition \ref{prop.periodicity-criterion} that
\begin{equation}\label{eq.to-ignore}
\sum_{n \in J_{-1} \cup J_0 \cup J_1} \Var(\mu_n) < + \infty \; ,
\end{equation}
and it follows from Proposition \ref{prop.further-concentration} that we find for every $k \in I$ an element $b_k \in [k,k+1)$ such that
\begin{equation}\label{eq.my-conc-here}
\sum_{k \in I} \sum_{n \in J_k} p_n (b_k - d_n)^2 < + \infty \; .
\end{equation}
By \eqref{eq.to-ignore}, the tail boundary flow of the measures $(\mu_n)_{n \in J_{-1} \cup J_0 \cup J_1}$ is given by the translation action $\R \actson \R$, so that these measures may be ignored. Defining for every $k \in I$ and $n \in J_k$ the probability measure $\eta_n = (1-p_n) \delta_0 + p_n \delta_{b_k}$, note that \eqref{eq.my-conc-here} is saying that
$$\sum_{k \in I} \sum_{n \in J_k} W_2(\mu_n,\eta_n)^2 <+\infty \; .$$
It thus follows from Proposition \ref{prop.iso-wasserstein} that $\R \actson Z$ is the tail boundary flow of the measures $\eta_n$, $k \in I$, $n \in J_k$. Write $\lambda_k = \sum_{n \in J_k} p_n$. By \cite[Theorem 1]{BH83}, we have for every $k \in I$ that
$$\dTV\Bigl( \bigconv_{n \in J_k} \eta_n \, , \, \cE(\lambda_k \delta_{b_k})\Bigr) \leq \lambda_k^{-1} \sum_{n \in J_k} p_n^2 \leq \max_{n \in J_k} p_n \leq \max_{n \in J_k} 2 C d_n^{-2} \leq \frac{2C}{(|k|-1)^2} \; ,$$
where we used that $d_n \in [k,k+1)$ if $n \in J_k$. Since the right hand side is summable, it follows from Lemma \ref{lem.close-in-hellinger} that $\R \actson Z$ is also the tail boundary flow of the Poisson distributions $(\cE(\lambda_k \delta_{b_k}))_{k \in I}$ and thus is a Poisson flow.

The positive case is proven entirely analogously.
\end{proof}

\subsection{Infinitely divisible flows}

\begin{definition}\label{def.infinitely-divisible}
We say that a flow $\R \actson (Z,\eta)$ is \emph{infinitely divisible} if for every integer $L \geq 1$, there exists a flow $\R \actson (Z_1,\eta_1)$ such that $\R \actson (Z,\eta)$ is isomorphic with the joint flow of $L$ copies of $\R \actson (Z_1,\eta_1)$.
\end{definition}

\begin{proposition}\label{prop.infinitely-divisible}
Every Poisson flow is infinitely divisible.

Every tail boundary flow of a sequence of infinitely divisible distributions is a Poisson flow.
\end{proposition}
\begin{proof}
If $\R \actson Z$ is the tail boundary flow of the sequence $(\cE(\mu_n))_{n \in \N}$, where $\mu_n$ is a sequence of finite positive measures on $\R$, and if $L \geq 1$ is an integer, we can define $\R \actson Z_1$ as the tail boundary flow of the sequence $(\cE(L^{-1} \mu_n))_{n \in \N}$. By construction, $\R \actson Z$ is isomorphic with the joint flow of $L$ copies of $\R \actson Z_1$.

Since compound Poisson distributions $\cE(\mu)$ are weakly dense in the set of infinitely divisible distributions (see e.g.\ \cite[Theorem 3.2.7]{Stro11}), the second statement follows directly from Lemma \ref{lem.density} below.
\end{proof}

\begin{lemma}\label{lem.density}
Let $\cF \subset \Prob(\R)$ with weak closure $\overline{\cF}$. Every ergodic flow that can be obtained as the tail boundary flow of a sequence in $\overline{\cF}$ can also be obtained as the tail boundary flow of a sequence in $\cF$.
\end{lemma}
\begin{proof}
Let $(\mu_n)_{n \in \N}$ be a sequence in $\overline{\cF}$. Denote by $\nu_n$ the uniform probability measure on the interval $[-1/n,1/n]$, so that $\Var \nu_n = n^{-2}/2$ is summable. By e.g.\ \cite[Proposition 2.1]{BV20}, the tail boundary flow of the sequence $(\nu_n)_{n \in \N}$ is the translation action $\R \actson \R$. Therefore, $(\mu_n)_{n \in \N}$ and $(\mu_n * \nu_n)_{n \in \N}$ give rise to isomorphic tail boundary flows.

Since $\cF$ is weakly dense in $\overline{\cF}$ and since $\nu_n$ is absolutely continuous, we can choose $\mu'_n \in \cF$ such that $\dTV\bigl(\mu'_n * \nu_n , \mu_n * \nu_n\bigr) \leq n^{-2}$. By Lemma \ref{lem.close-in-hellinger}, also $(\mu_n * \nu_n)_{n \in \N}$ and $(\mu'_n * \nu_n)_{n \in \N}$ give rise to isomorphic tail boundary flows, with the latter being isomorphic to the tail boundary flow of $(\mu'_n)_{n \in \N}$.
\end{proof}

Define for every $\lambda > 0$ and $a \in \R$ the standard Poisson distribution with support $\{k a \mid k=0,1,2,\ldots\}$ given by
\begin{equation}\label{eq:Poisson dist}
\si_{\lambda,a}(\{ka\}) = \exp(-\lambda) \, \frac{\lambda^k}{k!} \quad\text{for all $k \in \{0,1,2,\ldots\}$.}
\end{equation}

\begin{proposition}\label{prop.standard-poisson}
Every Poisson flow is the tail boundary flow of a sequence $(\si_{\lambda_n,a_n})_{n \in \N}$ with $\lambda_n > 0$ and $a_n \in \R$.
\end{proposition}
\begin{proof}
The convolution products of measures of the form $\si_{\lambda,a}$ are precisely the compound Poisson distributions $\cE(\mu)$ where $\mu$ is a finite positive measure with finite support. These $\cE(\mu)$ are weakly dense in the set of all compound Poisson distributions. The conclusion thus immediately follows from Lemma \ref{lem.density}.
\end{proof}

\begin{remark}\label{rem.concluding-flows}
Because of Proposition \ref{prop.infinitely-divisible}, it is tempting to speculate that every infinitely divisible flow is a Poisson flow, at least if the flow is assumed to be approximately transitive (and thus, the tail boundary flow of some sequence of probability measures, by \cite[Theorem 3.2]{CW88}). We have however no idea how to prove such a statement.

Going back to Theorem \ref{thm.Poisson-vs-ITPFI-2}, it is also unclear whether every Poisson flow is automatically of positive type. Because of Theorem \ref{thm.Poisson-vs-ITPFI-2}, this question is equivalent with the following seemingly innocent, but highly tantalizing problem: if $\R \actson^\al Z$ is the flow of weights of an ITPFI$_2$ factor, does it follow that the reverse flow $\be_t(z) = \al_{-t}(z)$ also is the flow of weights of an ITPFI$_2$ factor?

Combining both open problems, it is tempting to speculate that the ITPFI$_2$ factors are precisely the injective factors $M$ that are infinitely divisible, in the sense that for every integer $L \geq 1$, there exists a factor $N$ such that $M \cong N^{\ovt L}$, or to speculate that at least, infinite divisibility characterizes the ITPFI$_2$ factors among the ITPFI factors.
\end{remark}

\section{Nonsingular Bernoulli shifts: proof of Theorem \ref{thm.main-structure}}\label{sec:nonsingular Z}

The goal of this section is to prove the following more precise formulation of Theorem \ref{thm.main-structure}. Since we want to describe absolutely general Bernoulli shifts, the formulation becomes a bit lengthy, because we have to deal with the less interesting cases that may arise where the space has atoms or the action is dissipative.

We denote by $\cS$ the group of finite permutations of the countable set $\Z$ and let $\cS$ act on $(X,\mu) = \prod_{n \in \Z} (X_n,\mu_n)$ by $(\si \cdot x)_n = x_{\si^{-1}(n)}$.

\begin{theorem}\label{thm.main-structure-variant}
Let $X_0$ be any standard Borel space and let $(\mu_n)_{n \in \Z}$ be any family of equivalent probability measures on $X_0$ such that the Bernoulli shift $\Z \actson (X,\mu) = \prod_{n \in \Z} (X_0,\mu_n)$ is nonsingular.

Then precisely one of the following statements holds.

\begin{enumlist}
\item There exists an atom $b \in X_0$ with $\sum_{n \in \Z} (1-\mu_n(\{b\})) < +\infty$. Define $a \in X$ by $a_n = b$ for all $n \in \Z$. Then, $a$ is an atom in $X$ that is fixed by $\Z$. The action of $\Z$ on $X \setminus \{a\}$ is essentially free and dissipative.
\item The action $\Z \actson (X,\mu)$ is essentially free and dissipative.
\item The space $(X,\mu)$ is nonatomic and there exists a Borel set $C_0 \subset X_0$ of positive measure, unique up to a null set, such that $\sum_{n \in \Z} (1-\mu_n(C_0)) < + \infty$ and such that the following holds.
\begin{itemlist}
\item The action $\Z \actson C_0^\Z$ is a weakly mixing Bernoulli shift and its associated flow is infinitely divisible. Moreover, the permutation action $\cS \actson C_0^\Z$ is ergodic and for every ergodic pmp action $\Z \actson (Y,\nu)$, the actions $\Z \actson C_0^\Z$, $\Z \actson C_0^\Z \times Y$ and $\cS \actson C_0^\Z$ have the same associated flow.
\item The action $\Z \actson X \setminus C_0^\Z$ is essentially free and dissipative.
\end{itemlist}
\end{enumlist}
\end{theorem}

As already suggested by the formulation of Theorem \ref{thm.main-structure-variant}, we again exploit the relation between a Bernoulli shift and the action $\cS \actson (X,\mu)$. This method was discovered in \cite{Kos18,Dan18} and has been further developed in \cite{BKV19,BV20}.

Because of Theorem \ref{thm.main-structure-variant}, to study nonsingular Bernoulli shifts in their full generality, it suffices to consider Bernoulli shifts that are conservative. Also, given any conservative Bernoulli action $G\actson \prod_{g\in G}(C_0,\eta_g)$, we can add a standard Borel space $C_1$ to $C_0$ by putting $X_0=C_0\sqcup C_1$ and by choosing $\mu_g$ to be a family of equivalent measures on $X_0$ such that $\sum_{g\in G}\mu_g(C_1)<+\infty$ and such that $\mu_g|_{C_0}=\mu_g(C_0)\eta_g$. The result is again a nonsingular Bernoulli action $G\actson (X,\mu)=\prod_{g\in G}(X_0,\mu_g)$ with dissipative part $X\setminus C_0^{G}$ and conservative part $C_0^{G}$.

Note that if $\Z\curvearrowright (X,\mu)$ is a conservative ergodic nonsingular Bernoulli shift, it follows from Theorem \ref{thm.main-structure-variant} that it is has stable type in the strongest possible sense. Indeed, whenever $\Z\actson (Y,\nu)$ is an ergodic pmp action, the diagonal action $\Z\actson X\times Y$ is ergodic of the same type, and with the same associated flow as $\mathbb{Z}\curvearrowright (X,\mu)$. This rigidity property for the group $\Z$ is not shared by all countable infinite groups $G$, see \cite[Proposition 7.3]{VW17} and \cite[Remark 6.4]{BKV19}.

When we rephrase Theorem \ref{thm.main-structure-variant} in terms of von Neumann algebras we get the following result.

\begin{corollary}
Every crossed product $L^{\infty}(X)\rtimes \Z$ by a nonsingular Bernoulli shift $\Z\actson (X,\mu)$, decomposes as
\begin{align*}
L^{\infty}(X)\rtimes \Z= N\oplus M,
\end{align*}
where $N$ is a type I von Neumann algebra, and $M$ is an injective factor that is a $p$'th tensor power for every integer $p\geq 2$.
\end{corollary}

In addition to the structure result of Theorem \ref{thm.main-structure-variant}, we are able to give a complete type classification for conservative nonsingular Bernoulli shifts, distinguishing between types II$_{1}$, II$_{\infty}$ and III. The formulation of the result is similar to \cite[Theorem 4.1]{BV20}, but the important difference is that in our Theorem \ref{thm:type classification Z}, we make no assumptions on the behavior of the Radon-Nikodym derivatives $d\mu_n/d\mu_0$.

\begin{theorem}\label{thm:type classification Z}
Let $\mathbb{Z}\actson(X,\mu)=\prod_{n\in \mathbb{Z}}(X_0,\mu_n)$ be a conservative nonsingular Bernoulli shift. Then $\mathbb{Z}\actson (X,\mu)$ is weakly mixing and the following holds.
\begin{enumlist}
\item $\mathbb{Z}\curvearrowright (X,\mu)$ is of type II$_{1}$ if and only if there exists a probability measure $\nu\sim \mu_0$ on $X_0$ such that $\nu^{\mathbb{Z}}\sim\mu$.
\item $\mathbb{Z}\curvearrowright (X,\mu)$ is of type II$_{\infty}$ if and only if there exists a $\sigma$-finite measure $\nu\sim \mu_0$ on $X_0$ and Borel sets $\mathcal{U}_n\subset X_0$ such that $\nu(\mathcal{U}_n)<+\infty$ for all $n\in \mathbb{Z}$ and such that
\begin{align*}
\sum_{n\in\Z}\mu_n(X_0\setminus \mathcal{U}_n)<+\infty,\;\; \sum_{n\in \Z}H^{2}\left(\mu_n, \nu(\mathcal{U}_n)^{-1}\nu|_{\mathcal{U}_n}\right)<+\infty,\;\;\sum_{n\in \Z}\nu(X_0\setminus \mathcal{U}_n)=+\infty.
\end{align*}
\item $\mathbb{Z}\curvearrowright (X,\mu)$ is of type III in all other cases.
\end{enumlist}
\end{theorem}

A key point in relating the Bernoulli shift $\Z \actson (X,\mu)$ to the permutation action $\cS \actson (X,\mu)$ is the following lemma. The proof uses a key idea of \cite[Theorem 3.3]{BKV19}.

\begin{lemma}\label{lem:Z is in S}
Let $\Z\curvearrowright (X,\mu)=\prod_{n\in \mathbb{Z}}(X_0,\mu_n)$ be a nonsingular Bernoulli shift that is essentially free and not dissipative. When viewing $\mathbb{Z}$ and $\mathcal{S}$ as subgroups of $\Aut(X,\mu)$ we have that $\Z$ belongs to the closure of $\mathcal{S}$.
\end{lemma}

\begin{proof}
We start by proving the following claim.

\textbf{Claim.} There exist sequences $n_k\rightarrow -\infty, m_k\rightarrow+\infty$ such that
\begin{align}\label{eq:subsequence Hellinger distance}
\lim_{k\rightarrow \infty} H\bigl(\mu_{n_k},\mu_{m_k}\bigr)=0.
\end{align}

We prove this claim using essentially the same argument as the one given in \cite[Theorem 3.3]{BKV19}. Indeed if such sequences do not exist, there are $\delta>0$ and $N\in \mathbb{N}$ such that for all $n\leq -N, m\geq N$ we have that $H(\mu_n,\mu_m)>\delta$. Define the $1$-cocycle
\begin{align*}
c\colon \mathbb{Z}\rightarrow \ell^{2}(\mathbb{Z})\otimes L^{2}(X_0,\mu_0) : c_n(k,\cdot)=\sqrt{d\mu_{k}/d\mu_0} -\sqrt{d\mu_{k-n}/d\mu_{0}} \; ,
\end{align*}
as introduced in \cite[Theorem 3.1]{VW17}. We have that
\begin{align*}
\|c_k\|^2=\|c_{-k}\|^2=2 \sum_{m\in \mathbb{Z}}H^{2}(\mu_{m+k},\mu_m)\geq 2 \sum_{m=N-k}^{-N}H^{2}(\mu_{m+k},\mu_m)\geq 2(k-2N)\delta^2
\end{align*}
for every $k\geq 2N$. Therefore $\sum_{k\in \mathbb{Z}}\exp\left(-\|c_k\|^2/2\right)<+\infty$ and it follows from \cite[Theorem 4.1]{VW17} that the action $\mathbb{Z}\actson (X,\mu)$ is dissipative, which is in contradiction with our assumptions.

Let $n_k\rightarrow-\infty$ and $m_k\rightarrow +\infty$ be sequences such that \eqref{eq:subsequence Hellinger distance} holds. We may assume that $n_k<0<m_k$ for all $k\in \mathbb{N}$. Let $\alpha\in \Aut(X,\mu)$ denote the shift by one, i.e. $(\alpha(x))_k=x_{k-1}$. To prove the lemma, it suffices to show that $\alpha$ belongs to the closure of $\mathcal{S}$.

For each $k\in \N$ we define the permutation $\sigma_k\in \mathcal{S}$ by
\begin{align*}
\sigma_k(n)=\begin{cases}n &\;\;\text{if}\;\; n\leq n_k \;\;\text{or}\;\; n\geq 1+m_k\\
n-1 &\;\;\text{if}\;\; 2+n_k\leq n\leq m_k\\
m_k &\;\;\text{if}\;\; n=1+n_k\end{cases}
\end{align*}
and write $\beta_k=\sigma_k\circ \alpha\in \Aut(X,\mu)$. Define the unitary operators $\theta_k\colon L^2(X,\mu)\rightarrow L^{2}(X,\mu)$ by
\begin{align*}
\theta_k(F)(x)=\sqrt{\frac{d(\mu\circ \beta_k)}{d\mu}(x)}F(\beta_k(x)).
\end{align*}
We will show that $\|\theta_k(F)-F\|_2\rightarrow 0$ for all $F$ in a total subset of $L^{2}(X,\mu)$, so that $\beta_k\rightarrow
\id$ as $k\rightarrow +\infty$. This then concludes the proof of the lemma.

Take $F\in L^{\infty}(X,\mu)$ depending only on the coordinates $x_n$, for $|n|\leq N$, for some $N\in \mathbb{N}$. With unconditional convergence almost everywhere we have that
\begin{align*}
\frac{d(\mu\circ \beta_k)}{d\mu}(x)=\prod_{n=-\infty}^{n_k-1}\frac{d\mu_{n+1}}{d\mu_n}(x_n)\cdot \frac{d\mu_{m_k}}{d\mu_{n_k}}(x_{n_k}) \cdot \prod_{n=m_k}^\infty \frac{d\mu_{n+1}}{d\mu_n}(x_n).
\end{align*}
Therefore
\begin{align*}
\int_{X}\sqrt{d(\mu\circ \beta_k) / d\mu} \, d\mu= \bigl(1-H^{2}(\mu_{n_k},\mu_{m_k})\bigr) \, \prod_{\substack{n\leq n_k \;\;\text{or}\;\; n\geq 1+m_k}}\left(1-H^2(\mu_{n-1},\mu_n)\right)\rightarrow 1
\end{align*}
as $H^2(\mu_{n-1},\mu_n)$ is summable and by our choice of $n_k,m_k$. We see that for $n_k<-N$ and $m_k>N$ we have that
\begin{align*}
\|\theta_k(F)-F\|_2&\leq \|F\|_\infty \, \bigl\|1-\sqrt{d(\mu\circ \beta_k) / d\mu}\bigr\|_2,
\end{align*}
which converges to $0$ as $k$ tends to infinity.
\end{proof}

\begin{lemma}\label{lem:S is in Z}
Let $\Z\actson (X,\mu)=\prod_{n\in \mathbb{Z}}(X_0,\mu_n)$ be a nonsingular Bernoulli shift that is essentially free and not dissipative. Let $C\subset X$ denote its conservative part. Let $\mathbb{Z}\curvearrowright (Y,\nu)$ be any ergodic pmp action and consider the diagonal product action $\mathbb{Z}\actson X\times Y$. Then the following holds.
\begin{enumlist}
\item $C$ is $\mathcal{S}$-invariant.
\item The Maharam extensions satisfy
\begin{align*}
L^{\infty}(C\times \R\times Y)^{\Z}= L^{\infty}(C\times \R)^{\cS}\ovt 1=L^{\infty}(C\times \R)^{\Z}\ovt 1.
\end{align*}
\end{enumlist}
\end{lemma}

\begin{proof}
Let $\alpha_k\in \Aut(X,\mu)$ denote the translation by $k$, i.e. $(k\cdot x)_m=x_{m-k}$. We proceed as in the proof of \cite[Lemma 3.1]{BKV19} and show that the dissipative part $D\subset X$ given by
\begin{align*}
D=\Bigl\{x\in X \Bigm| \sum_{k\in \Z}\frac{d(\mu\circ \alpha_k)}{d\mu}(x)<+\infty\Bigr\}
\end{align*}
is $\mathcal{S}$-invariant. It suffices to show that $D$ is invariant under the permutation $\sigma_n\in \mathcal{S}$ that interchanges the coordinate $0$ and $n$.
Fix $n\in \mathbb{Z}\setminus \{0\}$. For each $\eta>0$ and each $k\in \mathbb{Z}$ we define
\begin{equation}\label{eq:definition A and B}
\begin{split}
& A^{\eta}_k= \bigl\{x\in X_0 \bigm| \; \bigl|\sqrt{d\mu_{n+k}/d\mu_k}(x)-1\bigr| \geq \eta \bigr\} \;\; ,\\
& B^{\eta}_k= \bigl\{x\in X_0 \bigm| \; \bigl|\sqrt{d\mu_k / d\mu_{n+k}}(x)-1\bigr|\geq \eta \bigr\} \; .
\end{split}
\end{equation}
For every $\eta>0$ we have that
\begin{align*}
\sum_{k\in \mathbb{Z}}\mu_k(A^{\eta}_k) &= \sum_{k\in \mathbb{Z}}\int_{A^{\eta}_k}\dfrac{d\mu_k}{d\mu_0}d\mu_0 \leq \eta^{-2}\sum_{k\in \mathbb{Z}}\int_{A^{\eta}_k}\dfrac{d\mu_k}{d\mu_0} \, \bigl|\sqrt{d\mu_{n+k}/d\mu_k}-1\bigr|^{2}d\mu_0\\
&\leq 2 \eta^{-2}\sum_{k\in \mathbb{Z}}H^{2}(\mu_{n+k},\mu_k)<+\infty.
\end{align*}
This means that the set
$$A_\eta := \{x \in X \mid x_k \in X_0 \setminus A^\eta_k \;\;\text{for all $k \in \Z$}\;\}$$
has positive measure, $\mu(A_\eta) > 0$. From the nonsingularity of $\mathbb{Z}\actson (X,\mu)$ it follows that $\mu(\al_m(A_\eta)) > 0$ for every $m \in \Z$. Since $x \in \al_m(A_\eta)$ if and only if $x_k \in X_0 \setminus A^\eta_{k-m}$ for all $k \in \Z$, we conclude that $\sum_{k\in \Z}\mu_{k+m}(A^{\eta}_k)<+\infty$ for every $m\in\mathbb{Z}$.

Similarly we have that $\sum_{k\in \Z}\mu_{k+m}(B^{\eta}_k)<+\infty$ for every $\eta>0$ and every $m\in\mathbb{Z}$. Write $C^\eta_k=A^\eta_k\cup B^\eta_k$. For $x\in X_0$, we define $W_x^{\eta}=\{k\in \mathbb{Z} \mid x\in C_k^\eta\}$. For $m\in \mathbb{Z}$, denote $\pi_m\colon X\rightarrow X_0$ for the coordinate projection $\pi_m(x)=x_m$. For any $m\in \mathbb{Z}$ we have that
\begin{align}\label{eq:integrable sum}
\begin{aligned}
\int_{X}\sum_{k\in W_{x_m}^\eta}\frac{d(\mu\circ \alpha_k)}{d\mu}d\mu&=\int_{X}\sum_{k\in \mathbb{Z}}1_{C_k^\eta}\circ \pi_m\frac{d(\mu\circ \alpha_k)}{d\mu}d\mu
=\sum_{k\in \mathbb{Z}}\int_X 1_{C^\eta_k}\circ \pi_{m+k} d\mu\\
&=\sum_{k\in \mathbb{Z}}\mu_{k+m}\left(C_{k}^\eta\right)<+\infty.
\end{aligned}
\end{align}
Therefore we have that
\begin{align}\label{eq:convergence on subsets W}
\sum_{k\in W_{x_0}^1\cup W_{x_n}^1}\frac{d(\mu\circ\alpha_k)}{d\mu}(x)<+\infty \;\;\text{for a.e.  $x\in X$}.
\end{align}
Borrowing some notation form \cite{BKV19}, we write
\begin{align*}
D_{n,m}(x,y)=\frac{d\mu_n}{d\mu_m}(y) \, \frac{d\mu_m}{d\mu_n}(x) \; \; ,\;\;\text{for $n,m\in \mathbb{Z}$ and $x,y\in X_0$.}
\end{align*}
Note that
\begin{align}\label{eq:permutation translation mixture}
\frac{d(\mu\circ \alpha_k)}{d\mu}(\sigma_n(x))=\frac{d(\mu\circ \alpha_k)}{d\mu}(x) \, D_{k,k+n}(x_0,x_n) \, D_{n,0}(x_0,x_n).
\end{align}
By the definition of $W_{x_i}^{1}$ and the sets $A_k^{\eta}$ and $B_k^{\eta}$ in \eqref{eq:definition A and B}, we have that $1/16\leq D_{k,k+n}(x_0,x_n)\leq 16$ whenever $k\notin W_{x_0}^1\cup W_{x_n}^1$. So it follows from \eqref{eq:convergence on subsets W} and \eqref{eq:permutation translation mixture} that $\mathcal{D}$ is invariant under $\sigma_n$.

We now prove point 2. Let $G\actson (Y,\nu)$ be an ergodic pmp action. One can repeat the proof of \cite[Lemma 3.1]{BKV19}, making use of the ergodic theorem established in \cite[Theorem A.1]{Dan18}, to conclude that
\begin{equation}\label{eq.first-inclu}
L^{\infty}(C\times \R\times Y)^{\Z}\subset L^{\infty}(C\times \R)^{\cS}\ovt L^{\infty}(Y) \; .
\end{equation}
Write $\lambda$ for the Lebesgue measure on $\R$. As the Maharam extension map
\begin{align*}
\Aut(X, \mu) \rightarrow \Aut(X\times \R,\mu\times \lambda) : \vphi \mapsto \widetilde{\vphi}
\end{align*}
is continuous, it follows from the $\cS$-invariance of $C$ and Lemma \ref{lem:Z is in S} that
\begin{align*}
L^{\infty}(C\times \R)^{\cS}\subset L^{\infty}(C\times \R)^{\Z} \; .
\end{align*}
In combination with \eqref{eq.first-inclu}, we get that
\begin{align*}
L^{\infty}(C\times \R\times Y)^{\Z}\subset L^{\infty}(C\times \R)^{\Z}\ovt L^{\infty}(Y) \; .
\end{align*}
As $\Z\actson (Y,\nu)$ is ergodic, the $\Z$-invariant elements of $L^{\infty}(C\times \R)^\Z \ovt L^{\infty}(Y)$ are contained in $L^{\infty}(C\times \R)^{\Z}\ovt 1$. The converse inclusion $L^{\infty}(C\times \R)^{\Z}\ovt 1\subset L^{\infty}(C\times \R\times Y)^{\Z}$ holds trivially, proving the second statement of the Lemma.
\end{proof}

\begin{remark}
Except for the point where we invoke Lemma \ref{lem:Z is in S}, the proof of Lemma \ref{lem:S is in Z} remains valid for a nonsingular Bernoulli action $G\curvearrowright \prod_{g\in G}(X_0,\mu_g)$ of any countable infinite amenable group $G$, as long as also the right Bernoulli action is nonsingular, e.g.\ when $G$ is abelian. However, we were unable to prove an analogue of Lemma \ref{lem:Z is in S} for arbitrary abelian groups. That is the main reason why this section is restricted to the group of integers. Note that it is nevertheless straightforward to generalize our results from $\Z$ to virtually cyclic abelian groups.
\end{remark}

\begin{proof}[Proof of Theorem \ref{thm.main-structure-variant}]
First assume that $(X,\mu)$ admits an atom $d \in X$. Then $d_n \in X_0$ is an atom for every $n \in \Z$ and $\sum_{n \in \Z} (1-\mu_n(\{d_n\})) < +\infty$. Writing
$$\cU = \{x \in X \mid x_n = d_n \;\;\text{for all but finitely many $n \in \Z$}\;\} \; ,$$
it follows that $\mu(X \setminus \cU) = 0$. Since the shift is nonsingular, the set $1 \cdot \cU \cap \cU$ has measure $1$. We thus find $N \in \N$ such that $d_{n-1} = d_n$ for all $|n| \geq N$.

There are now two possibilities. Either we find an atom $b \in X_0$ such that $d_n = b$ for all $|n| \geq N$, or we find two distinct atoms $b,c \in X_0$ such that $d_n = b$ for all $n \geq N$ and $d_n = c$ for all $n \leq -N$.

In the first case, we get that $\sum_{n \in \Z}(1-\mu_n(\{b\})) < + \infty$ and we define the atom $a \in X$ by $a_n = b$ for all $n \in \Z$. Clearly, $g \cdot a = a$ for all $g \in \Z$. We define for every $k \in \N$, the Borel set $\cW_k = \{x \in X \mid x_n = b \;\;\text{whenever $|n| \geq k$}\;\}$. We have $\bigcup_{k \in \N} \cW_k = \cU$ so that this set has a complement of measure zero. Also, $g \cdot (\cW_k \setminus \{a\}) \cap \cW_k = \emptyset$ whenever $g \in \Z$ and $|g| > 2k$. So, for every $k \in \N$, the set $\cW_k \setminus \{a\}$ belongs to the dissipative part of the essentially free action $\Z \actson X$. Taking the union over $k$, we conclude that $\Z \actson X \setminus \{a\}$ is essentially free and dissipative.

In the second case, we define for every $k \in \N$, the Borel set
$$\cV_k = \{x \in X \mid x_n = c \;\;\text{if $n \leq -k$, and}\;\; x_n = b \;\;\text{if $n \geq k$}\;\} \; .$$
Again $\bigcup_{k \in \N} \cV_k = \cU$ has a complement of measure zero and $g \cdot \cV_k \cap \cV_k = \emptyset$ whenever $g \in \Z$ and $|g| > 2k+1$. So, $\Z \actson X$ is essentially free and dissipative.

For the rest of the proof, we may thus assume that $(X,\mu)$ is nonatomic. Then the Bernoulli shift $\Z\actson (X,\mu)$ is essentially free, by \cite[Lemma 2.2]{BKV19}. If $\Z \actson (X,\mu)$ is dissipative, the conclusion of point~2 holds. It now remains to consider the case where the conservative part $C \subset X$ of $\Z \actson (X,\mu)$ has positive measure. We have to prove the structural result in point~3 of the theorem.

Note that $C$ is $\mathbb{Z}$-invariant. By Lemma \ref{lem:S is in Z}, $C$ is also $\mathcal{S}$-invariant. We claim that for any integer $p\geq 2$ we have that
\begin{equation}\label{eq: Z is pZ}
L^{\infty}(C)^{p\Z}=L^{\infty}(C)^{\Z}.
\end{equation}
To prove this claim, fix $p\geq 2$ and define $Y=\Z/p\Z$, equipped with the normalized counting measure $\nu$. We let $\mathbb{Z}$ act on $Y$ by translation. From Lemma \ref{lem:S is in Z} we know that $L^{\infty}(C\times Y)^{\mathbb{Z}}=L^{\infty}(C)^{\mathbb{Z}}\ovt 1$, and this is exactly the statement \eqref{eq: Z is pZ}.

For any integer $p\geq 2$ and $i\in \{0,1\dots, p-1\}$, we write $(Z_{p,i},\nu_{p,i})=\prod_{n \in i+p\mathbb{Z}}(X_0,\mu_n)$. We identify
\begin{align}\label{eq:product identification}
(X,\mu)=\prod_{i=0}^{p-1}(Z_{p,i},\nu_{p,i})
\end{align}
and we obtain measure preserving factor maps $\pi_{p,i}\colon X\rightarrow Z_{p,i}$. For each $i\in \{0,1,\dots,p-1\}$, we have a Bernoulli action $p\mathbb{Z}\curvearrowright Z_{p,i}$ and the factor maps $\pi_{p,i}$ are $p\mathbb{Z}$-equivariant. Let $\mathcal{S}_{p,i}$ denote the group of finite permutations of $i+p\Z$. We also have a nonsingular action $\mathcal{S}_{p,i}\actson Z_{p,i}$ and $\pi_{p,i}$ is $\mathcal{S}_{p,i}$-equivariant as well.

For $i\in \{0,1,\dots,p-1\}$, write $\alpha_i\in \Aut(X,\mu)$ for the shift by $i$. There is a natural nonsingular isomorphism $\theta_{p,i}\colon Z_{p,0}\rightarrow Z_{p,i}$ such that $\theta_{p,i}\circ \pi_{p,0}= \pi_{p,i}\circ \alpha_i$.

We start by using \eqref{eq: Z is pZ} for $p=2$ to show that $L^{\infty}(C)^{\mathbb{Z}}$ is discrete as a von Neumann algebra. To simplify the notation, we will drop the index $p$ for this special case $p=2$.

Let $E_0\subset Z_0$ be a Borel set such that $(\pi_0)_*\mu|_C\sim\nu_0|_{E_0}$. Then $E_0$ is uniquely determined up to a null set. As $C$ is $\Z$-invariant and $\pi_0$ is a $2\Z$-equivariant factor map, $E_0$ is $2\mathbb{Z}$-invariant. Similarly $E_0$ is $\mathcal{S}_0$-invariant. Since $2\mathbb{Z}\subset \Z$ has finite index, the action $2\Z\curvearrowright C$ is conservative. Therefore also $2\mathbb{Z}\actson E_0$ is conservative and it follows that $E_0$ is contained in the conservative part of the Bernoulli action $2\Z\actson Z_0$. By Lemma \ref{lem:S is in Z} we have that
\begin{align*}
L^{\infty}(E_0)^{2\mathbb{Z}}=L^{\infty}(E_0)^{\mathcal{S}_0}.
\end{align*}
Put $E_1=\theta_1(E_0)\subset Z_1$. As $\theta_1\circ \pi_0 =\pi_1\circ \alpha_1$, we have that $(\pi_1)_*\mu|_C\sim \nu_1|_{E_1}$. By the equivariance of $\theta_1$ we also have that $E_1$ is $2\Z$- and $\mathcal{S}_1$-invariant, and that
\begin{align*}
L^{\infty}(E_1)^{2\mathbb{Z}}=L^{\infty}(E_1)^{\mathcal{S}_1}.
\end{align*}
Let $F\in L^{\infty}(E_0)^{2\Z}$ be arbitrary. Then $F\circ \pi_0\in L^{\infty}(C)^{2\mathbb{Z}}$ and we apply \eqref{eq: Z is pZ} to conclude that $F\circ \pi_0$ is $\mathbb{Z}$-invariant. Using that $\theta_1\circ\pi_0=\pi_1\circ \alpha_1$ and viewing $X=Z_0\times Z_1$, we can also express this as
\begin{equation}\label{eq:forces discreteness}
F\otimes 1= 1\otimes (F\circ \theta_1^{-1}) \;\;\text{a.e. on $C$}.
\end{equation}
The equality \eqref{eq:forces discreteness} holds for any $F\in L^{\infty}(E_0)^{2\mathbb{Z}}$, forcing $L^{\infty}(E_0)^{2\mathbb{Z}}$ to be discrete. Similarly, we see that $L^{\infty}(E_1)^{2\Z}$ is discrete as well.

Using once more the identification $X=Z_0\times Z_1$, we have that $C\subset E_0\times E_1$. Therefore, by Lemma \ref{lem:S is in Z}, we have that
\begin{align*}
&L^{\infty}(C)^{\mathbb{Z}}=L^{\infty}(C)^{\mathcal{S}}\subset L^{\infty}(C)^{\mathcal{S}_0\times \mathcal{S}_1}\\
&=1_C\bigl(L^{\infty}(E_0)^{\mathcal{S}_0}\ovt L^{\infty}(E_1)^{\mathcal{S}_1} \bigr)=1_C\bigl(L^{\infty}(E_0)^{2\Z}\ovt L^{\infty}(E_1)^{2\Z} \bigr) \; ,
\end{align*}
so that also $L^{\infty}(C)^{\mathbb{Z}}$ is discrete.

Take a $\Z$-invariant Borel set $\mathcal{U}\subset X$ with $\mu(\cU)>0$ such that $1_\cU$ is a minimal projection in $L^{\infty}(C)^{\Z}$. So, $\Z\actson \cU$ is ergodic. We prove that $\mathcal{U}$ is of the form $\mathcal{U}=C_0^{\Z}$ for some $C_0\subset X_0$.

For any integer $p\geq 2$ and $i\in\{0,1,\dots, p-1\}$ define the Borel set $\mathcal{U}_{p,i}$ by $(\pi_{p,i})_*\mu|_{\mathcal{U}}\sim (\nu_{p,i})|_{\mathcal{U}_{p,i}}$. First of all, note that $\mathcal{U}_{p,i}$ is $p\Z$-invariant and $\mathcal{S}_{p,i}$-invariant. By \eqref{eq: Z is pZ} the action $p\Z\curvearrowright \mathcal{U}$ is ergodic, so that also $p\mathbb{Z}\actson \mathcal{U}_{p,i}$ is ergodic. Since we can view $p\mathbb{Z}\curvearrowright Z_{p,i}$ as a Bernoulli action, it follows from Lemma \ref{lem:S is in Z} that $\mathcal{S}_{p,i}\actson \mathcal{U}_{p,i}$ is ergodic as well.

Using the identification \eqref{eq:product identification}, we have that $\mathcal{U}\subset \mathcal{U}_{p,0}\times \mathcal{U}_{p,1}\times \dots\times \mathcal{U}_{p,p-1}$ for any $p\geq 2$. As $\mathcal{U}$ is invariant under the subgroup $\mathcal{S}_{p,0}\times \mathcal{S}_{p,1}\times \dots \mathcal{S}_{p,p-1}$ and as $\mathcal{S}_{p,i}\actson \mathcal{U}_{p,i}$ acts ergodically for each $i\in \{0,1,\dots,p-1\}$, we have that
\begin{align}\label{eq:p product decomposition}
\mathcal{U}=\mathcal{U}_{p,0}\times \mathcal{U}_{p,1}\times \dots \mathcal{U}_{p,p-1} \mod \mu,\;\; \text{for every $p\geq 2$}.
\end{align}
Let $n\in \mathbb{N}$ and let $A_n\subset L^{\infty}(X)$ denote the subalgebra of elements only depending on the variables $x_j$, for $-n\leq j\leq n$, and let $E_n\colon L^{\infty}(X)\rightarrow A_n$ be the unique conditional expectation preserving the measure $\mu$. For any $n\in \mathbb{N}$ we apply the decomposition \eqref{eq:p product decomposition} to $p=2n+1$. Since the numbers $\{j \mid -n\leq j\leq n\}$ are distinct representatives of the elements of $\Z/(2n+1)\Z$, there exist $a_{n,j}\in L^{\infty}(X_0,\mu_j)$ such that
\begin{align*}
E_n(1_{\mathcal{U}})=a_{n,-n}\otimes a_{n,-n+1}\otimes \dots \otimes a_{n,n} \;\;\text{and}\;\; 0\leq a_{n,j}\leq 1 \;\;\text{a.e. for every $-n\leq j\leq n$.}
\end{align*}
Expressing that $E_n\circ E_m=E_n$ for $m\geq n$ yields
\begin{equation}\label{eq:conditional decomposition}
E_n(1_{\mathcal{U}})=a_{n,-n}\otimes a_{n,-n+1}\otimes \dots\otimes a_{n,n}=a_{m,-n}\otimes a_{m,-n+1}\otimes \dots\otimes a_{m,n}\cdot \prod_{n<|j|\leq m}\mu_{j}(a_{m,j}) \; .
\end{equation}
For each $j\in \mathbb{Z}$, letting $m\rightarrow +\infty$, we have that $a_{m,j}$ is a sequence in $L^{\infty}(X_0,\mu_j)$ such that $0\leq a_{m,j}\leq 1$ for all $m$. Similarly, for a fixed $n\in \mathbb{N}$, we have that
\begin{align*}
\prod_{n<|j|\leq m}\mu_{j}(a_{m,j})\in [0,1] \;\;\text{for every $m\geq n$}.
\end{align*}
Thus we can choose a subsequence $m_k\rightarrow +\infty$ such that $a_{m_k,j}\rightarrow b_j$
weakly for every $j\in \mathbb{Z}$ and $b_j\in L^{\infty}(X_0,\mu_j)$ satisfying $0\leq b_j\leq 1$, and such that $\prod_{n<|j|\leq m_k}\mu_{j}(a_{m_k,j})\rightarrow \lambda_n$ for every $n\in \mathbb{N}$ for some $\lambda_n\in [0,1]$. The equality \eqref{eq:conditional decomposition} implies that
\begin{align*}
E_n(1_{\mathcal{U}})=\lambda_n b_{-n}\otimes b_{-n+1}\otimes \dots \otimes b_{n}\;\;\text{for every $n\in \mathbb{N}$}.
\end{align*}
As $E_n(1_{\mathcal{U}})$ is nonzero for every $n$, we see that $\lambda_n$ and $b_j$ are nonzero for every $n\in\mathbb{N}$ and $j\in \mathbb{Z}$. Expressing once more that $E_n\circ E_m = E_n$ for $m\geq n$, we obtain that
\begin{align*}
\lambda_n=\lambda_m \prod_{n<|j|\leq m}\mu_{j}(b_j)\leq \prod_{n<|j|\leq m}\mu_{j}(b_j) \; .
\end{align*}
which shows that the infinite product $\prod_{|j|>n}\mu_{j}(b_j)$ converges to a nonzero limit for each $n\in \mathbb{N}$. Let $\lambda\in [0,1]$ be any limit point of the sequence $\lambda_n$. Using that $E_n(1_{\mathcal{U}})\rightarrow 1_{\mathcal{U}}$ strongly as $n\rightarrow +\infty$, we see that the infinite product of the $b_j$ converges and that we have an equality
\begin{align*}
1_{\mathcal{U}}=\lambda\bigotimes_{j\in \Z}b_j \; .
\end{align*}
Together with the fact that $\mathcal{U}$ is $\Z$-invariant, this implies that there is a Borel set $C_0\subset X_0$ such that $\mathcal{U}=C_0^{\mathbb{Z}}$. Write $C_1=X_0\setminus C_0$. As $\mu(\mathcal{U})>0$ we have that
\begin{align*}
\sum_{n\in \Z}\mu_{n}(C_1)<+\infty \; .
\end{align*}
By construction, the action $\Z\curvearrowright X\setminus C_0^{\Z}$ is dissipative. It follows that $C=C_0^{\mathbb{Z}}=\mathcal{U}$. We have chosen $\mathcal{U}$ such that $\Z\actson \mathcal{U}$ is ergodic, thus it follows from Lemma \ref{lem:S is in Z} that $\cS\actson C_0^{\Z}$ is ergodic, that $\Z\actson C_0^{\Z}$ is weakly mixing, and that for any ergodic pmp action $\Z\actson (Y,\nu)$ the nonsingular actions $\Z\actson C_0^{\Z}$, $\cS\actson C_0^{\Z}$ and $\Z\actson C_0^{\Z}\times Y$ have the same associated flow. It remains to prove that this flow is infinitely divisible. For this remaining part of the proof, we may replace $X_0$ by $C_0$ and thus assume that $C_0 = X_0$.

To prove this, let $p\geq 1$ be an integer. We use the notation introduced in \eqref{eq:product identification}. Let $\mathcal{S}_{p,i}\subset \mathcal{S}$ denote the subgroup of finite permutations of $i+p\Z$. We have that
\begin{align}\label{eq:inclusion chain 1}
\begin{aligned}
L^{\infty}(X\times \R)^{\Z}=L^{\infty}(X\times \R)^{\mathcal{S}}\subset L^{\infty}(Z_{p,0}\times \dots \times Z_{p,p-1}\times \mathbb{R})^{\mathcal{S}_{p,0}\times\dots \times\mathcal{S}_{p,p-1}}.
\end{aligned}
\end{align}
For each $i\in \{0,1,\dots ,p-1\}$, we write $\Gamma_{p,i}$ for the group $\Gamma_{p,i}=p\Z$, acting naturally on $Z_{p,i}$. We can view the action $\Gamma_{p,i}\curvearrowright Z_{p,i}$ as a nonsingular Bernoulli action, which is a factor of the conservative nonsingular Bernoulli action $p\Z\curvearrowright (X,\mu)$. Therefore $\Gamma_{p,i}\curvearrowright Z_{p,i}$ is conservative and by Lemma \ref{lem:Z is in S}, we have that $L^{\infty}(Z_{p,i}\times \mathbb{R})^{\mathcal{S}_{p,i}}\subset L^{\infty}(Z_{p,i}\times \mathbb{R})^{\Gamma_{p,i}}$, for each $i\in \{0,1,\dots ,p-1\}$, so that
\begin{align}\label{eq:inclusion chain 2}
\begin{aligned}
L^{\infty}(Z_{p,0}\times  \dots \times Z_{p,p-1}\times \mathbb{R})^{\mathcal{S}_{p,0}\times \dots \times \mathcal{S}_{p,p-1}}\subset L^{\infty}(Z_{p,0}\times \dots \times Z_{p,p-1}\times \mathbb{R})^{\Gamma_{p,0}\times \dots \times\Gamma_{p,p-1}}.
\end{aligned}
\end{align}
Each $\Gamma_{p,i}$ is a copy of $p\Z$ and the diagonal copy of $p\Z$ inside $\Gamma_{p,0}\times \dots\times \Gamma_{p,p-1}$ acts on $X$ by the Bernoulli action $p\Z\actson X$. Continuing the chain of inclusions \eqref{eq:inclusion chain 2}, we obtain
\begin{align}\label{eq:inclusion chain 3}
L^{\infty}(Z_{p,0}\times \dots \times Z_{p,p-1}\times \mathbb{R})^{\Gamma_{p,0}\times \dots \times\Gamma_{p,p-1}}\subset L^{\infty}(X\times \mathbb{R})^{p\Z}=L^{\infty}(X\times \R)^{\mathbb{Z}},
\end{align}
where the last equality follows from point 2 of Lemma \ref{lem:S is in Z}, applied to the ergodic pmp action $\Z\actson Y=\Z/p\Z$. Combining \eqref{eq:inclusion chain 1}, \eqref{eq:inclusion chain 2} and \eqref{eq:inclusion chain 3}, we see that all inclusions must in fact be equalities.

Put $(D_p,\eta_p)=\prod_{n\in p\Z}(X_0,\mu_n)$. For each $i\in \{0,1,\dots,p-1\}$, the action $\Gamma_{p,i}\curvearrowright Z_{p,i}$ is conjugate with $p\mathbb{Z}\curvearrowright D_p$. From the equality
\begin{align*}
L^{\infty}(X\times \mathbb{R})^{\Z}=L^{\infty}(Z_{p,0}\times \dots \times Z_{p,p-1}\times \mathbb{R})^{\Gamma_{p,0}\times \dots \times \Gamma_{p,p-1}}
\end{align*}
it then follows that the associated flow of $\Z \actson X$ is the joint flow of $p$ copies of the associated flow of $p\Z \actson D_p$. This concludes the proof of the theorem.
\end{proof}

We end this section by proving Theorem \ref{thm:type classification Z}. We first need the following lemma.

\begin{lemma}\label{lem:type permutation action}
Let $X_0$ be a standard Borel space equipped with a sequence of equivalent probability measures $\mu_n$. Let $\mathcal{S}$ denote the group of finite permutations of $\mathbb{N}$ and let $\mathcal{S}_1\subset \mathcal{S}$ be the subgroup fixing $1\in \mathbb{N}$. Consider the nonsingular group actions $\mathcal{S}\actson (X,\mu)=\prod_{n=1}^{\infty}(X_0,\mu_n)$ and $\mathcal{S}_1\actson(Z,\eta)= \prod_{n=2}^{\infty}(X_0,\mu_n)$. Assume that the action $\mathcal{S}_1\curvearrowright (Z,\eta)$ is ergodic.

Then $\mathcal{S}\curvearrowright (X,\mu)$ is ergodic and the following holds.
\begin{enumlist}
\item $\mathcal{S}\curvearrowright (X,\mu)$ is of type II$_{1}$ if and only if there exists a probability measure $\nu\sim \mu_1$ on $X_0$ such that $\nu^{\mathbb{N}}\sim\mu$.
\item $\mathcal{S}\curvearrowright (X,\mu)$ is of type II$_{\infty}$ if and only if there exists a $\sigma$-finite measure $\nu\sim \mu_1$ on $X_0$ and Borel sets $\mathcal{U}_n\subset X_0$ such that $\nu(\mathcal{U}_n)<+\infty$ for all $n\in \mathbb{N}$ and such that
\begin{align*}
\sum_{n=1}^{\infty}\mu_n(X_0\setminus \mathcal{U}_n)<+\infty,\;\; \sum_{n=1}^{\infty}H^{2}\bigl(\mu_n, \nu(\mathcal{U}_n)^{-1}\nu|_{\mathcal{U}_n}\bigr)<+\infty,\;\;\sum_{n=1}^{\infty}\nu(X_0\setminus \mathcal{U}_n)=+\infty.
\end{align*}
\end{enumlist}
\end{lemma}

Note that Lemma \ref{lem:type permutation action} strongly resembles \cite[Theorem 3.3]{BV20}. There is however an important difference: in \cite[Theorem 3.3]{BV20}, it is part of the hypotheses that the Radon-Nikodym derivatives $d\mu_n / d\mu_0$ satisfy a certain boundedness condition. We do not make such an assumption, because we will use Lemma \ref{lem:type permutation action} in the context of totally arbitrary Bernoulli shifts. As a compensation, we make an ergodicity assumption on the permutation action. Thanks to Theorem \ref{thm.main-structure-variant}, this ergodicity assumption will hold automatically when the Bernoulli shift $\Z \actson X$ is conservative.

When $X_0$ is a finite set and $(\mu_n)_{n\geq 1}$ are equivalent probability measures on $X_0$, there is a necessary and sufficient ergodicity criterion for the nonsingular permutation action $\cS\actson (X,\mu)=\prod_{n=1}^{\infty}(X_0,\mu_n)$ in terms of the measures $\mu_n$, see \cite[Theorem 1.6]{AP77}. However, when $X_0$ is infinite, only sufficient conditions are known, see \cite[Theorems 1.8 \& 1.12]{AP77}.

The measure $\nu$ appearing in statement 2 of Lemma \ref{lem:type permutation action} is either finite, or infinite. Of course, if $\nu$ is finite, the condition $\nu(\mathcal{U}_n)<+\infty$ is automatically fulfilled. Similarly, when $\nu$ is infinite, the conditions $\nu(\mathcal{U}_n)<+\infty$ trivially imply that $\sum_{n=1}^{\infty}\nu(X_0\setminus\mathcal{U}_n)=+\infty$.

\begin{proof}
Suppose that $F\in L^{\infty}(X)$ is $\mathcal{S}$-invariant. As $\mathcal{S}_1$ acts ergodically on $(Z,\eta)$, we see that $F$ essentially only depends on the coordinate $x_1$. But as $F$ is $\mathcal{S}$-invariant, it follows that $F$ essentially only depends on the coordinate $x_2$, thus $F$ must be essentially constant. So the action $\mathcal{S}\curvearrowright (X,\mu)$ is ergodic.

If $x,y\in X$ are elements that differ in only finitely many coordinates, we write
\begin{align*}
\alpha(x,y)=\sum_{n\in \mathbb{N}}\big(\alpha_n(x_n)-\alpha_n(y_n)\big) \;\; ,\;\; \text{where}\;\; \alpha_n=\log\frac{d\mu_n}{d\mu_1} \; .
\end{align*}
Assume that the action $\mathcal{S}\curvearrowright (X,\mu)$ is semifinite. Then there exists a Borel map $F\colon X\rightarrow \mathbb{R}$ such that
\begin{align}\label{eq:conjugacy F}
\alpha(x,\sigma(x))=F(x)-F(\sigma(x)) \;\;\text{for every $\sigma\in \mathcal{S}$ and a.e.\ $x\in X$.}
\end{align}
Define $(\widetilde{X},\widetilde{\mu})=(X_0\times X_0\times Z,\mu_1\times \mu_1\times \eta)$, by doubling the first coordinate and consider the map
\begin{align*}
H\colon \widetilde{X}\rightarrow \mathbb{R}:\;\; H(x,x',z)=F(x,z)-F(x',z) \; .
\end{align*}
For each $\sigma\in \mathcal{S}_1$, we have that $H(x,x',z)=H(x,x',\sigma(z))$ for a.e. $(x,x',z)\in \widetilde{X}$. As the action $\mathcal{S}_1\curvearrowright (Z,\eta)$ is ergodic, $H$ is essentially independent of the $z$-variable. Therefore there exists a Borel map $L\colon X_0\times X_0 \rightarrow \R$ such that $H(x,x',z)=L(x,x')$ for a.e. $(x,x',z)\in \widetilde{X}$. Let $z\in Z$ be an element that witnesses this equality a.e. and put $\beta(x)=F(x,z)$. So we have found a Borel map $\beta\colon X_0\rightarrow \mathbb{R}$ such that
\begin{align*}
F(x)-F(y)=\beta(x_1)-\beta(y_1) \; ,
\end{align*}
when $x$ and $y$ are unequal only in the first coordinate. For $n\geq 2$, using \eqref{eq:conjugacy F} and the element $\sigma_n\in \mathcal{S}$ flipping the elements $1$ and $n$, we see that
\begin{align}\label{eq: difference coordinate n}
F(x)-F(y)=\alpha_n(x_n)+\beta(x_n)-\alpha_n(y_n)-\beta(y_n) \; ,
\end{align}
whenever $x,y\in X$ are unequal only in the $n$'th coordinate. When $x,y\in X$ are elements that differ in only finitely many coordinates, write
\begin{align*}
\Omega(x,y)=\sum_{n\in \mathbb{N}} \bigl(\alpha_n(x_n)+\beta(x_n)-\alpha_n(y_n)-\beta(y_n)\bigr)\; .
\end{align*}
Let $\mathcal{R}_{\Omega}$ be the equivalence relation on $X\times \mathbb{R}$ that is given by $(x,t)\sim (y,s)$ if and only if $x$ and $y$ differ only in finitely many coordinates and $s-t=\Omega(x,y)$. Then the flow $\mathbb{R}\curvearrowright L^{\infty}(X\times \mathbb{R})^{\mathcal{R}_{\Omega}}$ is isomorphic with the tail boundary flow associated to the sequence of probability measures $(\alpha_n+\beta)_*\mu_n$. By \eqref{eq: difference coordinate n} we have that $\Omega(x,y)=F(x)-F(y)$ for $x,y\in X$ that differ only in finitely many coordinates. We conclude that the tail boundary flow associated to $(\alpha_n+\beta)_*\mu_n$ is isomorphic with the translation action $\mathbb{R}\actson \mathbb{R}$.

We again use the cutoff function $T_\kappa : \R \to \R$ for $\kappa > 0$, as defined in \eqref{eq.cutoff}. By \cite[Proposition 2.1]{BV20} there exists a sequence $t_n\in \R$ such that
\begin{equation}\label{eq: van Kampen summability}
\sum_{n=1}^{\infty}\int_{X_0}T_{\kappa}(\alpha_n(x)+\beta(x)-t_n)^2d\mu_n(x)<+\infty,
\end{equation}
for every $\kappa>0$. Define the $\sigma$-finite measure $\nu \sim \mu_1$ by $d\nu/d\mu_1=\exp(-\beta)$. If $\nu$ is finite, then we can add a constant to $\beta$, so that $\nu$ becomes a probability measure. Then \eqref{eq: van Kampen summability} still holds with a potentially different sequence $t_n\in \R$. Thus we may assume that $\nu$ is either infinite, or a probability measure. Define the sets
\begin{align*}
\mathcal{U}_n=\{x\in X_0 \mid \; |\alpha_n(x)+\beta(x)-t_n|\leq 1\} \; .
\end{align*}
By \eqref{eq: van Kampen summability} we have that $\sum_{n=1}^{\infty}\mu_n(X_0\setminus \mathcal{U}_n)<+\infty$. There exists a $C>0$ be such that $|1-r|\leq C|\log(r)|$ for all $r\in [\exp(-1),\exp(1)]$, so that
\begin{align*}
|1-\exp(-(\alpha_n(x)+\beta(x)-t_n)/2)|\leq C|\alpha_n(x)+\beta(x)-t_n|, \;\;\text{for every}\;\; x\in \mathcal{U}_n \; .
\end{align*}
Taking $\kappa=1$, it follows from \eqref{eq: van Kampen summability} that
\begin{align*}
\sum_{n=1}^{\infty}\int_{\mathcal{U}_n}|1-\exp(-(\alpha_n(x)+\beta(x)-t_n)/2)|^2d\mu_n(x)<+\infty \; .
\end{align*}
The left hand side equals
\begin{align*}
\sum_{n=1}^{\infty}\int_{\mathcal{U}_n}\bigl(\sqrt{d\mu_n / d\mu_1}-\exp(t_n/2)\sqrt{d\nu / d\mu_1}\bigr)^{2} \, d\mu_1
\end{align*}
and since
\begin{align*}
\sum_{n=1}^{\infty}\int_{X_0}\bigl(\sqrt{d\mu_n / d\mu_1}-1_{\mathcal{U}_n}\, \sqrt{d\mu_n/d\mu_1}\bigr)^{2} \, d\mu_1=\sum_{n=1}^{\infty}\mu_n(X_0\setminus \mathcal{U}_n)<+\infty \; ,
\end{align*}
we conclude that also
\begin{equation}\label{eq: van Kampen summability 2}
\sum_{n=1}^{\infty}\int_{X_0}\bigl(\sqrt{d\mu_n / d\mu_1}-\exp(t_n/2) \, 1_{\mathcal{U}_n} \, \sqrt{d\nu / d\mu_1}\bigr)^{2} \, d\mu_1< +\infty \; .
\end{equation}
On $\mathcal{U}_n$ we have that $\exp(-\beta)\leq \exp(1+|t_n|)\exp(\alpha_n)$. As $\exp(\alpha_n)$ is $\mu_1$-integrable, it follows that $\exp(-\beta)$ is $\mu_1|_{\mathcal{U}_n}$-integrable, so $\nu(\mathcal{U}_n)<+\infty$. For each $n\in \mathbb{N}$ define the probability measure $\nu_n=\nu(\mathcal{U}_n)^{-1}\nu|_{\mathcal{U}_n}$.

For any pair of nonzero vectors $\xi_1,\xi_2$ in a Hilbert space, with $\|\xi_1\|=1$, we have that $\|\xi_1-\|\xi_2\|^{-1} \xi_2\|\leq 2\|\xi_1-\xi_2\|$. Thus it follows from \eqref{eq: van Kampen summability 2} that
\begin{align*}
\sum_{n=1}^{\infty}H^2(\mu_n,\nu_n)<+\infty \; .
\end{align*}
Define the map $\gamma_n$ by
\begin{align*}
\gamma_n\colon X_0\rightarrow \R:\;\; \gamma_n(x)=\log\frac{d\mu_n}{d\nu}(x)+\log(\nu(\mathcal{U}_n)) \; .
\end{align*}
Let $D>0$ be such that $T_{1}(r)\leq D|1-\exp(-r/2)|$, for every $r\in \R$. It follows that
\begin{align}\label{eq: gamma van Kampen}
\int_{\mathcal{U}_n}T_1(\gamma_n(x))^2 \, d\mu_n(x) \leq D^2 \, \int_{\mathcal{U}_n}|1-\exp(-\gamma_n/2)|^2d\mu_n\leq D^2 \, H^2(\mu_n,\nu_n)\; .
\end{align}
Define an increasing sequence of subsets $Z_k\subset X$ by
\begin{align*}
Z_k=\{x\in X \mid x_m\in \mathcal{U}_m \;\;\text{for every}\;\; m> k \} \; .
\end{align*}
Note that $\mu(X\setminus Z_k)\rightarrow 0$ as $\sum_{n=1}^{\infty}\mu_n(X\setminus \mathcal{U}_n)<+\infty$. By \eqref{eq: gamma van Kampen} the function
\begin{align}\label{eq:a.e. convergent sum}
G(x)=\sum_{n=1}^{\infty}\bigl(\log\frac{d\mu_n}{d\nu}(x_n)+\log(\nu(\mathcal{U}_n))\bigr)
\end{align}
converges unconditionally a.e. on $Z_k$, for every $k\in \mathbb{N}$. Therefore the sum converges unconditionally a.e. As in the proof of \cite[Theorem 3.3]{BV20} the $\sigma$-finite measure $\zeta\sim \mu$ defined by $d\zeta/d\mu=\exp(-G)$ is $\mathcal{S}$-invariant and $\zeta$ is a finite measure if and only if $\nu$ is finite and $\sum_{n=1}^{\infty}\nu(X_0\setminus \mathcal{U}_n)<+\infty$.

If $\mathcal{S}\curvearrowright (X,\mu)$ is of type II$_{1}$, then $\zeta$ must be finite and we have that $\mu\sim \nu^{\mathbb{N}}$. Conversely, if there exists a probability measure $\nu\sim \mu_1$ such that $\mu\sim \nu^{\mathbb{N}}$, it follows directly that $\mathcal{S}\curvearrowright (X,\mu)$ is of type II$_{1}$.

If $\mathcal{S}\curvearrowright (X,\mu)$ is of type II$_{\infty}$, then $\zeta$ must be infinite. So we have shown that all conditions in point 2 of the lemma are true. Conversely, if there exist a $\sigma$-finite measure $\nu$ and subsets $\mathcal{U}_n\subset X_0$ satisfying the conditions of point 2, then the sum \eqref{eq:a.e. convergent sum} converges a.e. and the $\sigma$-finite measure $\zeta\sim \mu$ defined by $d\zeta/d\mu=\exp(-G)$ is infinite and $\mathcal{S}$-invariant.
\end{proof}

\begin{proof}[Proof of Theorem \ref{thm:type classification Z}]
First of all, by Theorem \ref{thm.main-structure-variant} the Bernoulli shift $\Z\actson (X,\mu)$ is weakly mixing and $\mathcal{S}\curvearrowright (X,\mu)$ is ergodic. Moreover the Maharam extensions of $\Z\actson (X,\mu)$ and $\cS\actson (X,\mu)$ satisfy
\begin{align}\label{eq:conjugate semifinite flows}
L^{\infty}(X\times \mathbb{R})^{\mathbb{Z}}=L^{\infty}(X\times \mathbb{R})^{\mathcal{S}}.
\end{align}
Write $(Z,\eta)=\prod_{n\in \mathbb{Z}, n\neq 1}(X_0,\mu_n)$. Let $\mathcal{S}_1\subset\mathcal{S}$ be the subgroup of finite permutations that fix $1\in \mathbb{Z}$, so that $\mathcal{S}_1\actson (Z,\eta)$. Let $\theta$ denote the partial shift
\begin{align*}
\theta\colon (X,\mu)\rightarrow (Z,\eta):\;\; \theta(x)_n=\begin{cases}x_n &\;\;\text{if}\;\; n\leq 0\\
x_{n-1} &\;\;\text{if}\;\; n\geq 2 \end{cases}.
\end{align*}
It follows directly from the Kakutani criterion that $\theta$ is a nonsingular isomorphism. It intertwines the actions $\mathcal{S}\curvearrowright (X,\mu)$ and $\mathcal{S}_1\curvearrowright (Z,\eta)$. So $\mathcal{S}_1\actson (Z,\eta)$ is ergodic and Lemma \ref{lem:type permutation action} applies. If $\mathbb{Z}\actson (X,\mu)$ is semifinite, then by \eqref{eq:conjugate semifinite flows} also $\mathcal{S}\actson (X,\mu)$ is semifinite. To complete the proof of the theorem, it suffices to show that $\mathbb{Z}\actson (X,\mu)$ is of type II$_{\infty}$ if there exist a $\sigma$-finite measure $\nu$ and Borel sets $\mathcal{U}_n\subset X_0$ satisfying the conditions of the second point of the theorem.

Assume we are given such $\nu$ and $\mathcal{U}_n$. Then, as in \eqref{eq:a.e. convergent sum}, the sum
\begin{align*}
G(x)=\sum_{n\in \Z}\left(\log\frac{d\mu_n}{d\nu}(x_n)+\log(\nu(\mathcal{U}_n))\right)
\end{align*}
is unconditionally convergent a.e. By \eqref{eq:conjugate semifinite flows} the map $(x,t)\mapsto t-G(x)$ is invariant under the Maharam extension of $\Z$. Also, it is $\R$-equivariant.

It follows that the $\sigma$-finite measure $\zeta\sim \mu$ defined by $d\zeta/d\mu=\exp(-G)$ is an infinite $\Z$-invariant measure.
\end{proof}

\section{Nonsingular Poisson suspensions: proof of Proposition \ref{prop.poisson-suspension}}\label{sec:Poisson suspension}

We start this section by recalling the construction of the Poisson suspension. For a detailed treatment, we refer to \cite{Roy08,DKR20}. Let $(X_0,\mu_0)$ be a $\sigma$-finite standard measure space. We write $\mathcal{B}_0=\{A\subset X_0 \mid A \;\;\text{is Borel and}\;\; \mu_0(A)<+\infty\}$. To $(X_0,\mu_0)$, one associates a standard probability space $(X,\mu)$ and random variables $P_A\colon X\rightarrow \{0,1,2,\dots\}$ for every $A\in \mathcal{B}_0$ such that the following holds.
\begin{enumlist}
\item The random variable $P_A$ is Poisson distributed with intensity $\mu_0(A)$.
\item If $A,B\in \mathcal{B}_0$ are disjoint, then $P_A$ and $P_B$ are independent random variables and we have that $P_{A\cup B}=P_A+P_B$.
\item The family $(P_A)_{A\in \mathcal{B}_0}$ separates the points of $X$.
\end{enumlist}
These three properties uniquely characterize $(X,\mu)$ and the random variables $(P_A)_{A\in \mathcal{B}_0}$. The probability space $(X,\mu)$ is called the \emph{Poisson suspension} over the base space $(X_0,\mu_0)$.

By the functoriality of this construction, every measure preserving Borel automorphism $\theta : X_0 \to X_0$ gives rise to an essentially unique, measure preserving Borel automorphism $\thetah : X \to X$ such that for every $A \in \cB_0$, we have
\begin{equation}\label{eq.characterize-suspension-aut}
P_A(\thetah(x)) = P_{\theta^{-1}(A)}(x) \quad\text{for $\mu$-a.e.\ $x \in X$.}
\end{equation}
In \cite[Theorem 3.3]{DKR20}, it was discovered that a nonsingular Borel automorphism $\theta : X_0 \to X_0$ gives rise to a nonsingular Borel automorphism $\thetah : X \to X$ satisfying \eqref{eq.characterize-suspension-aut} if and only if
\begin{equation}\label{eq.L2-bound}
\sqrt{\frac{d(\theta_* \mu_0)}{d\mu_0}} - 1 \in L^2(X_0,\mu_0) \; .
\end{equation}
For completeness, we include below a short proof of one implication, namely that every $\theta$ satisfying \eqref{eq.L2-bound} has a suspension $\thetah$. This proof is essentially taken from \cite{DKR20}, but presented in a more direct way.

So, whenever $G \actson (X_0,\mu_0)$ is a nonsingular action of a locally compact second countable (lcsc) group such that
\begin{equation}\label{eq.new-bound-K}
\sup_{g \in K} \bigl\| \sqrt{d(g \mu_0)/d\mu_0} - 1 \bigr\|_2 < + \infty \quad\text{for every compact $K \subset G$,}
\end{equation}
we have an essentially unique nonsingular action $G \actson (X,\mu)$ characterized by $P_A(g \cdot x) = P_{g^{-1}\cdot A}(x)$, which is called the \emph{Poisson suspension action}.

The main goal of this section is to prove Proposition \ref{prop.poisson-suspension}. We actually prove the following \emph{stable} version, that also considers the associated flow of the diagonal action $G\actson X\times Y$ for any ergodic pmp action $G\actson (Y,\nu)$.

\begin{proposition}\label{prop:stable Poisson suspension}
Let $G$ be any lcsc group that does not have property $(T)$. Let $\R\actson (Z,\zeta)$ be any Poisson flow. Then $G$ admits a nonsingular action $G\actson (X_0,\mu_0)$ of which the Poisson suspension $G\actson (X,\mu)$ is well-defined, weakly mixing, essentially free and such that for any ergodic pmp action $G\actson (Y,\nu)$ the diagonal action $G\actson X\times Y$ has associated flow $\R\actson Z$.
\end{proposition}

Proposition \ref{prop:stable Poisson suspension} is sharp in the following sense: by \cite[Theorem G]{DKR20}, if $G$ has property~(T), then every Poisson suspension action admits an equivalent invariant probability measure. This follows by applying the Delorme-Guichardet theorem (see e.g.\ \cite[Theorem 2.12.4]{BHV08}) to the $1$-cocycle $g \mapsto \sqrt{d(g \mu_0)/d\mu_0} - 1$ with values in the Koopman representation of $G \actson (X_0,\mu_0)$.

Before proving Proposition \ref{prop:stable Poisson suspension}, we introduce some further background, based on \cite{DKR20}. In particular, we give a short proof that every nonsingular automorphism $\theta : X_0 \to X_0$ satisfying \eqref{eq.L2-bound} admits a Poisson suspension $\thetah$. This proof is essentially taken from \cite{DKR20}, but since our approach is direct and short, we include it here for convenience of the reader.

Write $\cH = L^2_\R(X_0,\mu_0)$ and denote by $\cH_n \subset \cH^{\ot n}$ the closed subspace of symmetric vectors (i.e.\ invariant under the action of the symmetric group $S_n$ on $H^{\ot n}$). The key point is that there is a canonical isometric isomorphism
\begin{equation}\label{eq.fock-iso}
U : \R \oplus \bigoplus_{n=1} \cH_n \to L^2_\R(X,\mu)
\end{equation}
between the symmetric Fock space $\cF_s(\cH)$ of $\cH$ and $L^2_\R(X,\mu)$. This isomorphism $U$ is defined as follows. For every $\xi \in \cH$, denote by $\exp(\xi) \in \cF_s(\cH)$ the usual exponential given by
$$\exp(\xi) = 1 \oplus \bigoplus_{n=1}^\infty \frac{1}{n!} \xi^{\ot n} \; .$$
For every $A\in \mathcal{B}_0$ and $\alpha\in \mathbb{R}$ define $T_{\alpha,A}\colon X\rightarrow \R$ by
\begin{align}\label{eq:T notation}
T_{\alpha,A}=\exp(-(\alpha-1)\mu_0(A)) \, \alpha^{P_A} \; .
\end{align}

If $\xi\in 1+L^2_\R(X_0,\mu_0)$ takes only finitely many function values, i.e.\ is of the form $\xi=1_A +\sum_{i=1}^m\alpha_i1_{A_i}$ for $\alpha_i\in \R$, disjoint $A_i\in \mathcal{B}_0$ and $A = \R \setminus \bigcup_{i=1}^m A_i$, we define $\cT(\xi)\in L^2_{\R}(X,\mu)$ by
\begin{align*}
\cT(\xi)=\prod_{i=1}^{m}T_{\alpha_i,A_i} \; .
\end{align*}
A direct computation shows that
$$\langle \cT(\xi) , \cT(\xi') \rangle = \exp(\langle \xi-1,\xi'-1\rangle) = \langle \exp(\xi-1), \exp(\xi'-1) \rangle$$
for all $\xi,\xi'$ of such a form. The vectors $\exp(\xi-1)$ are total in the symmetric Fock space $\cF_s(\cH)$ and the functions $\cT(\xi)$ are total in $L^2_\R(X,\mu)$. Therefore, $\cT$ uniquely extends to a continuous map $\cT : 1+L^2_\R(X_0,\mu_0) \to L^{2}_\R(X,\mu)$ and the isomorphism $U$ in \eqref{eq.fock-iso} is uniquely defined by $U(\exp(\xi-1)) = \cT(\xi)$ for all $\xi \in 1 + L^2_\R(X_0,\mu_0)$.

When $\xi,\xi' \in 1+L^2_\R(X_0,\mu_0)$ are such that their product $\xi \xi'$ still belongs to $1 + L^2_\R(X_0,\mu_0)$, we have
\begin{equation}\label{eq.product-formula}
\cT(\xi) \, \cT(\xi') = \exp(\langle \xi-1,\xi'-1\rangle) \, \cT(\xi \xi') \; ,
\end{equation}
which is immediate when $\xi,\xi'$ only take finitely many values and is then valid for all $\xi,\xi'$ by density and continuity.

When $\xi \in 1 + L^2_\R(X_0,\mu_0)$ satisfies $\xi(x) \geq 0$ for a.e.\ $x \in X_0$, we have by construction that $\cT(\xi)(x) \geq 0$ for a.e.\ $x \in X$. We next prove that if $\xi(x) > 0$ for a.e.\ $x \in X_0$, then $\cT(\xi)(x) > 0$ for a.e.\ $x \in X$. To see this, define $C_n = \{x \in X \mid n^{-1} \leq \xi(x) \leq n\}$ and define $\xi_n \in 1 + L^2_\R(X_0,\mu_0)$ by $\xi_n(x) = \xi(x)^{-1}$ if $x \in C_n$ and $\xi_n(x) = 1$ if $x \not\in C_n$. Define $F_n \in L^2_\R(X,\mu)$ by $F_n = \exp(-\langle \xi - 1, \xi_n - 1 \rangle) \, \cT(\xi_n)$. By \eqref{eq.product-formula}, we get that $\cT(\xi) F_n \to 1$ in $L^2_\R(X,\mu)$. Therefore, $\cT(\xi)(x) \neq 0$ for a.e.\ $x \in X$.

Note, for later use, that it follows in particular that if $\xi \in 1+L^2_\R(X_0,\mu_0)$ is such that also $\xi^{-1} \in 1+L^2_\R(X_0,\mu_0)$, then
\begin{equation}\label{eq.inverse}
\cT(\xi)^{-1} = \exp(-\langle \xi-1,\xi^{-1} - 1 \rangle) \, \cT(\xi^{-1}) \; .
\end{equation}

Now assume that $\nu_0$ is a $\sigma$-finite measure on $X_0$ such that $\nu_0 \sim \mu_0$. Write $\xi = \sqrt{d\nu_0/d\mu_0}$ and assume that $\xi \in 1 + L^2_\R(X_0,\mu_0)$. Define $F \in L^2(X,\mu)$ by $F = \exp(-\|\xi-1\|^2_2/2) \, \cT(\xi)$. By the discussion above, $F(x)>0$ for a.e.\ $x \in X$. Also,
$$\int_X F^2 \, d\mu = \exp(-\|\xi-1\|_2^2) \; \langle \cT(\xi), \cT(\xi) \rangle = 1 \; .$$
We can thus define a unique probability measure $\nu \sim \mu$ such that
\begin{equation}\label{eq.formula-nu}
\sqrt{d\nu/d\mu} = \exp(-\|\sqrt{d\nu_0/d\mu_0}-1\|^2_2/2) \; \cT(\sqrt{d\nu_0/d\mu_0}) \; .
\end{equation}
We show that for every $A \in \cB_0$, the variable $P_A$ has, w.r.t.\ $\nu$, a Poisson distribution with intensity $\nu_0(A)$. Let $\al > 0$ be arbitrary. Still writing $\xi = \sqrt{d\nu_0/d\mu_0}$, define $\xi_A \in 1 + L^2_\R(X_0,\mu_0)$ by $\xi_A(x) = \al \xi(x)$ if $x \in A$ and $\xi_A(x) = \xi(x)$ if $x \not\in A$. By \eqref{eq.product-formula}, we get that
$$\al^{P_A} \, \cT(\xi) = \exp\Bigl( (\al-1) \int_A \xi \, d\mu_0 \Bigr) \; \cT(\xi_A) \; .$$
It follows that
$$\int_X \al^{P_A} \, d\nu = \exp(-\|\xi-1\|_2^2) \; \langle \al^{P_A} \, \cT(\xi) , \cT(\xi) \rangle = \exp((\al-1)\nu_0(A)) \; .$$
Since this holds for every $\al > 0$, the distribution of $P_A$ w.r.t.\ $\nu$ is Poisson with intensity $\nu_0(A)$.

Finally, given $A \in \cB_0$, the partition $X_0 = A \sqcup (X_0 \setminus A)$ gives rise to a decomposition of $X$ as the direct product of the Poisson suspension of $(A,\mu_0)$ and the Poisson suspension of $(X_0 \setminus A,\mu_0)$. By \eqref{eq.product-formula} and \eqref{eq.formula-nu}, also $\nu$ is a product measure in this decomposition. It follows that $P_A$ is independent from $P_B$ for every $B \in \cB_0$ that is disjoint from $A$.

Altogether, we have thus proven that $(X,\nu)$ is the Poisson suspension of $(X_0,\nu_0)$. This observation implies that for every automorphisms $\theta$ satisfying \eqref{eq.L2-bound}, there is an essentially unique nonsingular Poisson suspension $\thetah$ given by \eqref{eq.characterize-suspension-aut}. So also, for every nonsingular action $G \actson (X_0,\mu_0)$ satisfying \eqref{eq.new-bound-K}, the Poisson suspension action $G \actson (X,\mu)$ is well defined.

An important step in the proof of Proposition \ref{prop:stable Poisson suspension} is to ensure that the Poisson suspension action $G \actson (X,\mu)$ is conservative. If in the discussion above, both $\xi = \sqrt{d\nu_0/d\mu_0}$ and $\xi^{-1}$ belong to $1 + L^2_\R(X_0,\mu_0)$, it follows from \eqref{eq.inverse} that
$$\sqrt{d\mu/d\nu} = \exp\bigl( \|\xi-1\|_2^2/2 - \langle \xi-1,\xi^{-1} - 1 \rangle \bigr) \, \cT(\xi^{-1}) \; ,$$
so that
\begin{equation}\label{eq:expectation inverse RN}
\int_X \frac{d\mu}{d\nu} \, d\mu = \exp\bigl( \|\xi-1\|_2^2 - 2 \langle \xi-1,\xi^{-1} - 1 \rangle + \|\xi^{-1}-1\|_2^2\bigr) = \exp(\|\xi-\xi^{-1}\|_2^2) \; .
\end{equation}
So whenever $G \actson (X_0,\mu_0)$ is a nonsingular action, we denote
\begin{align}\label{eq:kappa}
\kappa(g)=\bigl\|\sqrt{d(g\mu_0) / d\mu_0}- \sqrt{d\mu_0 / d(g \mu_0)}\bigr\|^2_2 \in [0,+\infty] \; .
\end{align}
The following result is then an immediate consequence of \eqref{eq:expectation inverse RN} and \cite[Lemma 2.6]{BV20} (adapted to the locally compact setting), and is similar to \cite[Lemma 1.3]{DK20}.

\begin{proposition}\label{prop:conservativeness criterion Poisson}
Let $G\actson(X_0,\mu_0)$ be a nonsingular action of a locally compact second countable group $G$ on an infinite, $\sigma$-finite standard measure space $(X_0,\mu_0)$ such that \eqref{eq.new-bound-K} holds. Let $\lambda$ denote a left invariant Haar measure on $G$. If
\begin{align*}
\limsup_{s\rightarrow +\infty}\frac{\log\lambda(\{g\in G \mid \kappa(g^{\pm 1})\leq s\})}{s}>3 \; ,
\end{align*}
then the nonsingular Poisson suspension $G\actson (X,\mu)$ is conservative.
\end{proposition}

For lcsc groups $G$ with the Haagerup approximation property, one can give a very short proof of Proposition \ref{prop:stable Poisson suspension} by only using Proposition \ref{prop:conservativeness criterion Poisson}. For the general case, where $G$ is only assumed to be non-(T), we have to make use of recurrence of $G \actson (X,\mu)$ along specific subsets $\Lambda \subset G$. We thus recall the following two definitions from \cite{AIM19}.

\begin{definition}[{\cite[Definitions 7.12 and 7.13]{AIM19}}]
Let $\pi\colon G\actson \cH$ be an orthogonal representation on a real Hilbert space $\cH$ and $\Lambda \subset G$ a countable subset. We say that $\pi$ is \emph{mixing along $\Lambda$} if
\begin{align*}
\lim_{g\in \Lambda,\;\;g\rightarrow \infty} \langle\pi_g(\xi),\eta\rangle=0 \;\;\text{for every $\xi,\eta\in \mathcal{H}$.}
\end{align*}
A pmp action $G \actson (X,\mu)$ is \emph{mixing along $\Lambda$} if the reduced Koopman representation
\begin{align*}
\pi\colon G\rightarrow \mathcal{O}(L^2_\R(X,\mu)\ominus \R 1): (\pi_g\xi)(x)=\xi(g^{-1} \cdot x)
\end{align*}
is mixing along $\Lambda$.

A nonsingular action $G \actson (X,\mu)$ is said to be \emph{recurrent along $\Lambda$} if for every nonnegligible Borel set $A\subset X$ there exist infinitely many $g\in \Lambda$ such that $\mu(g\cdot A\cap A)>0$.
\end{definition}

Similarly to Proposition \ref{prop:conservativeness criterion Poisson}, combining \cite[Lemma 7.15]{AIM19} and \cite[Lemma 2.6]{BV20}, we then get the following.

\begin{proposition}\label{prop.recurrence-along-Poisson}
Let $G\actson(X_0,\mu_0)$ be a nonsingular action of a locally compact second countable group $G$ on an infinite, $\sigma$-finite standard measure space $(X_0,\mu_0)$ such that \eqref{eq.new-bound-K} holds. Let $\Lambda \subset G$ be a countable infinite subset. If
\begin{align*}
\limsup_{s\rightarrow +\infty}\frac{\log |\{g\in \Lambda \mid \kappa(g^{\pm 1})\leq s\}|}{s}>3 \; ,
\end{align*}
then the nonsingular Poisson suspension $G\actson (X,\mu)$ and, for every pmp action $G \actson (Y,\nu)$, the Maharam extension of the diagonal action $G \actson X \times Y$ are recurrent along $\Lambda\Lambda^{-1}$.
\end{proposition}

The first step in proving Proposition \ref{prop:stable Poisson suspension} is to translate the absence of property~(T) to a dynamical property: there exists a measure preserving action $G \actson (W,\rho)$ on a standard, infinite, $\sigma$-finite measure space such that the Koopman representation on $L^2(W,\rho)$ is weakly mixing and the action admits a F{\o}lner sequence, i.e.\ a sequence of Borel sets $A_n \subset W$ such that $0<\rho(A_n)<+\infty$ for every $n \in \N$ and
$$\lim_{n\rightarrow \infty}\frac{\rho(g \cdot A_n\vartriangle A_n)}{\rho(A_n)} = 0 \quad\text{uniformly on compact subsets $K\subset G$.}$$
This characterization is proven in \cite[Section 4]{Dan21} and is a consequence of the Connes-Weiss characterization of property~(T) (see \cite{CW80} and \cite[Theorem 6.3.4]{BHV08}): taking a weakly mixing pmp action $G \actson (B,\beta)$ that admits a F{\o}lner sequence $B_n \subset B$ with $\beta(B_n) = 1/2$ for all $n \in \N$, we can pass to a subsequence and assume that for every compact $K \subset G$,
$$\sum_{n=1}^\infty \sup_{g \in K} \beta(g \cdot B_n \vartriangle B_n) < +\infty \; .$$
Then, the diagonal action of $G$ on the restricted infinite product
$$W = \prod_{n=1}^\infty (B,B_n) = \{x \in B^\N \mid x_n \in B_n \;\;\text{for all but finitely many $n \in \N$}\;\}$$
equipped with the product measure $\rho = (2\be)^\N$ is well defined and satisfies all the requirements.

\begin{lemma}\label{lem.technical-suspension-lemma}
Let $G$ be a lcsc group and $G \actson (W,\rho)$ a measure preserving action on a standard, infinite, $\sigma$-finite measure space such that its Koopman representation is weakly mixing and the action admits a F{\o}lner sequence $A_n \subset W$. Let $\lambda_n > 0$ and $a_n \in \R \setminus \{0\}$. Define $X_0 = W \times \N$ and define the measure $\mu_0$ on $X_0$ by
$$\mu_0(\cU \times \{n\}) = \lambda_n \, \rho(A_n)^{-1} \, \bigl( \rho(\cU \cap A_n) + e^{a_n} \, \rho(\cU \setminus A_n)\bigr) \; .$$
Consider the nonsingular action $G \actson (X_0,\mu_0)$ defined by $g \cdot (x,n) = (g \cdot x , n)$.

After replacing $A_n$ by a subsequence, we get that the Poisson suspension $G \actson (X,\mu)$ of $G \actson (X_0,\mu_0)$ is well defined and weakly mixing, and has the property that for any ergodic pmp action $G\actson (Y,\nu)$, the associated flow of the diagonal action $G\actson X\times Y$ is given by the tail boundary flow of the sequence of Poisson distributions $\si_{\lambda_n,a_n}$ defined by \eqref{eq:Poisson dist}.
\end{lemma}
\begin{proof}
Denote by $\pi : G \to \cU(L^2(W,\rho)) : (\pi(g)\xi)(x) = \xi(g^{-1}\cdot x)$ the Koopman representation. By \cite[Lemma 7.17]{AIM19}, we can fix an infinite subset $\Lambda\subset G$ such that $\pi$ is mixing along $\Lambda \Lambda^{-1}$. Define
\begin{align*}
& \kappa_n : G \to [0,+\infty) : \kappa_n(g) = \frac{1}{2} \, \lambda_n \, \frac{\rho(g \cdot A_n \vartriangle A_n)}{\rho(A_n)} \, (e^{a_n} - e^{-a_n})(e^{a_n} - 1) \quad\text{and}\\
& \kappa_0 : G \to [0,+\infty] : \kappa_0(g) = \sum_{n=1}^\infty \kappa_n(g) \; .
\end{align*}
Replacing $A_n$ by a subsequence, we may assume that
\begin{align*}
& \sup_{g \in K} \kappa_0(g) <+\infty \quad\text{for every compact subset $K \subset G$, and}\\
& \limsup_{s\rightarrow +\infty}\frac{\log | \{g\in \Lambda \mid \kappa_0(g^{\pm 1})\leq s \}|}{s}>3 \; .
\end{align*}

Define the measure $\mu_0$ as in the formulation of the lemma. A direct computation gives that
$$\bigl\| \sqrt{d(g \cdot \mu_0)/d\mu_0} - \sqrt{d\mu_0 / d(g \cdot \mu_0)} \bigr\|^2_2 = \kappa_0(g) < +\infty \; .$$
Since $|a-1| \leq |a - a^{-1}|$ for all $a > 0$, we also find that \eqref{eq.new-bound-K} holds. So, the Poisson suspension $G\actson (X,\mu)$ of $G\actson (X_0,\mu_0)$ is well-defined. By Proposition \ref{prop.recurrence-along-Poisson}, the action $G \actson (X,\mu)$ is recurrent along $\Lambda \Lambda^{-1}$.

Let $\gamma : G\actson (Y,\nu)$ be any ergodic pmp action. We prove that the diagonal action $G \actson X \times Y$ is ergodic and that its associated flow is isomorphic with $\R \actson Z$. Part of this argument is similar to the proof of \cite[Proposition 1.17]{Dan21}.

Denote by $\rho_n$ the measure on $W$ given by $\rho_n(\cU) = \mu_0(\cU \times \{n\})$. Viewing $X_0$ as the disjoint union of the $G$-invariant subsets $W\times \{n\}$, we identify $(X,\mu)$ with $\prod_{n = 1}^\infty (X_n,\mu_n)$, where $(X_n,\mu_n)$ is the Poisson suspension of $(W,\rho_n)$. Then $\al : G \actson (X,\mu)$ is the diagonal product action of $\al_n : G \actson (X_n,\mu_n)$. Define the probability measures $\nu_n\sim \mu_n$ by
\begin{align}\label{eq:curly E RN derivative}
\frac{d\nu_n}{d\mu_n}=\cT(1_{W \setminus A_n} + e^{a_n} 1_{A_n}) \; .
\end{align}
Then $G\actson (X_n,\nu_n)$ is the Poisson suspension of $G\actson (W,\beta_n \rho)$, with $\beta_n = \lambda_n e^{a_n} \rho(A_n)^{-1}$. Since $G\actson (W,\beta_n\rho)$ is measure preserving, the probability measure $\nu_n$ is $G$-invariant. Through \eqref{eq.fock-iso}, the Koopman representation of $G \actson (X_n,\nu_n)$ is the natural representation of $G$ on the symmetric Fock space $\cF_s(L^2_\R(W,\beta_n \rho))$ associated with the Koopman representation of $G$ on $L^2_\R(W,\beta_n \rho)$. Therefore, $G \actson (X_n,\nu_n)$ is weakly mixing along $\Lambda \Lambda^{-1}$.

Consider the Maharam extension $G \actson X \times Y \times \R$ of the diagonal action $G \actson X \times Y$. Fix $F \in L^\infty(X \times Y \times \R)^G$. Define $\psi_n=-\log d\mu_n/d\nu_n$. Fix $N \in \N$ and write
$$(X'_N,\mu'_N) = \prod_{n=1}^N (X_n,\mu_n) \quad\text{and}\quad (X\dpr_N,\mu\dpr_N) = \prod_{n=N+1}^\infty (X_n,\mu_n) \; .$$
We identify $X = X'_N \times X\dpr_N$ and define
\begin{multline*}
\mbox{}\hspace{1cm}\Psi_N : X'_N \times X\dpr_N \times Y \times \R \to X'_N \times X\dpr_N \times Y \times \R : \\ \Psi_N(x',x\dpr,y,t) = \Bigl(x',x\dpr,y,t+\sum_{n=1}^N \psi_n(x'_n)\Bigr) \; .\hspace{1cm}\mbox{}
\end{multline*}
We denote by $\al'_N : G \actson X'_N$ and $\al\dpr_N : G \actson X\dpr_N$ the obvious diagonal product actions, so that $\al$ is the diagonal product of $\al'_N$ and $\al\dpr_N$.

Since $\nu_n$ is $G$-invariant, the map $\Psi_N$ intertwines the Maharam extension of $\al \times \gamma$ with the diagonal product of $\al'_N$ and the Maharam extension of $\al\dpr_N \times \gamma$. Since the pmp action $\al'_N$ of $G$ on $(X'_N,\nu_1 \times \cdots \times \nu_N)$ is mixing along $\Lambda \Lambda^{-1}$ and since the Maharam extension of $\al\dpr_N \times \gamma$ is recurrent along $\Lambda \Lambda^{-1}$, it follows from \cite[Theorem 7.14]{AIM19} that the function $F \circ \Psi_N^{-1}$ is essentially independent of the $X'_N$-variable.

So, for every $N \in \N$, we find a unique $F_N \in L^\infty(X\dpr_N \times Y \times \R)$ with $\|F_N\|_\infty = \|F\|_\infty$ and $F = F_N \circ \Psi_N$ a.e. Define the probability measures $\zeta_n = (\psi_n)_*(\mu_n)$ on $\R$. Denote
$$(\Omega,\zeta) = \prod_{n =1}^\infty (\R,\zeta_n)$$
and realize the tail boundary $B$ of the sequence $(\zeta_n)_{n \in \N}$ inside $L^\infty(\R \times \Om)$. Consider the natural factor map
$$\Phi : X \times Y \times \R \to \R \times \Om \times Y : \Phi(x,y,t) = (t, (\psi_n(x_n))_{n \in \N} , y) \; .$$
We have thus proven that for every $F \in L^\infty(X \times Y \times \R)^G$, there exists a unique $H \in B \ovt L^\infty(Y)$ such that $F = H \circ \Phi$ a.e.

We also consider the diagonal action $\al'_N \times \id \times \gamma$ of $G$ on $X \times Y$, with its Maharam extension $M_N : G \actson X \times Y \times \R$. Since $(\al'_N \times \id \times \gamma)_g \to (\al \times \gamma)_g$ in $\Aut(X \times Y)$, we also have that $M_{N,g}(F) \to F$ weakly. For every $K \in B \ovt L^\infty(Y)$, we have that $\Phi_*((\id \ot \gamma_g)(K)) = M_{N,g}(\Phi_*(K))$ for all $N \in \N$. Since $F = \Phi_*(H)$, we conclude that $\Phi_*((\id \ot \gamma_g)(H)) = \Phi_*(H)$ for all $g \in G$. Since the action $\gamma$ is ergodic, we conclude that $H \in B \ot 1$.

Conversely, for every $H \in B$, we find that $\Phi_*(H \ot 1)$ is invariant under $M_{N,g}$ for all $N \in \N$ and $g \in G$. Since $M_{N,g}$ converges to the Maharam extension of $\al \times \gamma$, it follows that $\Phi_*(H \ot 1)$ is invariant under the latter. We have thus identified the associated flow of $G \actson X \times Y$ with the tail boundary flow of the sequence $(\zeta_n)_{n \in \N}$.

By \eqref{eq:curly E RN derivative} and using the notation \eqref{eq:T notation},
$$
\psi_n =\log(T_{e^{a_n}, A_n}) = (1-e^{a_n})\rho_n(A_n)+a_nP_{A_n} \; .
$$
Therefore, $\zeta_n$ is a translation of the Poisson distribution with support $\{ka_n: k=0,1,2,\dots\}$ and intensity $\rho_n(A_n)=\lambda_n$, i.e.\ a translation of $\si_{\lambda_n,a_n}$. This concludes the proof of the lemma.
\end{proof}

\begin{proof}[Proof of Proposition \ref{prop:stable Poisson suspension}]
Take a lcsc group $G$ that does not have property (T) and take a Poisson flow $\R \actson Z$. By Proposition \ref{prop.standard-poisson} and using the notation in \eqref{eq:Poisson dist}, we write $\R \actson Z$ as the tail boundary flow of a sequence $\sigma_{\lambda_n,a_n}$ with $\lambda_n > 0$ and $a_n \in \R \setminus \{0\}$. As explained above, we can take a measure preserving action $G \actson (W,\rho)$ on a standard, infinite, $\sigma$-finite measure space such that the Koopman representation $\pi$ of $G$ on $L^2(W,\rho)$ is weakly mixing and the action admits a F{\o}lner sequence $A_n \subset W$.

Replace $A_n$ by a subsequence and define $(X_0,\mu_0)$ so that the conclusions of Lemma \ref{lem.technical-suspension-lemma} hold. Define $V_0 = G \times \N$ and denote by $\eta_0$ the product of the Haar measure on $G$ and the counting measure on $\N$. Let $G \actson (V,\eta)$ be the Poisson suspension of $G \actson (V_0,\eta_0)$. Then, $G \actson (V,\eta)$ is a mixing pmp action that can be written as the infinite product of a faithful pmp $G$-action. Hence, $G \actson (V,\eta)$ is essentially free. Taking the disjoint union of $X_0$ and $V_0$, we obtain the nonsingular action $G \actson (X_0 \sqcup V_0,\mu_0 \sqcup \eta_0)$ that, by Lemma \ref{lem.technical-suspension-lemma} satisfies all the conclusions of the proposition.
\end{proof}

\section{Bernoulli actions of amenable groups: proof of Theorem \ref{thm.main-example}}\label{sec:amenable}

In Theorem \ref{thm:amenable Bernoulli flow} below, as in the previous section, we prove a stable version of Theorem \ref{thm.main-example}.


\begin{theorem}\label{thm:amenable Bernoulli flow}
Let $\R\actson (Z,\zeta)$ be a Poisson flow and let $G$ be a countable infinite amenable group. Then there exists a countable set $X_0$ and a family of equivalent probability measures $(\mu_g)_{g\in G}$ on $X_0$ such that the Bernoulli action
\begin{align*}
G\actson \prod_{g\in G}(X_0,\mu_g)
\end{align*}
is nonsingular, weakly mixing and such that for any ergodic pmp action $G\actson (Y,\eta)$ the associated flow of the diagonal action $G\actson X\times Y$ is isomorphic with $\R\actson Z$.
\end{theorem}

For the periodic flows $\R \actson \R / \Z \log \lambda$, i.e.\ the type III$_\lambda$ case, Theorem \ref{thm:amenable Bernoulli flow} was proven independently in \cite{BV20} and \cite{KS20}. For the trivial flow, which corresponds to the type III$_1$ case, the result was already proven in \cite[Theorem 6.1]{BKV19}. The novelty lies in the type III$_0$ case, i.e.\ the properly ergodic flows.

As in the introduction, for any countable subgroup $\Lambda \subset \R$, we denote by $M_\Lambda$ the unique injective factor whose flow of weights is the almost periodic flow $\R \actson \Lambdah$. Combining Theorems \ref{thm.Poisson-vs-ITPFI-2} and \ref{thm.almost-periodic-Poisson} with Theorem \ref{thm:amenable Bernoulli flow}, we obtain the following immediate corollary.

\begin{corollary}\label{cor: amenable crossed product}
Let $G$ be a countable infinite amenable group. Any ITPFI$_2$ factor of type III, and in particular all the injective factors $M_\Lambda$ where $\Lambda \subset \R$ is a countable subgroup, are isomorphic to the crossed product $L^{\infty}(X)\rtimes G$ of a nonsingular Bernoulli action $G\curvearrowright (X,\mu)$ of $G$.
\end{corollary}

If $\gamma_0$ is a measure on $V_0 = G \times \N$ that is equivalent with the counting measure and if the nonsingular action $G \actson (V_0,\gamma_0) : g \cdot (h,n) = (gh,n)$ has a well defined Poisson suspension $G \actson (V,\eta)$, then $G \actson (V,\eta)$ is canonically isomorphic with the nonsingular Bernoulli action
\begin{equation}\label{eq.my-poisson-identification}
G \actson \prod_{g \in G} (X_0,\mu_g) \quad\text{with}\quad X_0 = (\N \cup \{0\})^\N \quad\text{and}\quad \mu_g = \prod_{n \in \N} \mu_{n,g} \, ,
\end{equation}
where $\mu_{n,g}$ is the Poisson distribution with intensity $\gamma_0(g,n)$.

Therefore, Theorem \ref{thm:amenable Bernoulli flow} can be deduced as follows from Lemma \ref{lem.technical-suspension-lemma}.

\begin{proof}[Proof of Theorem \ref{thm:amenable Bernoulli flow}]
By Proposition \ref{prop.standard-poisson} and using the notation in \eqref{eq:Poisson dist}, we write $\R \actson Z$ as the tail boundary flow of a sequence $\sigma_{\lambda_n,a_n}$ with $\lambda_n > 0$ and $a_n \in \R \setminus \{0\}$.

Denote by $\rho$ the counting measure on $G$. Since $G$ is amenable, the translation action $G \actson (G,\rho)$ admits a F{\o}lner sequence $A_n \subset G$, which we may choose in such a way that $|A_n|$ is an increasing unbounded sequence. Replacing $A_n$ by a subsequence, we may further assume that
\begin{equation}\label{eq.atomic-condition}
\lambda_n \, |A_n|^{-1} \, (1+e^{a_n}) \leq 2^{-n}
\end{equation}
for all $n \in \N$. Define $V_0 = G \times \N$ and define the measure $\gamma_0$ on $V_0$ by
$$\gamma_0(g,n) = \begin{cases} \lambda_n \, |A_n|^{-1} &\;\;\text{if $g \in A_n$,}\\ \lambda_n \, |A_n|^{-1} \, e^{a_n} &\;\;\text{if $g \not\in A_n$.}\end{cases}$$
After replacing once more $A_n$ by a subsequence, it follows from Lemma \ref{lem.technical-suspension-lemma} that the Poisson suspension $G \actson (X,\mu)$ of $G \actson (V_0,\gamma_0)$ is well defined, weakly mixing and has the property that for every ergodic pmp action $G \actson (Y,\eta)$, the associated flow of the diagonal action $G \actson X \times Y$ is isomorphic with $\R \actson Z$. Note that \eqref{eq.atomic-condition} remains valid in this passage to a subsequence.

By \eqref{eq.my-poisson-identification}, we may view $G \actson (X,\mu)$ as a nonsingular Bernoulli action. Since by \eqref{eq.atomic-condition}
$$\mu_{n,e}(0) = \exp(-\gamma_0(e,n)) \geq 1-\gamma_0(e,n) \geq 1-2^{-n} \; ,$$
the base space $(X_0,\mu_e)$ of the nonsingular Bernoulli action $G \actson (X,\mu)$ is atomic, so that Theorem \ref{thm:amenable Bernoulli flow} is proven.
\end{proof}

\begin{remark}
Note that although our proof of Theorem \ref{thm:amenable Bernoulli flow} makes use of Lemma \ref{lem.technical-suspension-lemma}, the proof actually does not rely in an essential way on the theory of Poisson suspensions. It follows from \eqref{eq.my-poisson-identification} that the nonsingular Bernoulli action $G \actson (X,\mu)$ that we construct in the proof of Theorem \ref{thm:amenable Bernoulli flow} is an infinite product action $G \actson \prod_{n \in \N} (X_n,\mu_n)$, where for each $n$, there is a natural $G$-invariant probability measure $\nu_n \sim \mu_n$ such that $G \actson (X_n,\nu_n)$ is a pmp Bernoulli action. The argument in the proof of Lemma \ref{lem.technical-suspension-lemma} then provides the identification of the associated flow as the correct tail boundary flow.
\end{remark}



\begin{thebibliography}{CCMT12}\setlength{\itemsep}{-1mm} \setlength{\parsep}{0mm} \small
\bibitem[AIM19]{AIM19} Y. Arano, Y. Isono and A. Marrakchi, Ergodic theory of affine isometric actions on Hilbert spaces. {\it Geom. Funct. Anal.} {\bf 31} (2021), 1013-1094.

\bibitem[AP77]{AP77} D. Aldous and J. Pitman, On the zero-one law for exchangeable events. {\it Ann. Probab.} {\bf 7} (1979), 704-723.

\bibitem[AW68]{AW68} H. Araki and E.J. Woods, A classification of factors. {\it Publ. Res. Inst. Math. Sci. Ser. A} {\bf 4} (1968/1969), 51-130.

\bibitem[BH83]{BH83} A.D. Barbour and P. Hall, On the rate of Poisson convergence. {\it Math. Proc. Cambridge Philos. Soc.} {\bf 95} (1984), 473-480.

\bibitem[BHV08]{BHV08} B. Bekka, P. De la Harpe and A. Valette, Kazhdan's Property $(T)$. {\it Cambridge University Press}, Cambridge, 2008.

\bibitem[BK18]{BK18} M. Bj\"{o}rklund and Z. Kosloff, Bernoulli actions of amenable groups with weakly mixing Maharam extensions. \textit{Preprint.} \href{http://arxiv.org/abs/1808.05991}{ arXiv:1808.05991}

\bibitem[BKV19]{BKV19} M. Bj\"{o}rklund, Z. Kosloff and S. Vaes, Ergodicity and type of nonsingular Bernoulli actions. {\it Invent. Math.}  {\bf 224} (2021), no. 2, 573-625.

\bibitem[BV20]{BV20} T. Berendschot and S. Vaes, Nonsingular Bernoulli actions of arbitrary Krieger type. To appear in {\it Anal. PDE.} \href{https://arxiv.org/abs/2005.06309}{arXiv:2005.06309}

\bibitem[CD60]{CD60} G. Choquet and J. Deny, Sur l'\'{e}quation de convolution $\mu = \mu * \sigma$. {\it C. R. Acad. Sci. Paris} {\bf 250} (1960), 799-801.

\bibitem[CFS82]{CFS82} I.P. Cornfeld, S.V. Fomin and Ya.G. Sinai, Ergodic Theory. Springer Verlag, New York/Heidelberg/Berlin, 1982.

\bibitem[CFW81]{CFW81} A. Connes, J. Feldman and B. Weiss, An amenable equivalence relation is generated by a single transformation. \textit{Ergodic Theory Dynam. Systems} \textbf{1} (1981), 431-450.

\bibitem[CI09]{CI09} I. Chifan and A. Ioana, Ergodic subequivalence relations induced by a Bernoulli action. \textit{Geom. Funct. Anal.} \textbf{20} (2010), 53-67.

\bibitem[Con76]{Con76} A. Connes, Classification of injective factors. \textit{Ann. of Math. (2)} \textbf{104} (1976), 73-115.

\bibitem[CT77]{CT77} A. Connes and M. Takesaki, The flow of weights on factors of type III. \textit{Tohoku Math. J.} \textbf{29} (1977), 473-575.

\bibitem[CW80]{CW80} A. Connes and B. Weiss, Property T and asymptotically invariant sequences. \textit{Israel J. Math.} {\bf 37} (1980), 209-210.

\bibitem[CW83]{CW83} A. Connes and E.J. Woods, Approximately transitive flows and ITPFI factors. \textit{Ergodic Theory Dynam. Systems} \textbf{5} (1985), 203-236.

\bibitem[CW88]{CW88} A. Connes and E.J. Woods, Hyperfinite von Neumann algebras and Poisson boundaries of time dependent random walks. {\it Pacific J. Math.} {\bf 137} (1989), 225-243.

\bibitem[Dan18]{Dan18} A.I. Danilenko, Weak mixing for nonsingular Bernoulli actions of countable amenable groups. \textit{Proc. Amer. Math. Soc.} \textbf{147} (2019), 4439-4450.

\bibitem[Dan21]{Dan21} A.I. Danilenko, Krieger's type for ergodic nonsingular Poisson actions of non-(T) locally compact groups. {\it Preprint.} \href{https://arxiv.org/abs/2105.06164}{arXiv:2105.06164}

\bibitem[DK20]{DK20} A.I. Danilenko and Z. Kosloff, Krieger's type of nonsingular Poisson suspensions and IDPFT systems. {\it Proc. Amer. Math. Soc.} {\bf 150} (2022), 1541-1557.

\bibitem[DKR20]{DKR20} A.I. Danilenko, Z. Kosloff and E. Roy, Nonsingular Poisson suspensions. To appear in {\it J. Anal. Math.} \href{http://arxiv.org/abs/2002.02207}{arXiv:2022.02207}

\bibitem[DL16]{DL16} A. Danilenko and M. Lema\'nczyk, $K$-property for Maharam extensions of nonsingular Bernoulli and Markov shifts. {\it Ergodic Theory Dynam. Systems} {\bf 39} (2019), 3292-3321.

\bibitem[DL20]{DL20} A.I. Danilenko and M. Lema\'nczyk, Ergodic cocycles of IDPFT systems and nonsingular Gaussian actions. To appear in {\it Ergodic Theory Dynam. Systems}. \href{https://arxiv.org/abs/2006.08567}{arXiv:2006.08567}

\bibitem[Dye59]{Dye59} H. Dye, On groups of measure preserving transformations. I. \textit{Amer. J. Math.} \textbf{81} (1959), 119-159.

\bibitem[GH08]{GH08} T. Giordano and D. Handelman, Matrix-valued random walks and variations on property AT. {\it M\"{u}nster J. Math.} 1 (2008), 15-72.

\bibitem[GS83]{GS83} T. Giordano and G. Skandalis, On infinite tensor products of factors of type I$_{2}$. \textit{Ergodic Theory Dynam. Systems} \textbf{5} (1985), 565-586.

\bibitem[GS84]{GS84} T. Giordano and G. Skandalis, Krieger factors isomorphic to their tensor square and pure point spectrum flows. \textit{J. Funct. Anal.} {\bf 64} (1985), 209-226.

\bibitem[GSW84]{GSW84} T. Giordano, S. Skandalis and E. Woods, On the computation of invariants of ITPFI factors. \textit{J. Operator Theory } \textbf{15} (1986), 83-107.

\bibitem[Haa85]{Haa85} U. Haagerup, Connes' bicentralizer problem and uniqueness of the injective factor of type III$_{1}$. \textit{Acta Math.} \textbf{158} (1987) 95-148.

\bibitem[Ham81]{Ham81} T. Hamachi, On a Bernoulli shift with non-identical factor measures. {\it Ergodic Theory Dynam. Systems} \textbf{1} (1981), 273-284.

\bibitem[HOO74]{HOO74} T. Hamachi, Y. Oka and M. Osikawa, Flows associated with ergodic non-singular transformation groups. \textit{Publ. Res. Inst. Math. Sci., Kyoto Univ.} {\bf 11} (1975), 31-50.

\bibitem[HO83]{HO83} T. Hamachi and M. Osikawa, Computation of the associated flows of ITPFI$_2$ factors of type III$_0$. In {\it Geometric methods in operator algebras (Kyoto, 1983),} Pitman Res. Notes Math. Ser. {\bf 123}, Longman Sci. Tech., Harlow, 1986, pp.\ 196-210.

\bibitem[Ioa10]{Ioa10} A. Ioana, $W^{*}$-superrigidity for Bernoulli actions of property (T) groups. \textit{J. Amer. Math. Soc.} \textbf{24} (2011), 1175-1226.

\bibitem[Kak48]{Kak48} S. Kakutani, On equivalence of infinite product measures. \textit{Ann. of Math.} \textbf{49} (1948), 214-224.

\bibitem[Kos09]{Kos09} Z. Kosloff, On a type III$_1$ Bernoulli shift. {\it Ergodic Theory Dynam. Systems} \textbf{31} (2011), 1727-1743.

\bibitem[Kos12]{Kos12} Z. Kosloff, On the K property for Maharam extensions of Bernoulli shifts and a question of Krengel. {\it Israel J. Math.} \textbf{199} (2014), 485-506.

\bibitem[Kos18]{Kos18} Z. Kosloff, Proving ergodicity via divergence of ergodic sums. \textit{Studia Math.} \textbf{248} (2019), 191-215.

\bibitem[Kri76]{Kri76} W. Krieger, On ergodic flows and the isomorphism of factors. \textit{Math. Ann.} \textbf{223} (1976), 17-70.

\bibitem[KS20]{KS20} Z. Kosloff and T. Soo, The orbital equivalence of Bernoulli actions and their Sinai factors. {\it J. Mod. Dyn.} {\bf 17} (2021), 145-182.

\bibitem[Ore66]{Ore66} S. Orey, Tail events for sums of independent random variables. {\it J. Math. Mech.} {\bf 15} (1966), 937-951.

\bibitem[OW80]{OW80} D. Ornstein and B. Weiss, Ergodic theory of amenable group actions. I. The Rohlin lemma. \textit{Bull. Amer. Math. Soc.} \textbf{2} (1980), 161-164.

\bibitem[Pop03]{Pop03} S. Popa, Strong rigidity of II$_{1}$ factors arising from malleable actions of $w$-rigid groups, I. \textit{Invent. Math.} \textbf{165} (2006), 369-408.

\bibitem[Pop06]{Pop06} S. Popa, On the superrigidity of malleable actions with spectral gap. \textit{J. Amer. Math. Soc.} \textbf{21} (2008), 981-1000.

\bibitem[PV21]{PV21} S. Popa and S. Vaes, W$^*$-rigidity paradigms for embeddings of II$_1$ factors. \textit{Preprint.} \href{https://arxiv.org/abs/2102.01664}{arXiv:2102.01664}

\bibitem[Pro53]{Pro53} Y. Prokhorov, Asymptotic behavior of the binomial distribution. {\it Uspehi Matem. Nauk (N.S.)} \textbf{8} (1953), no.3(55), 135-142.

\bibitem[Roy08]{Roy08} E. Roy, Poisson suspensions and infinite ergodic theory. {\it Ergodic Theory Dynam. Systems} \textbf{29} (2009), 667-683.

\bibitem[Shi19]{Shi19} A.N. Shiryaev, Probability -- 2. {\it Graduate Texts in Mathematics} {\bf 95}, Springer, New York, 2019.

\bibitem[Stro11]{Stro11} D.W. Stroock, Probability theory. An analytic view. Second edition. Cambridge University Press, Cambridge, 2011.

\bibitem[VW17]{VW17} S. Vaes and J. Wahl, Bernoulli actions of type III$_{1}$ and $L^2$-cohomology. \textit{Geom. Funct. Anal} \textbf{28} (2018), no.2, 518-562.
\end{thebibliography}
\end{document}